 \documentclass{article}

\usepackage{amsmath,amsfonts,amsthm,amssymb,amscd,cancel,color}
\usepackage{enumitem}
\usepackage{verbatim}
\usepackage{graphicx}
\usepackage{ulem,xcolor}

\renewcommand{\Re}{\mathop{\rm Re}\nolimits}
\renewcommand{\Im}{\mathop{\rm Im}\nolimits}
\def\S{\mathhexbox278}

\theoremstyle{plain}
\newtheorem{theorem}{Theorem}[section]
\newtheorem{lemma}[theorem]{Lemma}
\newtheorem{proposition}[theorem]{Proposition}

\theoremstyle{definition}

\theoremstyle{remark}
\newtheorem{remark}[theorem]{Remark}

\newtheorem{claim}[theorem]{Claim}
\newtheorem{notation}[theorem]{Notation}

\newcommand{\R}{{\mathbb R}}

\newcommand{\N}{{\mathbb N}}

\def\im{{\rm i}}

\newcommand{\C}{\mathbb{C}}

\def\({\left(}
\def\){\right)}
\def\<{\left\langle}
\def\>{\right\rangle}

\newcommand{\sech}{{\mathrm{sech}}}

\newcommand{\supp}{{\mathrm{supp}\ }}
    \newcommand{\diag}{{\mathrm{diag}}}
\newcommand{\Span}{{\mathrm{Span}}}

\numberwithin{equation}{section}

\begin{document}

\title{Conditional asymptotic stability of solitary waves of the  Euler-Poisson system  on the line}

\author{Junsik Bae, Scipio Cuccagna and Masaya Maeda}
\maketitle

\begin{abstract}   We apply  the idea of using a combination of virial inequalities and Kato smoothing, previously applied to    NLS and generalized KdV  pure power equations to  Euler-Poisson: we assume that a solution remains very close for all times to a soliton in an appropriate space    and then we prove an  asymptotic convergence to a soliton  for $t\to +\infty$.
\end{abstract}

\section{Introduction} 

In this paper, we consider the following Euler-Poisson system:
\begin{align}\label{EP}
    \left\{
        \begin{aligned}
            &\partial_t {n} +\partial_x\left( (1+n)  u \right) = 0 ,  \\
            &\partial_t{u} + \partial_x\( \frac{u^2}{2} + K \log (1+n)  \) = -\partial_x \phi,   \quad (t,x)\in \mathbb{R}_+\times \mathbb{R},   \\& -  \partial_x^2 \phi = 1+n - e^{\phi},
        \end{aligned}
    \right.
\end{align}
where 
  $1+n(t,x),\,u(t,x)$, and $\phi(t,x)$ are real-valued unknown functions that represent the density, velocity of the ions, and the electrostatic potential, respectively. In the model \eqref{EP}, the isothermal pressure $p=K(1+n)$ is considered, where the constant $K>0$ is the ratio of the ion temperature to the electron temperature. The electron density $\rho_e$ is assumed to follow the  Boltzmann relation $\rho_e = e^\phi$. 

The Euler-Poisson system \eqref{EP} is a fundamental fluid model that describes the dynamics of ions in an electrostatic plasma \cite{Ch84}, in particular, in connection with the emergence of traveling solitary waves.  The system \eqref{EP} admits smooth traveling solitary waves in the super-sonic regime \cite{CDMS96,BK19JDE}. In the long wave length scaling, the solitary waves are approximated by the KdV solitary waves \cite{BK19JDE}. While  Bae and Kwon \cite{BK22ARMA} prove a result on linear stability of \eqref{EP} around the solitary waves in a similar fashion to  \cite{PegoSunARMA2016,PegoWei2}, no result on the orbital or asymptotic stability of the nonlinear problem \eqref{EP} seems to be known.
One of the main difficulties is that smooth solutions to \eqref{EP} may develop singularities in  finite time \cite{BCK24}, and the global existence of smooth solutions has not yet been established due to the weak dispersive effect in low dimensions.  Another difficulty comes from the fact that, even formally, the conserved functional $E -c M$ does not have a definite sign at the far field, and hence it appears that one cannot show the orbital stability as in \cite{MR954673}. For the 3D model, see \cite{GP11} for the global existence of smooth solutions, and \cite{RS25} for the transverse nonlinear stability of 1D solitary waves under 3D perturbations.

The goal of this paper is to investigate the conditional asymptotic stability. That is, we study the asymptotic behavior for $t\to +\infty$ of the solutions of \eqref{EP} which  a priori  we  assume very close to the solitary wave for all $t\ge 0$. In the present paper, we only consider the isothermal model. Throughout our analysis, $K>0$ plays a crucial role,  in particular, for obtaining the virial inequality. Essentially it is due to the fact that the pressure energy in the Hamiltonian is approximated by $\tfrac{Kn^2}{2}$, see also later Remark \ref{rem:Kpositive} for more.

We remark that the pressureless Euler-Poisson system (\eqref{EP} with $K=0$) is often considered in the literature on plasma physics under the cold ion assumption. This simplified model also admits traveling solitary waves, and the linear stability has been studied in \cite{HS2002}. On the other hand, the absence of the pressure term makes the system weakly coupled, and, in general, the behavior of solutions is qualitatively different from that in the presence of the pressure term. In the context of structures of singularities in solutions, we refer to \cite{BCK24,BKK24,BKK25,BMW24}. In this regard, the study of the conditional asymptotic stability would be an interesting problem.

\subsection{Main result}


In \cite{BK19JDE}, it is shown that there is a constant  $c_K> \mathsf{V}:=\sqrt{1+K} $ such that for all $c \in (\mathsf{V}, c_K)$ there exist solitary wave solutions $S_c(x-ct):=(n_c,u_c,\phi_c)(x-ct)$ of  \eqref{EP}.
    Furthermore, for $\varepsilon   := c- \mathsf{V}$,  we have
   \begin{equation}\label{soliton1}
        S_c(x)   = \varepsilon  \( 1, \mathsf{V}, 1     \)^\intercal  \psi_{\mathrm{KdV}}\( \sqrt{\varepsilon} x \) + O(\varepsilon^2) \quad \text{as } \varepsilon \to 0^+,
    \end{equation}
    where 
    \[
    \psi_{\mathrm{KdV}}\(  x \)= \frac{3}{\mathsf{V}} \sech ^2 \( \frac{\sqrt{\mathsf{V}}}{\sqrt{2}} x \)
    \]
    is the solitary wave for the Korteweg--De Vries equation (KdV) such that
\begin{align*}
  -\frac{1}{2\mathsf{V}}\partial_x^2\psi_{\mathrm{KdV}}+\psi_{\mathrm{KdV}} - \mathsf{V}   \frac{\psi_{\mathrm{KdV}}^2}{2}  =0.
\end{align*}

The local existence of classical solutions to  \eqref{EP} for the initial data $(n_0,u_0)\in H^s(\mathbb{R})\times H^s(\mathbb{R})$, where $s>3/2$, is studied in \cite{LLS13}. 
For $  U=(n,u,\phi)\in C^0 \left(   [0,+\infty ), H^1(\R )\right)$ we say $U$ is a   solution of \eqref{EP} if  thetime derivative equation in  \eqref{EP} holds in a distributional sense in $\mathcal{D}'(\R _+, L^2(\R) )$.
From this definition, 
 we have $  U\in C^1\(   [0,+\infty ), L^2(\R )\)$, see   \cite[Ch. 1]{CazHarbook}.
The following   is the main result of this paper.

\begin{theorem}
  \label{thm:main} There exists a $\varepsilon _0\in (0, c_K-\mathsf{V}) $  such that for any $a>0$, $c_0\in (\mathsf{V}, \mathsf{V} + \varepsilon _0)$ and     $\epsilon >0$   there exists a  $\delta  >0$   such that for any
\begin{align}
  \label{eq:mainreg} U=(n,u,\phi)\in C^0\(   [0,+\infty ), H^1(\R, \R^3 )\)
\end{align}
  solution of \eqref{EP} which satisfies
  \begin{align}
    \label{eq:main1} \sup _{t\ge 0} \inf_{x_0\in \R} \|  (n,u, \phi)(t,\cdot) -(n_{c_0},u_{c_0}, \phi _{c_0})(\cdot -x_0) \| _{H^1(\R )}  \le \delta,
  \end{align}
 there exist a function  $ (c, D) \in C^1 \( [0,+\infty ) ,  (c_0-\epsilon , c_0+\epsilon) \times \R \)$ and   a $c_+>0$
     such that
  \begin{align} \label{eq:asstab1}
    & U(t)=    S _{c (t)} (\cdot -D(t) ) + V(t, \cdot -D(t) ) \text{   with}
 \\&  \label{eq:asstab2}   \int _{\R _+  }  dt \int _{   \R }  e^{- 2a\< x\>}  | V (t , x) | ^2  dx <  \epsilon  \text{  where }\< x\>:=\sqrt{1+x^2} , \\&  \label{eq:asstab20}
   \lim _{t\to +\infty  } \int _{   \R }  e^{- 2a\< x\>}  | V (t , x) | ^2  dx     =0\text{ and} \\&
\lim _{t\to +\infty}c (t)= c _+ . \label{eq:asstab3}
\end{align}

\end{theorem}

\begin{remark} \label{rem:exisphi} For any  given  $n\in L^2(\R )$ there is exists and is     unique a $\phi \in H ^1(\R )$ which satisfies  the third equation in \eqref{EP}.
This can be shown  by considering the functional
\begin{align*}
  F(\phi ) = 2^{-1}\| \phi ' \| _{L^2(\R )} ^2+ \int _{\R}\( e^{\phi } -\phi -1  -n \phi     \)  dx ,
\end{align*}
where  $F\in C^1 (H ^1(\R ), \R  )$  is convex and
\begin{align}\label{eq:Fcoerciv}
    \lim _{\| \phi \| _{H ^1(\R )}\to +\infty } F(\phi ) =+\infty .
\end{align}
Here  \eqref{eq:Fcoerciv} can be shown by a contradiction argument as follows.
First, notice that we have $\Phi(t):=e^{t}-1-t\gtrsim \min(|t|,t^2)$.
Suppose there exists $\{\phi_n\}\subset H^1$ s.t. $\|\phi_n\|_{H^1}\to \infty$ and $\sup_n F(\phi_n)=M<\infty$.
Then, by the Schwarz inequality,
$
2^{-1}\|\phi_n'\|_{L^2}^2+\int \Phi(\phi_n)\lesssim M+\|n\|_{L^2} \|\psi_n\|_{L^2}.
$
Thus, we have
\begin{align*}
\|\phi_n'\|_{L^2}^2+\int \Phi(\phi_n)\lesssim  \|\phi_n\|_{L^2}.
\end{align*}
Therefore  we have a contradiction   from
\begin{align*}
\|\phi_n\|_{L^2}^2&=\int_{|\phi_n|\leq 1}\phi_n^2+\int_{|\phi_n|>1}\phi_n^2\lesssim \int \Phi(\phi_n) + \|\phi_n\|_{L^\infty} \int_{|\phi_n|>1}|\phi_n|\\&\lesssim
\|\phi_n\|_{L^2}+\|\phi_n'\|_{L^2}^{1/2}\|\phi_n\|_{L^2}^{1/2}\int \Phi(\phi_n)
\lesssim \|\phi_n\|_{L^2} + \|\phi_n\|_{L^2}^{\frac{1}{4}+\frac{1}{2}+1}.
\end{align*}
From the convexity and \eqref{eq:Fcoerciv}, $F$ has a point of absolute minimum, see \cite[p.71]{brezisbook}, which is a critical point of $F$, i.e. the third equation in \eqref{EP} is satisfied:
\begin{align*}
 0=   \nabla F(\phi  )  = -  \partial_x^2 \phi +e^{\phi} -1-n.
\end{align*}
  If $\phi _1$ and $\phi _2$  are two distinct  critical points,  we obtain the following contradiction,
\begin{align*}
 0= \< \nabla F(\phi _2)- \nabla F(\phi _1) , \phi _2 - \phi _1 \> =  \| \phi '_2   -  \phi '_1\| _{L^2(\R )} ^2 +  \<      e^{\phi _2}- e^{\phi _1}, \phi _2 - \phi _1 \> >  \| \phi '_2   -  \phi '_1\| _{L^2(\R )} ^2    > 0.
\end{align*}

\end{remark}

We prove  a slightly stronger result than \eqref{eq:asstab2}, see \eqref{eq:asstab2strong}, with  the coordinates of \eqref{eq:ansatz}. 
The proof of Theorem \ref{thm:main}  is inspired by \cite{CM25D1}. The key   to prove asymptotic stability  is  showing the spacial  escape of the error term from the ground state, which is due to the dispersive nature of the equation and in particular of its linear part. There is a large literature, with \cite{PegoWei2,SW1,SW2} some of the classical contributions,  based on  decay inequalities like \eqref{eq:pwdisp1}. In the context of KdV type equations this has led to  results of asymptitic stability in weighted space, see \cite{PegoWei2, MizuKdV2001}.  A different approach based on virial inequalities
   initiated by Martel and Merle
\cite{MaMeARMA2001,MaMeNonlinearity2005,MaMegeneral2008,MaMeRefined2008}, which appears closely related to the Mourre Inequalities,   in turn   inspired by simple finite dimensional arguments of scattering trajectories,  see \cite[Ch. 2.3]{MR1459161}, has proven to be particularly robust,  leading to important breakthroughs in the theory of asymptotic stability of solitons, such as \cite{KMM2017,KM22,Martelcubquint1,Martelcubquint2} and has been used also for a a variety of other models such as good Boussinesq equation, \cite{MR4629757, maulen2025asymptoticstabilitystablegood},  one dimensional Gross--Pitaevskii equation,  \cite{kowalczyk2025newproofliouvilletheorem}, and in higher dimension, \cite{MR3461359}.  These methods were  modified in a series of papers   \cite{CM24D1,CM243,CM24D4}  devoted to the asymptotic stability of the ground states of the  pure NLS in dimension 1, by proving the virial inequalities in a two steps. The first step consisted in the proof of virial inequalities at high energies, which is  rather simple in many examples.  For \eqref{EP}, the fact that the $\partial _x$ derivative   originates from the symplectic form,  that is the energy in \eqref{eq:energy} does not involve derivatives, has the effect that the norms in \eqref{eq:normA} do not control norms of $\partial _x\widetilde{V}$.  In this sense, here it is somewhat improper to call these  estimates \textit{high energy}.
As in the work by Martel and Merle it is very important to control the term $\|  {V }  \|_{  \widetilde{\Sigma}    }^2$ in \eqref{eq:lem:1stV11}. For the generalized KdV equations in \cite{CM25D1} we treated the corresponding term using Kato smoothing. However, the derivative in the right hand side \eqref{eq:prekato1} and the apparent lack of smoothing of the semigroup  $e^{t\mathcal{L}  _{c } }$  require     to introduce    an additional virial inequality,   Lemma \ref{lem:2stV1}, not needed in  \cite{CM25D1}. This additional virial inequality allows to replace  
$\|  {V }  \|_{  \widetilde{\Sigma}    }^2$ with  $\|  {V } _\phi \|_{  \widetilde{\Sigma}    }^2$. Since ${V } _\phi$ solves an elliptic equation  and is more regular than $\widetilde{V}$, we neutralize the effect of  the derivative in the right hand side \eqref{eq:prekato1}.

The second step  of our  proof    
consists in replacing the second virial inequalities  of     Kowalczyk et al.  or the analysis of the positivity of a quadratic form in Martel and Merle, 
which can be a   complicated, with a smoothing estimate  inspired by Mizumachi \cite{MR2511047}.  In fact the   $\|   V  _\phi    \|_{   L^2( I ,\widetilde{\Sigma})    } $ in \eqref {eq:sec:1virial11} is a typical Kato smoothing term, see \cite{katoMA1966}.
Apart from the issue of loss of derivative, which in  \cite{CM25D1} was automatically solved by smoothing and which, as we discussed above, is here solved by the smoothing effect of the elliptic equation in \eqref{EP2}, the main difficulty here as in   \cite{CM25D1} is with the commutator term between the weight $\zeta  _B $ and the linearized operator in \eqref{eq:prekato1}.     By  inverse Laplace transform,  we are reduced  to an analysis of the resolvent of the linearization.  This is nontrivial because the spectrum,   which coincides  with the imaginary axis $\im \R$, contains in its interior the eigenvalue $\lambda =0$. Thanks to \cite{BK22ARMA} we know that 0 is the only eigenvalue embedded in $\im \R$. In turn, the resolvent is expressed in terms of Jost functions, which play the role of the  plane waves    for  differential operators  with constant coefficients.    The   use of Jost functions in problems on the line is  standard, see for example Weder \cite{MR1980087,MR1729096}. The theory is especially well developed  in the case of Schr\"odinger operators, \cite{DT1979}, and for  more general linear systems of ODE's in the line, especially in the context of the direct scattering  \cite{BealsCoif1987, Bealsbook1988}. Here the heart of the paper  are the  detailed    estimates,   in Sect. \ref{sec:jost},   for the Jost functions near the point   $\lambda=0$. 
Armed with  these estimates, after   expressing the resolvent of the linearization in terms of the Jost functions, we are able to obtain estimates for the resolvent.
In particular, after a Taylor expansion of the Jost functions at 0, we are able to project away the singularity at $\lambda=0$ and   to bound the operator involving the error terms of    Jost functions   in the Taylor expansions using the estimates of  Sect. \ref{sec:jost}. The arguments from  Sect. \ref{sec:jost} largely   parallel those of \cite{CM25D1}.  Plausibly  our method
could be very robust and    amenable to significant generalizations.

\section{Notation }\label{sec:notation}

\begin{notation}\label{not:notation}
Throughout this paper, we will use the following list of   notations and definitions.

\begin{enumerate}

\item We use the notation $\dot u=  \partial _t  u$ and $  u'=  \partial _x  u$.

\item We set $\C _+=\{ \lambda \in \C: \Re \lambda >0\} $  and  $\overline{\C _+}  $ the closure of  $\C _+  $ in $\C$.


\item In analogy to     Kowalczyk et al.   \cite{KMM2022},      we   consider constants  $A, B,  A_1, \epsilon , \delta  >0$ satisfying
 \begin{align}\label{eq:relABg}
\log(\delta ^{-1})\gg\log(\epsilon ^{-1}) \gg     A    \gg   B^2\gg B    \gg 1  \text{ and $ A_1:= B ^{3/5}$.}
 \end{align} All these constants depend on the $\varepsilon >0$   in \eqref{soliton1}.
 
 \item The notation    $o_{y}(1)$  means a function in  $y$ such that
 $o_{y}(1) \xrightarrow {y  \to 0^+   }0.$

\item
Following Pego and Weinstein  \cite{PegoWei2} for $a\in \R $ we set
\begin{align}\label{eq:pwweigh0}&
   L^2_{a}=\{ v: \ e^{ax} v\in L^2(\R )   \}  \text{ with }\|  v  \| _{L^2_{a}}=\| e^{ax} v  \| _{L^2 }.
\end{align}  We also define    \begin{align}\label{eq:Lebwight}&
   L ^{p,s}=\{ v: \ \< x \> ^{s}   v\in L^p(\R )   \}  \text{ with }\|  v  \| _{ L ^{p,s}}:=\| \< x \> ^{s} v  \| _{L^p(\R ) }.
\end{align}

\item Given two Banach spaces $X$ and $Y$ we denote by $\mathcal{L}(X,Y)$ the space of continuous linear operators from $X$ to $Y$. We write $\mathcal{L}(X ):=\mathcal{L}(X,X)$.

\item Following  Kowalczyk et al. \cite{KMM2022} we   fix an even function $\chi\in C_c^\infty(\R , [0,1])$ satisfying
\begin{align}  \label{eq:chi} \text{$1_{[-1,1]}\leq \chi \leq 1_{[-2,2]}$ and $x\chi'(x)\leq 0$ and set $\chi_C:=\chi(\cdot/C)$  for a  $C>0$}.
\end{align}

\item We consider a decreasing function $\vartheta _1\in C^\infty ( \R , [0,1])$ with $\vartheta _1(x) =\left\{
                                                                                                     \begin{array}{ll}
                                                                                                       1, & \hbox{ for $x\le 1$ } \\
                                                                                                       0, & \hbox{for $x\ge 2$  }
                                                                                                     \end{array}                                                                                         \right. $,   we set $ \vartheta _2=1-\vartheta _1 $ and consider the partition of unity $1= \vartheta _{1A_1} + \vartheta _{2A_1} $  where $ \vartheta _{iA_1}(x):=  \vartheta _{i }\(\frac{x}{A_1}\)$.

\item Following   Kowalczyk et al. \cite{KMM2022}  and
using the  $\chi$  in \eqref{eq:chi}
we  consider for   $C>0$  the   function
\begin{align}\label{def:zetaC}
\zeta_C(x):=\exp\(-\frac{|x|}{C}(1-\chi(x))\)   \text {  and   }  \varphi_C (x):=\int_0^x \zeta_C^2(y)\,dy.
\end{align}
We similarly define
\begin{align}\label{def:phiAi}   \varphi _{iAA_1}  (x):=\int_0^x \zeta_A^2(y) \vartheta _{iA_1} ^2 (y)\,dy.
\end{align}

\item For $x\in \R$ we set $x^\pm  :=\max (0, \pm x) $.

\item  For vector valued functions $ V=( V_n, V_u, V_\phi ) ^\intercal $  and  $\widetilde{V}=( V_n, V_u  ) ^\intercal $
 we  will consider the norms
\begin{align}
&
\| V  \|_{     { \Sigma}    _{iA A_1}} :=\left \|   \vartheta _{iA_1} \zeta_A V   \right \|_{L^2(\R )} + \left \|   \( \vartheta _{iA_1} \zeta_A V  _\phi\) '       \right\|_{L^2(\R )}    \text{ for $i=1,2$   } \label{eq:normA}\end{align}
and for functions $f$ and  for a fixed sufficiently small  $\kappa  >0$
\begin{align}
& \|  f \|_{ \widetilde{\Sigma} } :=\left \| \sech \(     \varepsilon \kappa  x\)  f\right \|_{L^2(\R )}      \quad .  \label{eq:normk}
\end{align}

\item In $\C^n$ we will consider the inner product $\< \mathbf{x}, \mathbf{y}\> _{\C ^n}=  \sum _{j=1}^{n}x_j {y}_j$, for $\mathbf{z}=(z_1,...,z_n)^\intercal $. All vectors will be column vectors.

\item Given  a linear operator $A \in \mathcal{L}(\C ^n)$ then considered for $1\le k \le n $   the space $  \wedge ^k (\C ^n)$ we denote by $A ^{(k)} \in \mathcal{L}(\wedge ^k (\C ^n))$,  see \cite[Proposition 3.2]{Bealsbook1988},
the operator
\begin{align*}
   A ^{(k)}( a_1 \wedge ...\wedge a_k) =A a_1 \wedge a_2 \wedge  ...\wedge a_k + a_1 \wedge A a_2 \wedge  ...\wedge a_k +...+a_1 \wedge   a_2 \wedge  ...\wedge Aa_k .
\end{align*}
 For an introduction to exterior algebras   see \cite[Definition 2.3.6]{AbrMarsdFound}.

\item  The Hodge star operator $ *:\wedge ^k (\C ^n) \to   \wedge  ^{n-k} (\C ^n)$ is such that
$\<  \alpha , \beta \> _{\wedge ^k (\C ^n)}   =\det ( \alpha \wedge *\beta ) $ where $\alpha \wedge *\beta = \det ( \alpha \wedge *\beta ) e_1 \wedge ...\wedge e_n $, with  $( e_i ) _{i=1}^{n}$ the standard basis of $\C^n$  and where the inner product  $\<  \cdot  , \cdot  \> _{\wedge ^k (\C ^n)} $ is such that $\{  e_{i_1}\wedge ...\wedge e_{i_k} \} $ for all $i_1<...<i_k$ is an orthonormal basis of
$\wedge ^k (\C ^n)$. Notice that for $n$ even we have $ * ^2 =(-1)^k$.  Notice that for vectors in $\C ^4$
 \begin{equation*}
 \det (   X _1 \wedge X_2\wedge X _3\wedge X_4 )=\det (X_1, X_2, X _3, X_4)    ,
 \end{equation*}
with in   the right hand side   the   determinant of the matrix and     for all $ X_1, X_2, X _3, X_4 \in \C ^4$
 \begin{equation}\label{eq:hodgestar}
   *:\wedge  ^3(\C ^4) \to \C ^4 \text { we have  }  \det (X_1, X_2, X _3, X_4)=-\<  X_1, * (  X_2\wedge X _3\wedge X_4)\> _{\C^4}  .
 \end{equation}

\item  Give an operator $T$ for us the resolvent is  $R_{T }(z)=( T-z)^{-1}$.

 \end{enumerate}

\end{notation} 

\section{Solitary waves, Linearization and  modulation  }\label{sec:SLM}

\subsection{Solitary waves}


The following  detailed description of solitary waves is in \cite{BK19JDE}.
\begin{theorem}[Theorem 1.2 of \cite{BK19JDE}]
  \label{thm:solwave} 
  Let $S_c(x-ct)=(n_c,u_c,\phi_c)(x-ct)$ be the solitary wave solutions to \eqref{EP}. Then, for all 
  $c \in (\mathsf{V}, c_K)$, 
  the components of $S_c$ are even  smooth functions strictly decreasing for  $x\in \R _+$.  
  Furthermore, for 
  $\varepsilon   = c- \mathsf{V}$, setting 
  \begin{align*}
     {S}^{\epsilon} _{R }( \sqrt{\varepsilon}  x )
  :=S_c(x)   - \varepsilon  \( 1, \mathsf{V}, 1     \)^\intercal  \psi_{\mathrm{KdV}}\left( \sqrt{\varepsilon} x \right)  
  \end{align*}
  there are constants $\alpha >0$   and, for any $k\in \N \cup \{  0  \}$, $C_k>0$ such that
  
\begin{align}
  \label{eq:solwave2}   &        \sup _{\xi \in \R }    | e^{\alpha    |\xi|}       \partial _\xi ^k {S}^{\varepsilon} _{R } (\xi ) |   \le C_k   \varepsilon    ^{2 }, \quad \xi=\sqrt{\varepsilon}x.
\end{align}

\end{theorem}

\begin{remark}\label{rem:expnS}  $\alpha = \varepsilon ^{-1/2}\mu_4 (0, \varepsilon)$   by  \eqref{eq:defeta2},  \eqref{eq:bigeta12}, \eqref{eq:mu40}, \eqref{eq:oscz4}, \eqref{eq:z4g1} and \eqref{eq:Z14at0}, see below.  
\end{remark}

In Appendix \ref{app:A} we  prove the following,     outlined in    Bae and Kwon \cite[Remark 1]{BK22ARMA}.

\begin{lemma}  \label{lem:solwave4} There are constants $\beta  >0$,   and, for any $k\in \N \cup \{  0  \}$, $C_k>0$ such that for $j=1,2$,
    \begin{align}
   \label{eq:solwave3}    &       \sup _{\xi\in \R }    | e^{\beta     |\xi |}      \partial _\xi ^k   \partial _c^j {S}^{\varepsilon} _{R } (\xi ) |   \le C_k   \varepsilon^{2-j}       .
\end{align}
\end{lemma}

\subsection{Linearized operator}

Following  Bae and Kwon \cite{BK22ARMA} we set
\begin{align}
  \label{eq:lineariz} \mathcal{L}_c :=- \partial _x  \left [   L _c    +    (-\partial ^2_x +e^{\phi _c}) ^{-1}  \left(
                                                  \begin{array}{cc}
                                                    0 & 0 \\
                                                   1 & 0 \\
                                                  \end{array}
                                                \right)
                                          \right ]    \text{ where }L_c=  \left(
                                  \begin{array}{cc}
                                    u_c-c & 1+n_c \\
                                     \frac{K}{1+n_c} & u_c-c \\
                                  \end{array}
                                \right)  .
\end{align}
 Then for
 \begin{align}\label{eq:xici}
   \xi _1 [c](x )   := \partial _x (n_c,u_c)^\intercal \text{ and } \xi _2 [c](x )   := \partial _c (n_c,u_c)^\intercal
 \end{align}
we have
\begin{align}\label{eq:gerker11}
   \mathcal{L}_c \xi _1 [c]=0 \text{ and  }\mathcal{L}_c \xi _2 [c]=- \xi _1 [c].
\end{align}
Notice that, trivially, $  \mathcal{L}_c ^* (a_0,b_0)^\intercal  =0$  for any fixed $a_0,b_0\in \R ^2$ and that it is easy to show that
 \begin{align}\label{eq:adjker}
    \mathcal{L}_c ^* (u_c,n_c)^\intercal  =0 .
 \end{align}
 
 \begin{proposition}\label{prop:Proposition 1.1}Consider the operator $\mathcal{L}_c: (L^2)^2 \to (L^2)^2$ with dense domain $(H^1)^2$. Then, for all sufficiently small $\epsilon>0$, there holds
\begin{equation}\label{eq:spectr}
 \sigma(\mathcal{L}_c)=\sigma_{\mathrm{ess}}(\mathcal{L})=\im \R.
\end{equation}
 Furthermore,
 \begin{equation}\label{eq:spectr1}
  \text{ $\lambda =0$ is the only eigenvalue of $ \mathcal{L}_c$ in }\overline{\C _+}.
\end{equation}

\end{proposition}

\begin{proof}
    While \eqref{eq:spectr}  is verbatim from Bae and Kwon \cite[Proposition 1.1]{BK22ARMA}, \eqref{eq:spectr1} also follows  from  the fact that the Evans function $D(\lambda, \varepsilon ) $, see \eqref{eq:evans}, vanishes only at $\lambda=0$ for all $\lambda \in \overline{\C} _+$. In  Bae and Kwon \cite {BK22ARMA} it is shown that   $\lambda \in {\C} _+$  is an eigenvalue of $\mathcal{L}_c$ if and only if
$D(\lambda , \varepsilon )=0 $. In the case $  \Re \lambda =0$ it is elementary to show,  see later {Claim} \ref{claim:noeig}, that if  $\lambda  $  is an eigenvalue we have $D(\lambda , \varepsilon )=0 $. From this we derive  \eqref{eq:spectr1}.
\end{proof}

 In analogy to Pego and Weinstein \cite {PegoWei2}, setting
 \begin{align}\label{eq:defeta2}
   \theta _3=\theta_3(c)= \frac{1}{\partial _c \< n_c , u _c\>} \text{ and }\eta  _2 [c]=  \theta _3  (u_c,n_c)^\intercal
 \end{align}
 we have, using  the Kronecker delta,
 \begin{align}\label{orth:xieta2}
    \<  \xi _i [c], \eta  _2 [c] \> = \delta _{i2}.
 \end{align}
 Setting
\begin{align}& \label{eq:ker2}
   \eta  _1 [c] (x) = \theta _1\int _{-\infty}^x  \partial _c (u_c(x'),n_c(x') )^\intercal   dx' +  \theta _2 (u_c(x ),n_c(x ) )^\intercal \text{ with} \\& \theta _1 =- \theta _3 \text{ and }
   \theta _2 = \theta _3^2 \int_\R \partial_c n_c\,dx \int_{\R}\partial_c u_c\,dx \nonumber
\end{align}
 we have
 \begin{align}\label{orth:xieta}
    \<  \xi _i [c], \eta  _1 [c] \> = \delta _{i1}.
 \end{align}
 We notice that from Theorem \ref{thm:solwave} and Lemma \ref{lem:solwave4}, we have $ \theta_3=-\theta_1\sim \varepsilon^{-1/2}$ and $\theta_2\sim \varepsilon^{-2}$ 
 and that
 \begin{align}
   \label{eq:dualkernel}  \mathcal{L}_c ^* \eta  _2 [c] =0 \text{   and }\mathcal{L}_c ^* \eta  _1 [c] =\eta  _2 [c].
 \end{align}

\subsection{Modulation}

We will write $ \widetilde{U}= (n , u    )^\intercal $   and seek an  ansatz
 \begin{align}
   \label{eq:ansatz}U= S_c   (\cdot -D) + V (\cdot -D)  \text{ where } V=( \widetilde{V}, V_\phi ) ^\intercal  \text{ and } \widetilde{V}=( V_n, V_u  ) ^\intercal.
 \end{align}
For  $c  _0 ,r \in \R_+:=(0,\infty)$ and for   $D _{L^2  (\R , \R ^2 )}(\widetilde{U},r):=\{ \widetilde{\Psi } \in L^2  (\R , \R ^2 ) \ |\ \|\widetilde{U}  -\widetilde{\Psi} \| _{ L^2   (\R )}<r \}$,  we set     $$\mathcal{U} (c _0 ,r  ) := \bigcup _{  x_0\in \R }  e^{-x_0 \partial _x}  D_{L^2  (\R , \R ^2 )}(\widetilde{S}_{c _0},r ). $$
The following can be proved  exactly like in \cite[Lemma 3.3] {CM25D1}. See Appendix \ref{app:B}    for the proof.
 
\begin{lemma}[Modulation]\label{lem:mod1} For $c_0\in ( \mathsf{V}, c_K)$ there exist     $\delta _0 >0$  and  $ B_0>0$   with $\delta _0 \ll B_0 ^{-1}$
such that for $  \delta _0 ^{-1 }\gg    B\ge B_0$ and for the  $\zeta _{B}$ introduced in \eqref{def:zetaC} there are
$(c,D)     \in C^1\(  \mathcal{U} (c_0 ,\delta _0 ), \R _+  \times \R   \) $
 such that for  any $\widetilde{U} \in \mathcal{U} (c_0 ,\delta _0   )$ we have
\small \begin{align}\label{61} &  \widetilde{V} (\widetilde{U} ):=e^{   D(\widetilde{U}) \partial _x } \widetilde{U} -  \widetilde{S}  _{c (\widetilde{U})}     \text{ satisfies }   \\& \nonumber \<   \widetilde{V} (\widetilde{U}) , \zeta _{B} \eta _1[c (\widetilde{U})]   \> = \<   \widetilde{V} (\widetilde{U}) ,   \eta _2[c (\widetilde{U})]   \> =0. \end{align} \normalsize
 For any $c_1$ with $ \widetilde{S}_{c_1}\in \mathcal{U}(c_0,\delta_0)$   and $D_1\in \R$         we have the identities
\begin{align}& c  (  \widetilde{S}   _{c_1}  )  =c_1  \text{, }   D  (  \widetilde{S}   _{c_1}  )=0
  \text{, }   D  (  e^{  - D _1 \partial _x }  \widetilde{U})    = D  (    \widetilde{U} )    + D_1 \text{ and } c  (  e^{  - D _1 \partial _x }  \widetilde{U})    =  c  (    \widetilde{U}).\label{eq:equivariance}
\end{align}
Furthermore  there exists a constant $C >0$ such that
\begin{align}
  \label{eq:mod2}  B ^{-\frac{1}2}|D(\widetilde{U})  -D_0| + |c(\widetilde{U} )- c_0|\le C       \delta _0 \text{ for all  } \widetilde{U}\in  e^{-D_0 \partial _x}  D_{L^2  (\R , \R ^2 )}(\widetilde{S}_{c _0}, \delta  _0    ).
\end{align}
\end{lemma}


By taking $\delta \in (0, \delta _0/2 )$ in \eqref{eq:main1}, there exists a function $t\mapsto x_0(t)$ such that
\begin{align*}
   \text{$\widetilde{U} (t)\in  e^{-x_0(t) \partial _x}  D_{L^2  (\R , \R ^2 )}(\widetilde{S}_{c _0}, 2\delta     )$ for all $t\ge 0$}
\end{align*}
and so, for $D(t)=D(\widetilde{U} (t)) $,   $c(t)=c(\widetilde{U} (t)) $  and $\widetilde{V}(t)=\widetilde{V} (\widetilde{U} (t)) $, we have
\begin{align*} &
   2\delta > \|  \widetilde{U} (t) - e^{-x_0(t) \partial _x} \widetilde{S}_{c _0} \| _{H^1}=  \|  e^{-D(t) \partial _x}\widetilde{V} (t)    +    e^{-D(t) \partial _x} \widetilde{S}_{c(t)}    - e^{-x_0(t) \partial _x} \widetilde{S}_{c _0} \| _{H^1} \\& \ge  \|   \widetilde{V} (t) \| _{H^1}    - \|    e^{-D(t) \partial _x} \widetilde{S}_{c(t)}    - e^{-x_0(t) \partial _x} \widetilde{S}_{c _0} \| _{H^1}
\end{align*}
so that by \eqref{eq:mod2}
\begin{align}\label{eq:h1errorbounds} 
   \|   \widetilde{V} (t) \| _{H^1} \lesssim \delta + |D(t) -  x_0(t) |+ |c(t) -  c_0  |  \lesssim  B ^{ \frac{1}2} \delta \text{ and }|c(t) -  c_0  |\lesssim \delta.
\end{align}

\section{Main estimates} \label{sec:mainest}

The proof of Theorem \ref{thm:main} is mainly  based on the following continuation argument.

\begin{proposition}\label{prop:continuation}
For any small $ \epsilon >0$ and for large constants $A$ and $B$ as in formula \eqref{eq:relABg}
there exists  a    $\delta _0= \delta _0(\epsilon )   $ s.t.   if  for an interval   $I=[0,T]$ for some $T>0$ the following inequality holds,
\begin{align}
  \label{eq:main2} \|V \| _{L^2(I,  { \Sigma }_{1A A_1} )} +
\|V \| _{L^2(I,  { \Sigma }_{2A A_1} )} +  \|V \| _{L^2(I, \widetilde{\Sigma}  )}  +   \| \dot D - c \| _{L^2(I  )} + \| \dot c \| _{L^2(I  )}\le \epsilon
\end{align}
   and if  the constant in \eqref{eq:main1} satisfies  $\delta  \in (0, \delta _0)$
then     \eqref{eq:main2} holds   for   $\epsilon$ replaced by $   o _{ B ^{-\frac{1}{2}} }(1) \epsilon $.
\end{proposition}

We will split the proof of Proposition \ref{prop:continuation}  into three lemmas obtained assuming the hypotheses of Proposition  \ref{prop:continuation}. The first is  proved in \S  \ref{sec:discr}, the second in  \S  \ref{sec:virial}, the last in the rest of the paper.
\begin{lemma}\label{lem:lemdscrt} We have the estimates
 \begin{align} \label{eq:discrest1D}
   &     |\dot D -c |    \lesssim  B ^{-\frac{1}{2}}   (\|V \| _{  { \Sigma }_{1A A_1}  } +
\|V \| _{   { \Sigma }_{2A A_1}  } )  \text{  and }\\& |\dot c   |     \lesssim  \|V \| _{   { \Sigma }_{1A A_1}  } ^2 +
\|V \| _{   { \Sigma }_{2A A_1} } ^2 . \label{eq:discrest1c}
\end{align}

\end{lemma}

\begin{lemma}[Virial Inequality]\label{prop:1virial}
 We have
 \begin{align}\label{eq:sec:1virial11}
   \|  V  \|_{ L^2( I ,  \Sigma  _{1A A_1} )} &\lesssim \sqrt[3]{ {\delta }} + \|   V  _\phi    \|_{   L^2( I ,\widetilde{\Sigma})    }  +  B ^{-1/2}
\|V \| _{  L^2( I ,    \Sigma  _{2A A_1} ) }   \text{ and}\\ \|  V  \|_{ L^2( I ,  \Sigma  _{2A A_1}) } &\lesssim \sqrt[3]{ {\delta }}   +    A_1 ^{-1}
\|  V  \|_{ L^2( I ,  \Sigma  _{1A A_1} )}. \label{eq:sec:1virial12}
 \end{align}
 \end{lemma}

\begin{lemma}[Smoothing Inequality]\label{prop:smooth11}
 We have
 \begin{align}\label{eq:sec:smooth11}
   \|  V  _\phi  \| _{L^2(I, \widetilde{\Sigma}   )} \lesssim       A _1  ^{3/2} B ^{-1}   \| V \| _{L^2(I,  { \Sigma }_{1A A_1}   )} +     B ^{ 1/2}   \| V \| _{L^2(I,  { \Sigma }_{2A A_1}   )} +  o _{B^{-1}} (1)  \epsilon     .
 \end{align}
 \end{lemma}

\textit{Proof of Proposition \ref{prop:continuation}  assuming Lemma \ref{prop:1virial}--\ref{prop:smooth11}.}   Using $A_1 = B ^{3/5}$ like in  \eqref{eq:relABg}
 from  \eqref{eq:sec:smooth11}  and  \eqref{eq:sec:1virial11}-- \eqref{eq:sec:1virial12} we obtain  \small
\begin{align*}
  & \|  V  _\phi  \| _{L^2(I, \widetilde{\Sigma}   )}   \lesssim       A _1  ^{3/2} B ^{-1}   \| V \| _{L^2(I,  { \Sigma }_{1A A_1}   )} +     B ^{ 1/2}   \| V \| _{L^2(I,  { \Sigma }_{2A A_1}   )} +  o _{B^{-1}} (1)  \epsilon 
\\& \lesssim  o _{B^{-1}} (1)  \epsilon  +     A _1  ^{\frac{3}{2}  } B ^{-1}      \|  V  _\phi  \| _{L^2(I, \widetilde{\Sigma}   )} + B ^{ \frac{1}{2} }  A_1 ^{-1} \| V \| _{L^2(I,  { \Sigma }_{1A A_1}   )} +  A _1  ^{\frac{3}{2} } B ^{-\frac{3}{2}}  \| V \| _{L^2(I,  { \Sigma }_{2A A_1}   )}\\& \lesssim  o _{B^{-1}} (1)  \epsilon  +   B ^{-\frac{1}{10}}  \(    \| V  _\phi  \| _{L^2(I, \widetilde{\Sigma}   )} +    \| V \| _{L^2(I,  { \Sigma }_{1A A_1}   )} +   \| V \| _{L^2(I,  { \Sigma }_{2A A_1}   )} \) = o _{B^{-1}} (1)  \epsilon .
\end{align*}
\normalsize
Then $\|  V  _\phi  \| _{L^2(I, \widetilde{\Sigma}   )}  =o _{B^{-1}} (1)  \epsilon . $ Inserting this in  \eqref{eq:sec:1virial11} shows that 
\begin{align*}
    \| V \| _{L^2(I,  { \Sigma }_{1A A_1}   )}  +       \| V \| _{L^2(I,  { \Sigma }_{2A A_1}   )} =o _{B^{-1}} (1)  \epsilon .
\end{align*}
Finally using this information and Lemma  \ref{lem:lemdscrt} the proof of  Proposition \ref{prop:continuation} is completed. \qed

We introduce the invariants $E$ and $M$.
There is an energy associated to \eqref{EP}
 \begin{align}&\label{eq:energy}
   E= E(U) = \int _\R e(U) dx \text{  with energy density} \\&  e(U)= (1+n) \frac{u^2}{2}+ K \( (1+n)\log (1+n)  -n  \) -\frac{\( \phi '\) ^2}{2}+n\phi - \(  e^{\phi} -1-\phi  \) .\nonumber
 \end{align}
 It will be convenient to split
 \begin{align}& \nonumber
   E = E_K + E _P  \text{  with  }  E_L(u) = \int _\R e_L(U) dx \text{ for }L=K,P  \text{  with  } \\& \label{eq:enkinpot}
     e_K(U)=  \frac{u^2}{2}+ K  \frac{n^2}{2}    -\frac{\phi ^{\prime 2}}{2}-\frac{\phi ^{2} }{2}    +n\phi   , \\&              e_P(U)=    n  \frac{u^2}{2}+ K \( (1+n)\log (1+n)  -n  -\frac{n^2}{2} \)   - \(  e^{\phi} -1-\phi -\frac{\phi ^{2} }{2} \) .  \nonumber
 \end{align}
 Equation \eqref{EP}  can be reformulated as
 \begin{align}\label{EP1}
    \left\{
        \begin{aligned}
            &\dot {\widetilde{U}} =-  \partial _x \sigma _1 \nabla_{\widetilde{U}} E(U)   \text{ where } \sigma _1= \left(
                                                                                                                      \begin{array}{cc}
                                                                                                                        0 & 1 \\
                                                                                                                        1 & 0 \\
                                                                                                                      \end{array}
                                                                                                                    \right)
             \\
            & \partial _{\phi }E(U) =0 ,
        \end{aligned}
    \right.
\end{align}
 where the gradient is computed in terms of the standard inner product in $L^2(\R , \R ^d)$.
 Formally, exploiting $\partial _\phi e (U)=0$ , the energy is conserved by  \eqref{EP1} by
 \begin{equation}\label{eq:consen}
  \begin{aligned}
    \frac{d}{dt}E(U ) &=\< \nabla_{\widetilde{U}} E(U), -  \partial _x \sigma _1 \nabla_{\widetilde{U}} E(U) \> +\< \dot \phi , \partial _{\phi }E(U)\> \\& = \frac{1}{2}\int _\R \(  \< \nabla_{\widetilde{U}} E(U), -  \sigma _1 \nabla_{\widetilde{U}} E(U) \> _{\R ^2}\) '  dx =0.
  \end{aligned}
 \end{equation}
Another formal invariant   for \eqref {EP1} is the following which,  unlike  for $E$, we   show is    conserved,
\begin{align}\label{eq:mass}
  M(U) =\int _{\R} m(U) dx \text{ where } m(U) = n u.
\end{align}

\begin{lemma}\label{lem:conserveM}
    We have $M(U(t))=M(U(0))$ for all $t>0$.
\end{lemma}

\begin{proof} Above Theorem \ref{thm:main}, we already observed that 
  $U\in C^1([0,\infty),L^2(\R))$. Then $M(U(\cdot))\in C^1([0,\infty),\R)$.
      We show $\frac{d}{dt}M(U(t))=0$.
        Assumption \eqref{eq:main1} implies $\|n\|_{L^\infty}\lesssim \delta+ \varepsilon_0\ll 1$.
    Thus 
    \begin{align}
        \frac{d}{dt}M(U(t))&=-\int_{\R}((1+n)u)'u-\int_{\R}n\left(\frac{1}{2}u^2+K\log(1+n)+\phi\right)'\nonumber\\&
        =\int_{\R}(1+n)uu'+\frac{1}{2}\int_{\R} n'u^2 + K\int_{\R} n'\log(1+n)-\int n\phi'.\label{eq:momentum}
    \end{align}
    Notice that all the integrands in the right hand side of \eqref{eq:momentum} are integrable and the first two terms cancel each other.
    For the 3rd term,  using $(1+n)\log(1+n)-n,\phi'^2 \in W^{1,1}$  we have
    \begin{align*}
        K\int_{\R}n'\log(1+n)=K\int_{\R}((1+n)\log(1+n)-n)'=0,
    \end{align*}
    and for the 4th term of \eqref{eq:momentum},    using $e^\phi-1-\phi  \in W^{1,1}$  we have
    \begin{align*}
        -\int_{\R} n\phi'&=\int_{\R} \phi'' \phi' -\int_{\R}(e^\phi-1)\phi'
        =\frac{1}{2}\int_{\R} (\phi'^2)'-\int_{\R} (e^\phi-1-\phi)'=0. 
    \end{align*} 
\end{proof}

For a solution $  U\in C^0\(   [0,+\infty ), H^1(\R )\)$, 
inserting in \eqref{EP1}    the ansatz   
 using  the identity
\begin{align*}
  \nabla  E( e^{  - D  \partial _x } U ) = e^{  - D  \partial _x } \nabla  E(   U )  \text{ for all $D\in \R$ and $U\in L^2(\R , \R ^3)$, }
\end{align*}
and  after elementary computations and cancellations   we obtain
\begin{align}\label{EP2}
    \left\{
        \begin{aligned}
            &\dot {\widetilde{V}} =-  \partial _x \sigma _1 \( \nabla_{\widetilde{U}} E(S_c+V)      -  \nabla_{\widetilde{U}} E(S_c ) \) +  c\widetilde{V}'+  (\dot D-c) \( \widetilde{S}_c+ \widetilde{V} \)'-\dot c \partial _c \widetilde{S}_c            \\
            &  V _\phi ''+V _n  =   e^{\phi _c} \(  e^{V _\phi}-1 \) .
        \end{aligned}
    \right.
\end{align}
Notice, for subsequent use,  that from the splitting of the energy in \eqref{eq:enkinpot},   we have
\begin{align}\label{eq:graden1}&
    \nabla_{ {U}} E_K(U) =   (K n +  \phi , u,  \phi ''- \phi + n)^\intercal   \\& \label{eq:graden2}
    \nabla_{ {U}} E_P(U) = \(    \frac{u^2}{2}   +   K \( ( \log (1+n)  -n   \)   , nu,  -\(  e^{\phi} -1-\phi \) \)^\intercal
\end{align}
and that
\begin{align}
  \label{eq:graden3}  \nabla_{ {U}}^2 E (U)  = \left(
                                                                        \begin{array}{ccc}
                                                                           \frac{K}{1+n} & u & 1 \\
                                                                          u & 1+n& 0 \\
                                                                          1 & 0 &  \partial _x^2-1-(e^{\phi} -1) \\
                                                                        \end{array}
                                                                      \right) .
\end{align}
In particular we have
\begin{align} &
  \nabla_{\widetilde{U},  U}^2 E(S_c ) =\left(
                                                                        \begin{array}{ccc}
                                                                           \frac{K}{1+n _c} & u_c & 1 \\
                                                                          u_c & 1+n_c& 0  \\
                                                                        \end{array}
                                                                      \right)    \text{ and}
\nonumber \\&
  \label{eq:graden4}  \nabla_{ {U}}^2 E _P(U)  = \left(
                                                                        \begin{array}{ccc}
                                                                          K \(  \frac{1}{1+n} -1\) & u & 0 \\
                                                                          u &  n& 0 \\
                                                                          0 & 0 &  -(e^{\phi} -1) \\
                                                                        \end{array}
                                                                      \right) .
\end{align}

\textit{Proof of Theorem \ref{thm:main}.}   By $U\in C^0([0,+\infty),H^1(\R))$ it is elementary that    in  Proposition \ref{prop:continuation}  we can take $I=\R _+$.   This  implies   the following stronger version of \eqref{eq:asstab2},
\begin{align}    \label{eq:asstab2strong}   \int _{\R _+} \|  e^{- a\< x\>}   \widetilde{V} (t ) \| _{L^2(\R )}^2 dt    +  \int _{\R _+} \|  e^{- a\< x\>}    {V}_\phi  (t ) \| _{H^1(\R )}^2 dt   <  \epsilon   .
\end{align}
Set now for $a>0$  \begin{equation*}
  \begin{aligned}
   \mathbf{a}(t) &:= 2 ^{-1}\<  e^{- a\< x\>}   \widetilde{V} (t), \sigma _1 \widetilde{V} (t)  \>.
  \end{aligned}
  \end{equation*}
By  Lemma \ref{lem:lemdscrt} and \eqref{eq:main1} the following  yields $\frac{d}{dt} \mathbf{a}\in L^\infty (\R _+ )$,
\begin{align} \label{eq:dera}
  \frac{d}{dt} \mathbf{a} &=\<    \dot { \widetilde{V}}  , \sigma _1    e^{- a\< x\>}\widetilde{V}    \>   = \<   e^{- a\< x\>} \dot {\widetilde{V}}  , \sigma _1\widetilde{V}  \> \\ &= (\dot D-c)  \<   e^{- a\< x\>} \xi _1 [c] , \sigma _1\widetilde{V}  \>  +  \dot c    \<   e^{- a\< x\>}   \xi _2 [c] , \sigma _1\widetilde{V}  \>  - 2^{-1} (\dot D-c)  \<  ( e^{- a\< x\>}) ' \widetilde{V}  , \sigma _1\widetilde{V}  \> \nonumber \\& - \<   e^{- a\< x\>}     \partial _x   \( \nabla_{ {U}} E(S_c+V)      -  \nabla_{ {U}} E(S_c ) \)  ,   {V}  \>  . \nonumber
\end{align} 
 Then $\mathbf{a}\in L^1(\R _+ )$   and $\frac{d}{dt} \mathbf{a}\in L^\infty (\R _+ )$ imply
$ \mathbf{a}(t) \xrightarrow{t\to +\infty} 0$, proving \eqref{eq:asstab20}.
Finally  \eqref{eq:asstab3} is an immediate consequence of \eqref{eq:asstab2} and \eqref{eq:discrest1c} which together imply
$\| \dot c \| _{L^1(\R _+ )}\lesssim \epsilon ^2.$

\section{Estimates  on some Schr\"odinger operators}\label{sec:prlest}

We  will need some estimates related to the Schr\"odinger operator with the potential $e^{\phi _c}-1$
\begin{align}
  \label{eq:defhc}h_c=-\partial _x^ 2 +  e^{\phi _c}-1.
\end{align}
\begin{lemma}
  \label{lem:estqc}
There exists a uniform constant $C$ such   that  for  $\mathsf{V}=\sqrt{1+K}$ we have
\begin{align}\label{eq:estpot}
  \| \< x \> (e^{\phi _c}-1)  \|_{L^1(\R )}\le C \text{ for any $c>\mathsf{V}$ close enough to $\mathsf{V}$} .
\end{align}
\end{lemma}
\begin{proof}
We claim that there exists a uniform constant $C$ such that
\begin{align}\label{eq:estpot1}
  \| \< x \> \phi _c \|_{L^1(\R )}\le C \text{ for any $c>\mathsf{V}$ close enough to $\mathsf{V}$} .
\end{align}
Indeed, if \eqref{eq:estpot1} is true, then since 
\[
| e^{\phi _c(x)}-1-\phi _c(x)| \le \phi _c^2(x) \le  C_0 \phi _c (x),
\]
it is immediate to see that \eqref{eq:estpot1} yields \eqref{eq:estpot}.
By Theorem~\ref{thm:solwave}, we see that
\[
\| \< x \> \phi_c(x) \|_{L^1(\R )} = \int_{\R} \sqrt{1+x^2}\,  |\phi_c(x)|\,dx \leq  \int_{\R} \sqrt{1+ \frac{\xi^2}{\varepsilon}} \,\varepsilon e^{-\alpha|\xi|}\,\varepsilon^{-1/2} d\xi \leq C.
\]
for all sufficiently small $\varepsilon$. We finish the proof of Lemma~\ref{lem:estqc}.
\end{proof}
We recall now the so called Jost solutions  $f_{ (c) \pm } (x,k )=e^{\pm \im k x}m_{(c) \pm } (x,k )$, characterized by
\begin{align*}
  (-\partial _x^ 2 +  e^{\phi _c}-1)f_{ (c) \pm } (\cdot ,k )=k^2 f_{ (c) \pm } (\cdot ,k ) \text{ and }\lim _{x\to \pm \infty }   m_{ (c) \pm  } (x,k)   =1 
\end{align*}
and  the  transmission coefficient $T(c, k)$ defined by
\begin{align*}
  \frac{1}{T(c, k)}=\frac{1}{2\im k }[f_{ (c) + } (\cdot ,k ),f_{ (c) - } (\cdot ,k )]  \text{   where }
[f(x ),g  (x )] := f'(x) g(x)-f(x) g'(x).
\end{align*}
A consequence of Lemma \ref{lem:estqc} is the following lemma, for which we refer to    \cite{DT1979}.
Notice that since $\phi_c(-x)=\phi_c( x)$ it follows $f_{ (c) - } (x ,k )= f_{ (c) + } (-x ,k )  $ so that it yields also estimates for $f_{ (c) - } (x ,k )   $.

\begin{lemma}\label{lerm:jost} 
Let $q_c = e^{\phi_c}-1$. For (only for this lemma)  \small
\begin{align}\gamma (x):=\int _{x} ^{+\infty}(y-x)  |q_{c} (y)| dy, \quad \eta (x):=\int _{x} ^{+\infty}   |q_{c} (y)| dy
\text{ and }
K:= 1+\gamma (0) +\gamma^2 (0) e^{\gamma (0)}  \label{eq:DTnot1}
\end{align}\normalsize
we have,   using from Notation \ref{not:notation} $x^-:=\max ( 0, -x)$, 
 \begin{align}  \label{eq:kernel2} &  | m_{ (c) +  } (x,k) |\le  K \(  1+ x ^{-}\)  \exp \( \int  _{x } ^{+\infty} (1+|y|)  |q_{c} (y)|  dy \)     \ ,   \\ & \label{eq:kernelb2}   | m_{ (c) +  } (x,k) |\le     e^{\gamma (x)}   \ ,   \\ &  |m_{ (c) + }(x, k )-1|\le    |    k|  ^{-1} \eta (x)    \exp \(  |   k|  ^{-1} \eta (x) \)    \text{ and} \label{eq:kernel2aw0}\\& |\partial _xm_{ (c) + }(x, k ) |\le K \int _{0} ^{+\infty}y  |q_{c} '(y)| dy    +  K x^{-} \(  1+ x ^{-}\) e^{ \int  _{x } ^{+\infty} \( (1+|y|)  |q_{c} (y)| + |q_{c} '(y)|\)   dy}   . \label{eq:kernel2der}  \end{align}
Furthermore we have  $|T(c, k)|\le 1$, $T(c, 0) =0$  and
\begin{align}\label{eq:Transm1}
 2 |k|\le |T(c, k)| \( 2 |k| +K \| (1+|x|) q _c \|_{L^1(\R )}\)   \text{ for all }k\in \R .
\end{align}

\end{lemma}
 \begin{proof}
   Inequality \eqref{eq:Transm1} is in p. 147 \cite{DT1979};  the estimates \eqref{eq:kernel2}--\eqref{eq:kernel2der} are in Lemma 1 p. 130 \cite{DT1979}. Notice that for all $c$ close enough to $\mathsf{V}$ there is a uniform bound in $c$ on the quantities in \eqref{eq:DTnot1}.
 \end{proof}

We recall now   for any $z\not \in [0,+\infty )$  the integral kernel of the resolvent $R _{h_c}(z) = ( h_c -z  )^{-1} $ ,
\begin{align}
  \label{eq:resolvKern}  R _{h_c   }(z ) (x,y)= \left\{
                                                  \begin{array}{ll}
                                                   - \frac{T(k)}{2\im k}     f_- (x, k)   f_+ (y, k), & \hbox{for $x<y$} \\
                                                     -\frac{T(k)}{2\im k}     f_+ (x, k)   f_- (y, k), & \hbox{for $x>y$ }
                                                  \end{array}
                                                \right.
\end{align} where $k ^2=z$ with $\Im k> 0$.  By
 Lemma \ref{lerm:jost}    there exists a fixed constant $C>0$ such that 
\begin{align}
  \label{eq:resolvat1} |\partial _x ^{a}\partial _y ^{b} R _{h_c   }(-1 ) (x,y)|\le C  \< x \> ^a  \< y \> ^b    e^{-|x-y|} \text{ for all $c$ close enough to }\mathsf{V} \text{ and for }a+b\le 1.
\end{align}

\section{Proof of Lemma \ref{lem:lemdscrt} }\label{sec:discr}

 We apply to the first equation  \eqref{EP2}  the inner product $\< \cdot , \zeta_{B} \eta _1[c] \>$.
 Using
\begin{align*}
    \< \dot {\widetilde{V}}  , \zeta_{B} \eta _1[c] \> =-\dot c   \<   {\widetilde{V}}  , \zeta_{B} \partial _c\eta _1[c] \> , 
\end{align*}
which follows from the second line in \eqref{61}, using \eqref{eq:xici} we
  obtain
\begin{align}&
  - ( \dot D -c )   \< \xi _1[c] +  \widetilde{V}' , \zeta_{B} \eta _1[c] \>      + \dot c  \left [ \< \xi _2[c] , \zeta_{B} \eta _1[c] \>    -    \<   {\widetilde{V}}  , \zeta_{B} \partial _c\eta _1[c] \> \right ] \label{eq:seceq}\\& = \<   -  \partial _x \( \sigma _1  \nabla_{\widetilde{U},  U}^2 E(S_c )    V  -c\widetilde{V} \),  \zeta_{B}  \eta _1[c] \>   \nonumber \\& + \<     \sigma _1 \( \nabla_{\widetilde{U}} E(S_c+V)      -  \nabla_{\widetilde{U}} E(S_c ) - \nabla_{\widetilde{U},  U}^2 E(S_c )    V \)   ,\( \zeta_{B} \eta _1[c]\) ' \>  =: A _{1 }+A _{2 }. \nonumber
\end{align}
Notice that  for the  $L_c$  of 
\eqref{eq:lineariz}  and from \eqref{EP2} we have
\begin{align}&  \nonumber
   \sigma _1     \nabla_{\widetilde{U},  U}^2 E(S_c )    V  -c\widetilde{V}  = L_c  \widetilde{V} + ( 0, V_\phi ) ^\intercal \\& =  L_c  \widetilde{V} + \( 0,  (-\partial _x ^2 + e^{\phi _c} )^{-1}     V_n \) ^\intercal   -  \( 0,  (-\partial _x ^2 + e^{\phi _c} ) ^{-1}  e^{\phi _c} \(  e^{V _\phi}-1-V _\phi \)  \) ^\intercal .    \label{eq:appLc}
\end{align}   Then for the $ \mathcal{L}_c$ in
\eqref{eq:lineariz} we have
\begin{align*}
   A _{1 }&=  \<   \mathcal{L}_c   \widetilde{V} ,  \zeta_{B}  \eta _1[c] \>  -    \<   \( 0,  (-\partial _x ^2 + e^{\phi _c} ) ^{-1}  e^{\phi _c} \(  e^{V _\phi}-1-V _\phi \)  \) ^\intercal  , ( \zeta_{B}  \eta _1[c] ) ' \> =:  A _{1 1}+ A _{1 2}.
\end{align*}    
We have
\begin{align*}
  A _{1 1}=  \<     \widetilde{V} ,  \zeta_{B}  \mathcal{L}_c^* \eta _1[c] \> + \<     \widetilde{V} , [  \mathcal{L}_c^*, \zeta_{B} ] \eta _1 [c] \> =:A _{1 11}+ A _{1 12}.
\end{align*}
Now, by \eqref{eq:dualkernel} and \eqref{61},
\begin{align*}
  A _{1 11}=    \<     \widetilde{V} ,  \zeta_{B}  \mathcal{L}_c^* \eta _1[c] \> = -\<     \widetilde{V} , \zeta_{B}  \eta _2[c] \>=\<     \widetilde{V} , (1-\zeta_{B})  \eta _2[c] \>  .
\end{align*}
Then by the exponential decay of $\eta _2[c]$, see \eqref{eq:defeta2} and Theorem \ref{thm:solwave}, 
\begin{align*}
  |A _{1 11}| &\le \| \zeta _A  (\vartheta _{1A_1} + \vartheta _{2A_1}) \widetilde{V} \| _{L^2}  \| \zeta _A ^{-1} \zeta_{B}  \(  e  ^{\frac{|x|}{B}(1-\chi )}    -1\)  \eta _2[c]  \| _{L^2}\\&  \le  (  \| V  \|_{  { \Sigma }_{1A A_1}} + \| V  \|_{  { \Sigma }_{2A A_1}}  )  \|   \zeta_{B}     \frac{|x|}{B}  \eta _2[c]  \| _{L^2}\lesssim B ^{-1}  (  \| V  \|_{  { \Sigma }_{1A A_1}} + \| V  \|_{  { \Sigma }_{2A A_1}}  )  .
\end{align*}
An  elementary computation gives
\begin{align*}
 A _{1 12}&=\<  \widetilde{V} ,  L_c^\intercal \zeta_{B}'  \eta _1[c]     \> + \<  \widetilde{V} ,  \left(
                                                  \begin{array}{cc}
                                                    0 &  (-\partial ^2_x +e^{\phi _c}) ^{-1} \\
                                                   0 & 0 \\
                                                  \end{array}
                                                \right)\zeta_{B}'  \eta _1[c]     \> \\& +  \<  \widetilde{V} ,  \left(
                                                  \begin{array}{cc}
                                                    0 &  [(-\partial ^2_x +e^{\phi _c}) ^{-1}, \zeta_{B} ] \\
                                                   0 & 0 \\
                                                  \end{array}
                                                \right)   \eta _1'[c]     \> = \sum _{k=1}^3 A _{1 12k}.
\end{align*}
By $2B<A$ and inserting the partition of unity $1=\vartheta _{1A_1} + \vartheta _{2A_1}$
we have
\begin{align*}
  |A _{1 121}|&\lesssim \| \zeta _A (\vartheta _{1A_1} + \vartheta _{2A_1})\widetilde{V} \| _{L^2} \frac{1}{B}\|\frac{ \zeta _B}{\zeta _A}   \| _{L^2}    \lesssim B ^{-1/2}(  \| V  \|_{  { \Sigma }_{1A A_1}} + \| V  \|_{  { \Sigma }_{2A A_1}}  ) .
\end{align*}
Abusing notation by treating $ \widetilde{V}$ like a scalar valued function,   schematically  we write
\begin{align*}
  A _{1 122}= \<  \widetilde{V} ,(-\partial ^2_x +e^{\phi _c}) ^{-1}\zeta_{B}'  \eta _1[c]     \>.
\end{align*}
Then
\begin{align*}
  |A _{1 122}|&\lesssim \| \zeta _A (\vartheta _{1A_1} + \vartheta _{2A_1})\widetilde{V} \| _{L^2} \| \zeta _A^{-1} R _{h_c   }(-1 )\zeta_{B}'  \eta _1[c] \| _{L^2} \\& \le  (  \| V  \|_{  { \Sigma }_{1A A_1}} + \| V  \|_{  { \Sigma }_{2A A_1}}  ) \| \zeta _A^{-1} R _{h_c   }(-1 )\zeta _A \|_{L^2\to L^2}    \|  \zeta _A^{-1}\zeta_{B}'  \eta _1[c] \| _{L^2} \\& \lesssim(  \| V  \|_{  { \Sigma }_{1A A_1}} + \| V  \|_{  { \Sigma }_{2A A_1}}  )  B ^{-1}\|\frac{ \zeta _B}{\zeta _A}   \| _{L^2}\lesssim B ^{-1/2}(  \| V  \|_{  { \Sigma }_{1A A_1}} + \| V  \|_{  { \Sigma }_{2A A_1}}  )
,
\end{align*}
where we used the fact that there exists a fixed constant $c_0>0$ such that
\begin{align*}
  \sup _{x\in \R}\int _\R  \zeta _A^{-1}(x)  |R _{h_c   }(-1 )(x,y) | \zeta _A(y) dy +  \sup _{y\in \R}\int _\R  \zeta _A^{-1}(x)  |R _{h_c   }(-1 )(x,y) | \zeta _A(y)  dx \le c_0C,
\end{align*}
which is a straightforward consequence of \eqref{eq:resolvat1} and   in turn, by the Young's inequality, \cite[p. 22]{SoggeFourIntegrals}, yields
$\| \zeta _A^{-1} R _{h_c   }(-1 )\zeta _A \|_{L^2\to L^2} \le c_0C$.
 Writing again schematically, we have
\begin{align*}
  A _{1 123}= \<  \widetilde{V} ,[(-\partial ^2_x +e^{\phi _c}) ^{-1}, \zeta_{B} ]  \eta _1 '[c]     \>
\end{align*}
which in turn splits as
\begin{align*}
  A _{1 123}&= \<  \widetilde{V} , R _{h_c   }(-1 ) [  \partial ^2_x  , \zeta_{B} ] R _{h_c   }(-1 )  \eta _1 '[c]     \> = \<  \widetilde{V} , R _{h_c   }(-1 )  \zeta_{B} '' R _{h_c   }(-1 )  \eta _j '[c]     \> \\&+2 \<  \widetilde{V} , R _{h_c   }(-1 )  \zeta_{B} ' \partial _x R _{h_c   }(-1 )  \eta _1 '[c]     \> =:   A _{1 1231}+ A _{1 1232}.
\end{align*}
We have
\begin{align*}
  |A _{1 1231}|& \lesssim \|   \zeta _A (\vartheta _{1A_1} + \vartheta _{2A_1})\widetilde{V} \| _{L^2} \| \zeta _A^{-1} R _{h_c   }(-1 )\zeta _A \|_{L^2\to L^2}
\| \frac{\zeta_{B} ''}{ \zeta _A }\| _{L^\infty} \|   R _{h_c   }(-1 ) \eta _1 '[c] \|_{L^2 }\\& \lesssim B^{-2}(  \| V  \|_{  { \Sigma }_{1A A_1}} + \| V  \|_{  { \Sigma }_{2A A_1}}  ) \|    \eta _1 '[c] \|_{L^2 }\lesssim B^{-2}(  \| V  \|_{  { \Sigma }_{1A A_1}} + \| V  \|_{  { \Sigma }_{2A A_1}}  )
\end{align*}
where for the last inequality we used Theorem \ref{thm:solwave}.  Next, omitting the irrelevant factor 2,  we consider
\begin{align*}
  A _{1 1232}&= \<  \widetilde{V} , R _{h_c   }(-1 )  \zeta_{B} '   R _{h_c   }(-1 )  \eta _1 ''[c]     \> - \<  \widetilde{V} , R _{h_c   }(-1 )  \zeta_{B} 'R _{h_c   }(-1 )  q'_c R _{h_c   }(-1 ) \eta _1 ' [c]     \> \\&=  A _{1 12321}+ A _{1 12322}.
\end{align*}
We have
\begin{align*}
  |A _{1 12321}| & \lesssim \|   \zeta _A (\vartheta _{1A_1} + \vartheta _{2A_1})\widetilde{V} \| _{L^2}   \| \zeta _A^{-1} R _{h_c   }(-1 )\zeta _A \|_{L^2\to L^2} ^2 \|   \zeta_{B} '\| _{L^\infty} \|   \zeta _A^{-1}\eta _1 ''[c]\| _{L^2}\\& \lesssim B^{-1}(  \| V  \|_{  { \Sigma }_{1A A_1}} + \| V  \|_{  { \Sigma }_{2A A_1}}  )
\end{align*}
and similarly $|A _{1 12322}|\lesssim  B^{-1}(  \| V  \|_{  { \Sigma }_{1A A_1}} + \| V  \|_{  { \Sigma }_{2A A_1}}  )$. We have
\begin{align*}
   |A _{1 2}|&= \| \zeta _A R _{h_c   }(-1 )    \(  e^{V _\phi}-1-V _\phi \) \| _{L^2}  \| \zeta _A ^{-1}(\zeta_{B}  \eta _1 ' [c] + \zeta_{B} ' \eta _1  [c]) \| _{L^2} \\& \lesssim \| \zeta _A R _{h_c   }(-1 )\zeta _A^{-1} \|_{L^2\to L^2}  \| \zeta _A     \(  e^{V _\phi}-1-V _\phi \) \| _{L^2}\lesssim  \| \zeta _A  (\vartheta _{1A_1} + \vartheta _{2A_1})   V _\phi ^2 \| _{L^2}\\& \le  \|      V _\phi  \| _{L^\infty} (  \| V  \|_{  { \Sigma }_{1A A_1}} + \| V  \|_{  { \Sigma }_{2A A_1}}  )\lesssim \delta  B ^{1/2}(  \| V  \|_{  { \Sigma }_{1A A_1}} + \| V  \|_{  { \Sigma }_{2A A_1}}  ) ,
\end{align*}
 where  $\|      V _\phi  \| _{H^1} \lesssim \| R_{h_c}  (-1) \(  O ( V _\phi^2) +V_n\)   \|_{H^1} \lesssim B^{1/2}  \delta$  from \eqref{eq:h1errorbounds}.   Next we look at $ A_{2 }$. We have
\begin{align*}&
  \nabla_{\widetilde{U}} E_P(S_c+V)      -  \nabla_{\widetilde{U}} E_P(S_c ) - \nabla_{\widetilde{U},  U}^2 E_P(S_c )    V  \\&= \left(
                                                                                                                                   \begin{array}{c}
                                                                                                                                     \frac{(u _c+V_u)^2}{2}   +   K \( ( \log (1+n_c+V_n)  - n_c -V_n)   \) \\
                                                                                                                                     (n_c+V_n)(u _c+V_u) \\
                                                                                                                                   \end{array}
                                                                                                                                 \right)   - \left(
                                                                                                                                   \begin{array}{c}
                                                                                                                                     \frac{ u _c ^2}{2}   +   K \( ( \log (1+n_c )  - n_c    \) \\
                                                                                                                                      n_c  u _c  \\
                                                                                                                                   \end{array}
                                                                                                                                 \right)  \\& - \left(
             \begin{array}{ccc}
                                                                          K\( \frac{ 1}{1+n _c}-1 \) & u_c & 0 \\
                                                                          u_c &  n_c& 0  \\
                                                                        \end{array}
                                                                      \right)   \left(
                                                                                  \begin{array}{c}
                                                                                    V_n \\
                                                                                    V_u \\
                                                                                  \end{array}
                                                                                \right)  \\& =  \left(
                                                                                  \begin{array}{c}
                                                                                    \frac{ V_u ^2}{2} +K \( \log (1+n_c ) + \log \(1+\frac{V_n}{1+n_c} \)   - n_c -V_n -\( ( \log (1+n_c )  - n_c    \)   - \frac{ V_n}{1+n _c}  +V_n\) \\
                                                                                    V_nV_u   \\
                                                                                  \end{array}
                                                                                \right) \\& =  \left(
                                                                                  \begin{array}{c}
                                                                                    \frac{ V_u ^2}{2} - K \frac{V_n^2}{2(1+n_c ) ^2}     +K \(   \log \(1+\frac{V_n}{1+n_c} \)      - \frac{ V_n}{1+n _c}    + \frac{V_n^2}{2(1+n_c ) ^2}\) \\
                                                                                    V_nV_u   \\
                                                                                  \end{array}
                                                                                \right)  .
\end{align*}
Hence, using the definition of $A_2$ in \eqref{eq:seceq}, we have
\begin{align*}
  | A_{2 }|&\lesssim   \| \widetilde{V} \| _{L^\infty} (  \| V  \|_{  { \Sigma }_{1A A_1}} + \| V  \|_{  { \Sigma }_{2A A_1}}  )  \| \zeta _A ^{-1}\( \zeta_{B} \eta _1[c]\) ' \| _{L^2} \lesssim \delta  B^{1/2} (  \| V  \|_{  { \Sigma }_{1A A_1}} + \| V  \|_{  { \Sigma }_{2A A_1}}  ) .
\end{align*}
Finally  in \eqref{eq:seceq} for the coefficients of $\dot D -c $ and  $\dot  c $  we have
\begin{align*}&
    \< \xi _1[c] +  \widetilde{V}' , \zeta_{B} \eta _1[c] \>  =1+ O(B^{-1})  \text{   and }\\& \< \xi _2[c] , \zeta_{B} \eta _1[c] \>    -    \<   {\widetilde{V}}  , \zeta_{B} \partial _c\eta _1[c] \> =  O(B^{-1}) .
\end{align*}
From this we derive
\begin{align} \label{eq:discrest1D1}
   |\dot D -c |   & \lesssim B^{-1} |\dot c | +B ^{-1/2}(  \| V  \|_{  { \Sigma }_{1A A_1}} + \| V  \|_{  { \Sigma }_{2A A_1}}  ) .
\end{align}
When we apply to  the first equation  \eqref{EP2}  the inner product
 the inner product $\< \cdot ,   \eta _2[c] \>$  we obtain
\begin{align}&
    ( \dot D -c )   \<   \widetilde{V}  ,  \partial _x \eta _2[c] \>      + \dot c  \left [ 1   -    \<   {\widetilde{V}}  ,  \partial _c\eta _2[c] \> \right ] \label{eq:seceq1}\\& = \<   -  \partial _x \( \sigma _1  \nabla_{\widetilde{U},  U}^2 E(S_c )    V  -c\widetilde{V} \),     \eta _2[c] \>   \nonumber \\& + \<     \sigma _1 \( \nabla_{\widetilde{U}} E(S_c+V)      -  \nabla_{\widetilde{U}} E(S_c ) - \nabla_{\widetilde{U},  U}^2 E(S_c )    V \)   ,  \eta _2[c]  ' \>  =: \widetilde{A} _{1 }+\widetilde{A} _{2 }. \nonumber
\end{align}
Proceeding as above,  we have
\begin{align}
  \label{eq:analysc1} |\widetilde{A} _{j }|\lesssim    \| V  \|_{  { \Sigma }_{1A A_1}}^2 + \| V  \|_{  { \Sigma }_{2A A_1}} ^2 \text{   for $j=1,2$,}
\end{align}
 where the proof, (skipped) is now  simpler because
 all the terms with derivatives of $\zeta _B$ are absent and $\eta _2[c]\xrightarrow{x\to  \infty }    0$ exponentially fast, unlike $\eta _1[c]$. Then by \eqref{eq:h1errorbounds} we obtain the following which with \eqref{eq:discrest1D1} yields \eqref{eq:discrest1D} and \eqref{eq:discrest1c},
\begin{align}
  \label{eq:analysc2}
   |\dot  c |   &  \lesssim  |\dot D -c |^2 +  \|V \| _{ { \Sigma }_{1A A_1}  }^2 +
\|V \| _{ { \Sigma }_{2A A_1}  }^2  .
\end{align}

\section{Virial inequalities: proof of Lemma \ref{prop:1virial} }\label{sec:virial}
Set  $\mathcal{I} ^{(i)}:=  \<  \varphi _{iAA_1}    , e(S_c+ V) -e(S_c) \>$
       for $i=1,2$. We  consider three distinct virial inequalities.

\begin{lemma}[The first two virial inequalities]\label{lem:1stV1}
There exist $ \varepsilon _0>0$ and $C>0$ such that for all $\varepsilon \in ( 0, \varepsilon _0)$ we have   \small
\begin{align}     \label{eq:lem:1stV11}
\|   {V } \|_{   \Sigma  _{1A A_1} }^2   & \le C \varepsilon ^{-1}  \(  -   \frac{d}{dt} \mathcal{I} ^{(1)} +   \frac{d}{dt} \<   \varphi_{1A A_1} \nabla_{ \widetilde{{U}}} e(S_c ) ,   { \widetilde{V}}\>   +B ^{-1}
\|V \| _{   { \Sigma }_{2A A_1}  } ^2  + \varepsilon \|  {V }  \|_{  \widetilde{\Sigma}    }^2\)    \text{ and}
\\     \label{eq:lem:1stV12}
\|   {V } \|_{ { \Sigma }_{2A A_1} }^2   & \le C \varepsilon ^{-1}   \(  -   \frac{d}{dt} \mathcal{I}^{(2)} +    \frac{d}{dt} \<   \varphi_{2A A_1} \nabla_{ \widetilde{{U}}} e(S_c ) ,   { \widetilde{V}}\>     + A_1 ^{-2}
\|V \| _{   { \Sigma }_{1A A_1}  } ^2    \) .
\end{align} \normalsize
\end{lemma}

\proof    By   $ {V}\in C^0(I, H^1)$ and  \eqref{EP2} we have $ {V}\in C^1(I, L^2)$.  Then for almost any $t$
\begin{align*}
  \dot {\mathcal{I}}^{(i)} &= \<  \varphi _{iAA_1} \( \nabla_{ \widetilde{{U}}} E(S_c+V)      -  \nabla_{ \widetilde{{U}}} E(S_c ) \),  \dot { \widetilde{V}}\>  + \<   \varphi _{iAA_1} \nabla_{ \widetilde{{U}}} E(S_c )  ,  \dot { \widetilde{V}}\> \\& +\dot c  \<  \varphi _{iAA_1} \( \nabla_{ \widetilde{{U}}} E(S_c+V)      -  \nabla_{ \widetilde{{U}}} E(S_c ) \),   \partial _c \widetilde{S} _c\> =: \mathcal{I} _1^{(i)}+\mathcal{I}_2^{(i)}+\mathcal{I}_3^{(i)} 
\end{align*}
since  ${\mathcal{I}}^{(i)} $ is absolutely continuous,
where we used $\partial _{\phi}e(S_c+V)=\partial _{\phi}e(S_c )=0$. Notice   that
\begin{align*}
\mathcal{I}_2^{(i)}= \< \varphi _{iAA_1} \nabla_{ \widetilde{{U}}} E(S_c )  ,  \dot { \widetilde{V}}\>  =\frac{d}{dt} \<  \varphi _{iAA_1} \nabla_{ \widetilde{{U}}} E(S_c ) ,   { \widetilde{V}}\> - \dot c \<   \varphi _{iAA_1} ,\nabla_{ \widetilde{{U}}}^2 E(S_c ) , (\partial _c S_c,  { \widetilde{V}})\>
\end{align*}
with,  by Lemma \ref{lem:solwave4} and by \eqref{eq:discrest1c}, 
  \begin{align}\label{eq:Ii2}
  \left |  \dot c \<   \varphi _{iAA_1} ,\nabla_{ \widetilde{{U}}}^2 E(S_c ) , (\partial _c S_c,  { \widetilde{V}})\> \right | \lesssim  B ^{ \frac{1}{2}} \delta   \(\|V \| _{  { \Sigma }_{1A A_1} }^2 +
\|V \| _{   { \Sigma }_{2A A_1}  } ^2\).
\end{align}
Similarly,   we have
\begin{align}\nonumber
  |\mathcal{I}_3^{(i)}|&\le |\dot c| \int _{0}^{1} \left |\<   \varphi _{iAA_1} ,  \nabla_{ \widetilde{{U}} U}^2 E(S_c+sV)     ( \partial _c \widetilde{S} _c  , V)   \> \right |ds\\& \lesssim   B ^{ \frac{1}{2}} \delta   (\|V \| _{L^2(I,  { \Sigma }_{1A A_1} )}^2 +
\|V \| _{L^2(I,  { \Sigma }_{2A A_1} )} ^2). \label{eq:Ii3}
\end{align}
Using \eqref{EP2} we have \small
\begin{align*}
    \mathcal{I}_1^{(i)}  =& -\< \varphi _{iAA_1} \( \nabla_{ \widetilde{{U}}} E(S_c+V)      -  \nabla_{ \widetilde{{U}}} E(S_c ) \) ,\sigma _1\partial _x   \( \nabla_{ \widetilde{{U}}} E(S_c+V)      -  \nabla_{ \widetilde{{U}}} E(S_c ) \)  \> \\& + c\<  \varphi _{iAA_1} \( \nabla_{ {{U}}} E(S_c+V)      -  \nabla_{  {{U}}} E(S_c ) \) , {V}'\>  +(\dot D-c)  \<   \varphi _{iAA_1} \( \nabla_{ \widetilde{{U}}} E(S_c+V)      -  \nabla_{ \widetilde{{U}}} E(S_c ) \), \widetilde{V} '  \> \\& +(\dot D-c)  \<  \varphi _{iAA_1} \( \nabla_{ \widetilde{{U}}} E(S_c+V)      -  \nabla_{ \widetilde{{U}}}E(S_c ) \), \widetilde{S}_c '  \>  \\&-\dot c  \<  \varphi _{iAA_1} \( \nabla_{ \widetilde{{U}}} E(S_c+V)      -  \nabla_{ \widetilde{{U}}} E(S_c ) \), \partial _c\widetilde{S}_c    \>      =:\sum _{j=1,...,5} B_j^{(i)}  .
\end{align*}\normalsize
Splitting  $E(U)=E_K(U)+E_P(U)$  like in   \eqref{eq:enkinpot}
we have
\begin{align*}
  B_4^{(i)} &=(\dot D-c)   \<  \varphi _{iAA_1}   \nabla_{ \widetilde{{U}}} E_K( V)    , \widetilde{S}_c '  \> + (\dot D-c)   \<  \varphi _{iAA_1} \( \nabla_{ \widetilde{{U}}}E_P(S_c+V)      -  \nabla_{ \widetilde{{U}}} E_P(S_c ) \), \widetilde{S}_c '  \>\\& =: B _{4,1}^{(i)}+ B_{4,2}^{(i)} \text{ where} \\
  B _{4,1}^{(i)}&= (\dot D-c) \< \zeta _A (\vartheta _{1A_1} + \vartheta _{2A_1})  V , \left(
                          \begin{array}{ccc}
                            K & 0   \\
                            1 & 1   \\
                            0 & 0   \\
                          \end{array}
                        \right)  \zeta _A ^{-1} \varphi _{iAA_1} \widetilde{S}_c '  \> .
\end{align*}
By   \eqref{eq:discrest1D}, by $\supp \vartheta _{2A_1}\subseteq [A_1, +\infty )$, by the exponential decay of $\widetilde{S}_c '$  and    by {Theorem}
  \ref{thm:solwave},  for $k=1$
\small\begin{align}\nonumber
  | B _{4,k}^{(1)}|&\lesssim  \varepsilon ^2 B ^{-1/2}(\|V \| _{  { \Sigma }_{1A A_1}  }  +
\|V \| _{   { \Sigma }_{2A A_1}  }  )  \(  \|V \| _{  { \Sigma }_{1A A_1}  }+ A_1^{-100}
   \|V \| _{  { \Sigma }_{2A A_1}  } \)\\& \nonumber \lesssim   B ^{-1/2} \|V \| _{  { \Sigma }_{1A A_1}  }^2 + B ^{-1/2}A_1^{-100}  \|V \| _{  { \Sigma }_{2A A_1}  }^2  +\varepsilon ^2  B ^{- \frac{1}{2}}  \|V \| _{  { \Sigma }_{1A A_1}  }    \|V \| _{  { \Sigma }_{2A A_1}  }       \\& \lesssim     \varepsilon  ^4 \|V \| _{  { \Sigma }_{1A A_1}  }  ^2 +   B ^{-1} \|V \| _{  { \Sigma }_{2A A_1}  }  ^2  . \label{eq:B4k1}
\end{align}\normalsize
 Using \eqref{eq:discrest1D}, $\supp \varphi _{2A_1}\subseteq [A_1, +\infty )$ and the exponential decay of $\widetilde{S}_c '$ we have for $k=1$
\begin{align}\label{eq:B4k2}
  | B _{4,k}^{(2)}|\lesssim A_1^{-100}   B ^{-1/2}(\|V \| _{  { \Sigma }_{1A A_1}  } ^2 +
\|V \| _{   { \Sigma }_{2A A_1}  } ^2 )    .
\end{align}
By the following,  estimates  \eqref{eq:B4k1} for $i=1$ and \eqref{eq:B4k2}   for $i=2$ are straightforward also  for $k=2$,
\begin{align*}
    B _{4,2}^{(i)}= (\dot D-c) \int _0 ^1 \< \zeta _A (\vartheta _{1A_1} + \vartheta _{2A_1})  V ,  \left [ \nabla_{\widetilde{U},  U}^2 E _P(S_c+sV) \right ] ^\intercal \zeta _A ^{-1} \varphi _{iAA_1} \widetilde{S}_c '  \> ds.
\end{align*}
Notice that then we conclude  that for $\ell =4$
\begin{equation}\label{eq:B4i}
  \begin{aligned}
     &  | B _{\ell }^{(1)}|\lesssim     \varepsilon  ^4 \|V \| _{  { \Sigma }_{1A A_1}  }  ^2 +   B ^{-1} \|V \| _{  { \Sigma }_{2A A_1}  }  ^2 \text{ and} \\&     | B _{\ell }^{(2)}|\lesssim  A_1^{-100}   B ^{-1/2}(\|V \| _{  { \Sigma }_{1A A_1}  } ^2 +
\|V \| _{   { \Sigma }_{2A A_1}  } ^2 ).
  \end{aligned}
\end{equation}
 By a similar argument it is easy to see that \eqref{eq:B4i} is true also for $\ell =5$.
 Next, we have
\begin{align}\label{eq:intpa}
    \<  \varphi_{iA A_1} \( \nabla_{ {{U}}} E(S_c+V)      -  \nabla_{  {{U}}} E(S_c ) \) , {V}'\> &= \<  \varphi_{iA A_1}  \nabla_{ {{U}}} E_K( V)      , {V}'\>\\& +  \<  \varphi_{iA A_1} \( \nabla_{ {{U}}} E_P(S_c+V)      -  \nabla_{  {{U}}}E_P(S_c ) \) , {V}'\>.\nonumber
\end{align}
By elementary differentiation we compute
\begin{align}& \nonumber
   \partial _x \( e_{P}(S_c+V) - e_{P}(S_c )- \nabla _U e_{P}(S_c )\cdot V     \) \\& =  \nabla _UE_P(S_c+V) \cdot(S_c'+V')- \nabla _U E_{P}(S_c ) \cdot S_c' - \nabla _U E_{P}(S_c )\cdot V'- \nabla _U^2 E_{P}(S_c )(S'_c,V).\nonumber
\end{align}
Rearranging the above terms we have
\begin{align}   \nonumber
  \( \nabla_{ {{U}}} E_P(S_c+V)      -  \nabla_{  {{U}}} E_P(S_c ) \)&  {V}'  = \partial _x \( e_{P}(S_c+V) - e_{P}(S_c )- \nabla _U E_{P}(S_c )\cdot V     \) \\& - \left [ \( \nabla_{ {{U}}} E_P(S_c+V)-\nabla_{ {{U}}} E_P(S_c )   \)S'_c  - \nabla _U^2 E_{P}(S_c )(S'_c,V)\right ] \nonumber
\end{align}
so that, integrating by parts and by \eqref{def:phiAi}, we have
\begin{align} \nonumber
  B_3^{(i)}&= -(\dot D-c) \<   \vartheta _{iA_1} ^2 \zeta _A ^2,e_K(V) \>    \\&  -(\dot D-c)  \<  \vartheta _{iA_1} ^2  \zeta _A ^2,  e_{P}(S_c+V) - e_{P}(S_c )- \nabla _UE_{P}(S_c )\cdot V      \>     \label{eq:intparn2}\\& -(\dot D-c) \int _{[0,1]^2}\< \varphi_{iA A_1}  ,   \nabla _U^3 E_{P}(S_c+s \tau V )   (S'_c,V, V)  \>   sdsd\tau=:\sum _{k=1,2,3}   B_{3k}^{(i)} . \nonumber
\end{align}
It is easy to see that we have
\begin{align*}
  | B_{31}^{(i)}|  +  | B_{32}^{(i)}| \lesssim \delta \|V \| _{   { \Sigma }_{iA A_1}  } ^2 \text{ and }  | B_{33}^{(i)}|\lesssim \delta \|V \| _{     \widetilde{\Sigma}     } ^2.
\end{align*}
Proceeding like for $B_3^{(i)}$ we obtain
\begin{align*}
  B_2^{(i)}&= - c\<   \vartheta _{iA_1} ^2\zeta _A ^2,e_K(V) \>  -c  \<     \vartheta _{iA_1} ^2\zeta _A ^2,  e_{P}(S_c+V) - e_{P}(S_c )- \nabla _U E_{P}(S_c )\cdot V      \> \\& -c \int _{[0,1]^2}\<  \varphi_{iA A_1}  ,   \nabla _U^3 E_{P}(S_c+s \tau V )   (S'_c,V, V)  \>   sdsd\tau =:\sum _{k=1,2,3} B _{2k}^{(i)}   .
\end{align*}
By \eqref{eq:enkinpot} we have
\begin{align}\label{eq:B21ifor}
   B_{21}^{(i)}  &=- \frac{c}{2} \<   \vartheta _{iA_1} ^2\zeta _A  ^2 ,      K V_n ^2+V_u^2\> + c  \<   \vartheta _{iA_1} ^2\zeta _A  ^2 ,    \frac{  V_\phi  ^{\prime 2} + V_\phi  ^{  2}}{2}-V_nV_\phi \>   \\& =  - \frac{c}{2} \<  \vartheta _{iA_1} ^2 \zeta _A  ^2 ,      K V_n ^2+V_u^2+V_\phi  ^{\prime 2} + V_\phi  ^{  2}\> + c  \<   \vartheta _{iA_1} ^2\zeta _A  ^2 ,    V_\phi  ^{\prime 2} + V_\phi  ^{  2} -V_nV_\phi \>
    =: B_{211}^{(i)}+B_{212}^{(i)}.\nonumber 
\end{align}
We will deal later with $B_{211}^{(i)}$ and focus on  $B_{212}^{(i)}$ for the moment.
By \eqref{EP2} we have
  \begin{align} \nonumber  &
    \<  \vartheta _{iA_1} ^2\zeta _A  ^2 ,      V_\phi  ^{\prime 2} + V_\phi  ^{  2} -V_nV_\phi \>  =   \<  \vartheta _{iA_1} ^2\zeta _A  ^2  V_\phi ,     - V_\phi  '' -V_n + V_\phi    \> +  2^{-1}\<  (\vartheta _{iA_1} ^2\zeta _A  ^2) ''    ,      V_\phi ^2   \> \\&   \nonumber  =\< \vartheta _{iA_1} ^2 \zeta _A  ^2  V_\phi ,  V_\phi   -  e^{\phi _c} \(  e^{V _\phi}-1 \)      \> +  2^{-1}\<  (\vartheta _{iA_1} ^2\zeta _A  ^2) ''    ,      V_\phi ^2   \> \\& =\<  \vartheta _{iA_1} ^2\zeta _A  ^2    ,  V_\phi ^2   \( 1 -  e^{\phi _c}\)   + O( V_\phi ^3)     \> +  2^{-1}\<  (\vartheta _{iA_1} ^2\zeta _A  ^2) ''    ,      V_\phi ^2   \>
 \label{eq:B212i1} \end{align}
  where from $|\zeta _A ^{(k)}|\lesssim A ^{-k}\zeta _A$ we obtain
 \begin{align*}
 | \<  (\vartheta _{iA_1} ^2\zeta _A  ^2) ''    ,      V_\phi ^2   \> |\lesssim    \<  |(\vartheta _{iA_1} ^2\zeta _A  ^2) '' |   ,      (\vartheta _{1A_1}   + \vartheta _{2A_1}  ) ^2 V_\phi ^2   \> | \lesssim A_1 ^{-2}(\|V \| _{  { \Sigma }_{1A A_1}  }^2  +
\|V \| _{   { \Sigma }_{2A A_1}  } ^2 )
 \end{align*}
  and similarly
  \begin{align*}
  | \<  \vartheta _{iA_1} ^2\zeta _A  ^2    ,   O( V_\phi ^3)     \>|\lesssim B^{1/2}\delta (\|V \| _{  { \Sigma }_{1A A_1}  }^2  +
\|V \| _{   { \Sigma }_{2A A_1}  } ^2 )  .
  \end{align*}
Hence, by $\supp \vartheta _{2A_1}\subseteq [A_1, +\infty )$ and bounding also the first term in \eqref{eq:B212i1},  for $i=1,2$ we conclude
\begin{align*}
  |B_{212}^{(i)}|\lesssim  A_1 ^{-100 (i-1)}  \|  {V }_\phi \|_{  \widetilde{\Sigma}   }^2+      A_1 ^{-2}(\|V \| _{  { \Sigma }_{1A A_1}  }^2  +
\|V \| _{   { \Sigma }_{2A A_1}  } ^2 )
\end{align*}
which yields also
\begin{align*}
  |B_{212}^{(2)}|\lesssim         A_1 ^{-2}(\|V \| _{  { \Sigma }_{1A A_1}  }^2  +
\|V \| _{   { \Sigma }_{2A A_1}  } ^2 ) .
\end{align*}
 Leaving for later  $B_{211}^{(i)}$,     we consider  \small
\begin{align*}
&  B_1^{(i)} = 2^{-1}\< \vartheta _{iA_1} ^2 \zeta _A  ^2   \nabla_{ \widetilde{{U}}} E_K( V)       ,\sigma _1   \nabla_{ \widetilde{{U}}} E_K( V)  \>  \\& +
  \<  \vartheta _{iA_1} ^2\zeta _A  ^2   \nabla_{ \widetilde{{U}}} E_K( V)       ,\sigma _1   \( \nabla_{ \widetilde{{U}}} E_P(S_c+V)      -  \nabla_{ \widetilde{{U}}} E_P(S_c ) \)  \>
  \\& +
  2^{-1}\<  \vartheta _{iA_1} ^2\zeta _A  ^2 \( \nabla_{ \widetilde{{U}}} E_P(S_c+V)      -  \nabla_{ \widetilde{{U}}} E_P(S_c ) \) ,\sigma _1    \( \nabla_{ \widetilde{{U}}} E_P(S_c+V)      -  \nabla_{ \widetilde{{U}}} E_P(S_c ) \)  \> =:B_{11}^{(i)}+B_{12}^{(i)}+B_{13}^{(i)}.
\end{align*} \normalsize
Using \eqref  {eq:h1errorbounds}  and   $\supp \vartheta _{2A_1}\subseteq [A_1, +\infty )$ it is easy to see that
\begin{align}\nonumber
  | B_{13}^{(2)}|+ | B_{12}^{(2)}|&\lesssim  A_1 ^{-100  }  \|  {V } \|_{  \widetilde{\Sigma}   }^2+     B ^{1/2}\delta (\|V \| _{  { \Sigma }_{1A A_1}  } ^2 +
\|V \| _{   { \Sigma }_{2A A_1}  } ^2 )   \\& \lesssim   A_1 ^{-2} (\|V \| _{  { \Sigma }_{1A A_1}  } ^2 +
\|V \| _{   { \Sigma }_{2A A_1}  } ^2 )  . \label{eq:B13212}
\end{align}
Schematically we have
\begin{align*}
B_{13}^{(1)}  &=   \<  \vartheta _{1A_1} ^2\zeta _A  ^2 ,O\( S_c ^2  V^2  \)  + O\(   V^3  \)   \>
\end{align*}
so that by Theorem \ref{thm:solwave}
\begin{align}\label{eq:B131}
|B_{13}^{(1)} | &\lesssim     \| S_c \| _{  L^\infty (\R )  } ^2     \|V \| _{  { \Sigma }_{1A A_1}  } ^2  \lesssim   \varepsilon ^2  \|V \| _{  { \Sigma }_{1A A_1}  } ^2 .
\end{align}
Next we consider
\begin{align*}
  B_{12}^{(1)} &= \< \vartheta _{1A_1} ^2\zeta _A  ^2  \sigma _1 \nabla_{ \widetilde{{U}}} E_K( V)       ,     \nabla_{ \widetilde{{U}} U}^2 E_P(S_c )  \cdot V  \>  \\& + \< \vartheta _{1A_1} ^2\zeta _A  ^2  \sigma _1  \nabla_{ \widetilde{{U}}}  E_K( V)       ,    \nabla_{ \widetilde{{U}}} E_P(S_c+V)      -  \nabla_{ \widetilde{{U}}} E_P(S_c )  -   \nabla_{ \widetilde{{U}} U}^2 E_P(S_c )  \cdot V \> = B_{121}^{(1)}+B_{122} ^{(1)}.
\end{align*}
Then   
\begin{align*}
  B_{122}^{(1)} &=  \int  _{[0,1] ^2}   \< \vartheta _{1A_1} ^2\zeta _A  ^2  \sigma _1  \nabla_{ \widetilde{{U}}} E_K( V)       ,   \nabla ^3_{ \widetilde{{U}}UU} E_P(S_c+ \tau sV) \cdot ( V,      V) \>   s ds  d\tau
\end{align*}
implies by  \eqref  {eq:h1errorbounds}    the following,
\begin{align*}
  | B_{122}^{(1)} |\lesssim  \|   {V } \|_{L^\infty }  \|   {V } \|_{{ \Sigma }_{1A A_1}}^2  \lesssim  B ^{1/2}\delta   \|   {V } \|_{{ \Sigma }_{1A A_1}}^2.
\end{align*}
By \eqref{eq:graden1} and  \eqref{eq:graden3} \begin{align}\nonumber
  B_{121}^{(1)}&=\<  \vartheta _{1A_1} ^2\zeta _A  ^2 \left(
                              \begin{array}{c}
                                 V _{u} \\
                                KV_n +V _{\phi} \\
                              \end{array}
                            \right) , \left(
                                                                        \begin{array}{ccc}
                                                                          K \(  \frac{1}{1+n _c} -1\) & u_c & 0 \\
                                                                          u _c&  n_c& 0  \\
                                                                        \end{array}
                                                                      \right)   V
    \> \\& \nonumber= \< \vartheta _{1A_1} ^2 \zeta _A  ^2,K  u_c V_n^2+ u_c V_u^2      + K  \frac{n_c ^2}{1+n_c}   V_u V_n\> +  \< \vartheta _{1A_1} ^2 \zeta _A  ^2,K   u_c V_n V_\phi + n_c V_u V_\phi     \> \\& =:   B_{1211}^{(1)}+ B_{1212}^{(1)}.\label{eq:B1211for}
\end{align}
Then by Theorem \ref{thm:solwave}   we have
\begin{align}\label{eq:B12121}
|B_{1212}^{(1)}  | &\lesssim    \varepsilon   \|  {V }   \|_{  \widetilde{\Sigma}    }^2.
\end{align}
Using  Theorem \ref{thm:solwave} we have
\begin{align}& \nonumber
 B_{1211}^{(1)} =  B_{12111}^{(1)}+ B_{12112}^{(1)}  \text{   where} \\  & B_{12111}^{(1)}= \varepsilon \sqrt{K+1} \< \vartheta _{1A_1} ^2 \zeta _A  ^2  \psi_{\mathrm{KdV}}\left( \sqrt{\varepsilon} \cdot  \right)   ,K  V_n^2+ V_u^2     \>   \text{   and } \label{eq:B121111} \\  & |B_{12112}^{(1)}|  \lesssim \varepsilon ^2  \|  {V }   \|_{  \widetilde{\Sigma}    }^2. \label{eq:B121112}
\end{align}
   By \eqref{eq:graden1} and an elementary application of Young's inequality we have \begin{align}\label{eq:B11i}
 B_{11}^{(i)}&= \< \vartheta _{iA_1} ^2\zeta _A  ^2 , KV_nV_u+ V_\phi V_u\>  \le 2^{-1}\sqrt{K+1} \< \vartheta _{iA_1} ^2\zeta _A  ^2 , K    V_n ^2 +   V_u ^2  +    V_\phi ^2    \>.
\end{align}
Then by $c=\sqrt{K+1}+\varepsilon$ we have
\begin{align*}
  B_{211} ^{(i)}+B_{11}^{(i)}\le  -2^{-1}\varepsilon   \< \vartheta _{iA_1} ^2 \zeta _A  ^2 ,      K V_n ^2+V_u^2+V_\phi  ^{\prime 2} + V_\phi  ^{  2}\>  .
\end{align*}
So we conclude  the following   for $i=1,2$, completing  the proof of  Lemma  \ref {lem:1stV1}, 
\begin{align*}
  \mathcal{I}_1  ^{(i)}\lesssim & -\varepsilon   \<  \vartheta _{iA_1} ^2\zeta _A  ^2 ,      K V_n ^2+V_u^2+V_\phi  ^{\prime 2} + V_\phi  ^{  2}\> \\& +  A_1 ^{-2}(\|V \| _{  { \Sigma }_{1A A_1}  }^2  +
\|V \| _{   { \Sigma }_{2A A_1}  } ^2 ) + (2-i)  \|  {V } \|_{  \widetilde{\Sigma}    }^2  .
\end{align*}

\qed

We introduce  $\mathcal{J}  :=  \<  \psi     , e(S_c+ V) -e(S_c) \>$
        with  for a sufficiently large constant $C_0\gg 1$ 
        \begin{align}\label{eq:defpsi}
         \psi (x)=   \int _0 ^x       \sech ^2\(     \varepsilon \kappa  x ' \)   dx'.
        \end{align}
\begin{lemma}[Third virial inequality]\label{lem:2stV1} There exist $ \varepsilon _0>0$ and $C>0$ such that for all $\varepsilon \in ( 0, \varepsilon _0)$ we have  \small
\begin{align}     \label{eq:lem:2stV11}
 \|         ( V_n , V_u, V_\phi  ^{\prime  })   \| _{  \widetilde{\Sigma}    }^2   & \le C   \(  -   \frac{d}{dt} \mathcal{J}  +   \frac{d}{dt} \<   \psi  \nabla_{ \widetilde{{U}}} e(S_c ) ,   { \widetilde{V}}\>    +B^{-1 }  \|V \| _{   { \Sigma }_{1A A_1}  } ^2   +B^{-1 }  \|V \| _{   { \Sigma }_{2A A_1}  } ^2 +  \|  {V }_\phi   \|_{  \widetilde{\Sigma}    }^2\)    .
\end{align} \normalsize
\end{lemma}

\proof  Like in Lemma \ref {lem:1stV1}
\begin{align*}
  \dot {\mathcal{J}}  &= \< \psi  \( \nabla_{ \widetilde{{U}}} E(S_c+V)      -  \nabla_{ \widetilde{{U}}} E(S_c ) \),  \dot { \widetilde{V}}\>  + \<   \psi  \nabla_{ \widetilde{{U}}} E(S_c )  ,  \dot { \widetilde{V}}\> \\& +\dot c  \<  \psi\( \nabla_{ \widetilde{{U}}} E(S_c+V)      -  \nabla_{ \widetilde{{U}}} E(S_c ) \),   \partial _c \widetilde{S} _c\> =: \mathcal{J} _1 +\mathcal{J}_2 +\mathcal{J}_3 .
\end{align*}
Then 
\begin{align*}
\mathcal{J} _2= \< \psi  \nabla_{ \widetilde{{U}}} E(S_c )  ,  \dot { \widetilde{V}}\>  =\frac{d}{dt} \<  \psi \nabla_{ \widetilde{{U}}} E(S_c ) ,   { \widetilde{V}}\> - \dot c \<  \psi  ,\nabla_{ \widetilde{{U}}}^2 E(S_c ) , (\partial _c S_c,  { \widetilde{V}})\>
\end{align*}
with, like in \eqref{eq:Ii2},    
\begin{align*}
  \left |  \dot c \<   \psi ,\nabla_{ \widetilde{{U}}}^2 E(S_c ) , (\partial _c S_c,  { \widetilde{V}})\> \right | \lesssim  B ^{ \frac{1}{2}} \delta   \(\|V \| _{  { \Sigma }_{1A A_1} }^2 +
\|V \| _{   { \Sigma }_{2A A_1}  } ^2\).
\end{align*}
Like in  \eqref{eq:Ii3},
\begin{align*}
  \left | \mathcal{J}_3  \right | \lesssim  B ^{ \frac{1}{2}} \delta   \(\|V \| _{  { \Sigma }_{1A A_1} }^2 +
\|V \| _{   { \Sigma }_{2A A_1}  } ^2\).
\end{align*}
 Like in Lemma  \ref{lem:1stV1} \small
\begin{align*}
    \mathcal{J}_1   =& -\< \psi  \( \nabla_{ \widetilde{{U}}} E(S_c+V)      -  \nabla_{ \widetilde{{U}}} E(S_c ) \) ,\sigma _1\partial _x   \( \nabla_{ \widetilde{{U}}} E(S_c+V)      -  \nabla_{ \widetilde{{U}}} E(S_c ) \)  \> \\& + c\<  \psi  \( \nabla_{ {{U}}} E(S_c+V)      -  \nabla_{  {{U}}} E(S_c ) \) , {V}'\>  +(\dot D-c)  \<  \psi \( \nabla_{ \widetilde{{U}}} E(S_c+V)      -  \nabla_{ \widetilde{{U}}} E(S_c ) \), \widetilde{V} '  \> \\& +(\dot D-c)  \<  \psi \( \nabla_{ \widetilde{{U}}} E(S_c+V)      -  \nabla_{ \widetilde{{U}}}E(S_c ) \), \widetilde{S}_c '  \>  \\&-\dot c  \<  \psi  \( \nabla_{ \widetilde{{U}}} E(S_c+V)      -  \nabla_{ \widetilde{{U}}} E(S_c ) \), \partial _c\widetilde{S}_c    \>      =:\sum _{j=1,...,5} F_j   .
\end{align*}\normalsize
Splitting  $E(U)=E_K(U)+E_P(U)$  
we have
\begin{align*}
  F_4  &=(\dot D-c)   \< \psi    \nabla_{ \widetilde{{U}}} E_K( V)    , \widetilde{S}_c '  \> + (\dot D-c)   \<  \psi  \( \nabla_{ \widetilde{{U}}}E_P(S_c+V)      -  \nabla_{ \widetilde{{U}}} E_P(S_c ) \), \widetilde{S}_c '  \>\\& =: F _{4,1} + F_{4,2}  \text{ where} \\
  F _{4,1} &= (\dot D-c) \< \zeta _A (\vartheta _{1A_1} + \vartheta _{2A_1})  V , \left(
                          \begin{array}{ccc}
                            K & 0   \\
                            1 & 1   \\
                            0 & 0   \\
                          \end{array}
                        \right)  \zeta _A ^{-1}\psi \widetilde{S}_c '  \>
\end{align*}
and   by a variation of  \eqref{eq:B4k1}  we have for $k=1,2$
\small\begin{align}\nonumber
  | F _{4,k} |&\lesssim       \varepsilon  ^4 \|V \| _{  \widetilde{{ \Sigma }}   }  ^2 +     B ^{-1} \(  \|V \| _{  { \Sigma }_{1A A_1}  }  ^2 +  \|V \| _{  { \Sigma }_{2A A_1}  }  ^2 \)    .
\end{align}
It is similarly  easy to see that   in analogy to \eqref{eq:B4i} we have
\begin{equation}\nonumber
  \begin{aligned}
     &  | F _{\ell } |\lesssim       \varepsilon  ^4 \|V \| _{  \widetilde{{ \Sigma }}   }  ^2 +     B ^{-1} \(  \|V \| _{  { \Sigma }_{1A A_1}  }  ^2 +  \|V \| _{  { \Sigma }_{2A A_1}  }  ^2 \)  \text{  for $\ell =4,5$}.
  \end{aligned}
\end{equation}
In analogy to  \eqref{eq:intpa} we write
\begin{align}  \nonumber
    \<  \psi  \( \nabla_{ {{U}}} E(S_c+V)      -  \nabla_{  {{U}}} E(S_c ) \) , {V}'\> &= \< \psi   \nabla_{ {{U}}} E_K( V)      , {V}'\>\\& +  \<  \psi  \( \nabla_{ {{U}}} E_P(S_c+V)      -  \nabla_{  {{U}}}E_P(S_c ) \) , {V}'\> \nonumber
\end{align}
and integrating by parts like in  \eqref{eq:intparn2} we obtain 
 \begin{align*} \nonumber
 F _3 &= -(\dot D-c) \<    \sech ^2\(     \varepsilon \kappa  \cdot  \) ,e_K(V) \>    \\&  -(\dot D-c)  \<   \sech ^2\(     \varepsilon \kappa  \cdot  \)  ,  e_{P}(S_c+V) - e_{P}(S_c )- \nabla _UE_{P}(S_c )\cdot V      \>    \nonumber\\& -(\dot D-c) \int _{[0,1]^2}\< \psi  ,   \nabla _U^3 E_{P}(S_c+s \tau V )   (S'_c,V, V)  \>   sdsd\tau    =:\sum _{k=1 } ^3  F_{3k}  . \nonumber
\end{align*}
It is easy to see that in all cases we have
\begin{align*}
  |  F_{3k}|    \lesssim \delta \|V \| _{   { \Sigma }_{1A A_1}  } ^2+ \delta \|V \| _{   { \Sigma }_{2A A_1}  } ^2+ \delta \|V \| _{     \widetilde{\Sigma}     } ^2.
\end{align*}
Proceeding like for $F_3 $ we obtain
\begin{align*} \nonumber
 F _2 &= -c\<    \sech ^2\(     \varepsilon \kappa  \cdot  \) ,e_K(V) \>    \\&  -c  \<  \sech ^2\(     \varepsilon \kappa  \cdot  \)  ,  e_{P}(S_c+V) - e_{P}(S_c )- \nabla _UE_{P}(S_c )\cdot V      \>    \nonumber\\& -c \int _{[0,1]^2}\< \psi  ,   \nabla _U^3 E_{P}(S_c+s \tau V )   (S'_c,V, V)  \>   sdsd\tau    =:\sum _{k=1 } ^3  F_{2k}  . \nonumber
\end{align*}
Like in \eqref{eq:B21ifor}
 we have
\begin{align*}
   F_{21}   &=  - \frac{c}{2} \<   \sech ^2\(     \varepsilon \kappa  \cdot  \),      K V_n ^2+V_u^2+V_\phi  ^{\prime 2} + V_\phi  ^{  2}\> + c  \<  \sech ^2\(     \varepsilon \kappa  \cdot  \) ,    V_\phi  ^{\prime 2} + V_\phi  ^{  2} -V_nV_\phi \>
    \\& =: F_{211} +F_{212} 
\end{align*}
 where, like in \eqref{eq:B212i1}, 
  \begin{align*} \nonumber  &
   F_{212} =     c\<  \sech ^2\(     \varepsilon \kappa  \cdot  \)   ,  V_\phi ^2   \( 1 -  e^{\phi _c}\)   + O( V_\phi ^3)     \> +  2^{-1}c\<  ( \sech ^2\(     \varepsilon \kappa  \cdot  \)) ''    ,      V_\phi ^2   \> . \end{align*}
  From this and    {Theorem} \ref{thm:solwave} we conclude 
  \begin{align*}
  |F_{212} |\lesssim  \varepsilon  \|  {V }_\phi   \|_{  \widetilde{\Sigma}   }^2 . 
\end{align*}
 Leaving for later  $F_{211} $  like in Lemma  \ref{lem:1stV1}  we consider  \small
\begin{align*}
&  F_1  = 2^{-1}\<  \sech ^2\(     \varepsilon \kappa  \cdot  \)  \nabla_{ \widetilde{{U}}} E_K( V)       ,\sigma _1   \nabla_{ \widetilde{{U}}} E_K( V)  \>  \\& +
  \<   \sech ^2\(     \varepsilon \kappa  \cdot  \)   \nabla_{ \widetilde{{U}}} E_K( V)       ,\sigma _1   \( \nabla_{ \widetilde{{U}}} E_P(S_c+V)      -  \nabla_{ \widetilde{{U}}} E_P(S_c ) \)  \>
  \\& +
  2^{-1}\<   \sech ^2\(     \varepsilon \kappa  \cdot  \)2 \( \nabla_{ \widetilde{{U}}} E_P(S_c+V)      -  \nabla_{ \widetilde{{U}}} E_P(S_c ) \) ,\sigma _1    \( \nabla_{ \widetilde{{U}}} E_P(S_c+V)      -  \nabla_{ \widetilde{{U}}} E_P(S_c ) \)  \>  \\&  
  =:\sum _{k=1}^{3} F_{1k} .
\end{align*} \normalsize
Like  in \eqref{eq:B131}  we have
\begin{align}\nonumber
|F_{13}  | &\lesssim        \varepsilon ^2  \|V \| _{  \widetilde{\Sigma }   } ^2 .
\end{align}
Like for $ B_{12}^{(1)}$  in Lemma \ref{lem:1stV1},  we have
\begin{align*}
  F_{12}  &= F_{121} +F_{122}  \text{ with } F_{121}=
  \< \sech ^2\(     \varepsilon \kappa  \cdot  \)     \sigma _1 \nabla_{ \widetilde{{U}}} E_K( V)       ,     \nabla_{ \widetilde{{U}} U}^2 E_P(S_c )  \cdot V  \>   \text{ and}\\
  F_{122}  &=  \int  _{[0,1] ^2}   \< \sech ^2\(     \varepsilon \kappa  \cdot  \)2\sigma _1  \nabla_{ \widetilde{{U}}} E_K( V)       ,   \nabla ^3_{ \widetilde{{U}}UU} E_P(S_c+ \tau sV) \cdot ( V,      V) \>   s ds  d\tau
\end{align*}
which implies     the following,
\begin{align*}
  |  F_{122}  |\lesssim   B ^{1/2}\delta   \|   {V } \|_{ \widetilde{ \Sigma}   }^2.
\end{align*}
Like in \eqref{eq:B1211for} we have  
\begin{align*}
  F_{121} &  = \< \sech ^2\(     \varepsilon \kappa  \cdot  \) ,K  u_c V_n^2+ u_c V_u^2      + K  \frac{n_c ^2}{1+n_c}   V_u V_n\> +  \< \sech ^2\(     \varepsilon \kappa  \cdot  \) ,K   u_c V_n V_\phi + n_c V_u V_\phi     \> \\& =:   F_{1211} + F_{1212}
\end{align*}
with like in \eqref{eq:B12121},  \eqref{eq:B121111} and  \eqref{eq:B121112}  
\begin{align}\nonumber &
|F_{1212}   | \lesssim    \varepsilon   \|  {V }   \|_{  \widetilde{\Sigma}    }^2\\& \nonumber
 F_{1211}  =  F_{12111}  + F_{12112}   \text{   where} \\  & F_{12111} = \varepsilon \sqrt{K+1} \< \sech ^2\(     \varepsilon \kappa  \cdot  \)2 \psi_{\mathrm{KdV}}\left( \sqrt{\varepsilon} \cdot  \right)   ,K  V_n^2+ V_u^2     \>   \text{   and } \label{eq:F121111} \\  & |F_{12112} |  \lesssim \varepsilon ^2  \|  {V }   \|_{  \widetilde{\Sigma}    }^2.  \nonumber
\end{align}
So far the proof has been similar  to Lemma \ref{lem:1stV1}, but for different and less subtle upper bounds. Now  we  write
\begin{align*}
 F_{11} &= K\< \sech ^2\(     \varepsilon \kappa  \cdot  \)  ,  V_nV_u+ V_\phi V_u\>   +   \< \sech ^2\(     \varepsilon \kappa  \cdot  \)  ,  V_\phi V_u\>     =: F_{111}+F_{112} .
\end{align*}
modifying  the analysis in  \eqref{eq:B11i}. Then, for a small constant $ \varepsilon _1>0$ we write 
 \begin{align*}
   |F_{112}|\lesssim  \varepsilon _1 \|  \widetilde{{V }}   \|_{  \widetilde{\Sigma}    }^2 +  \varepsilon _1 ^{-1}\|  V _\phi   \|_{  \widetilde{\Sigma}    }^2 .
 \end{align*}
We bound 
\begin{align*}
   F_{111} \le 2^{-1} \< \sech ^2\(     \varepsilon \kappa  \cdot  \)  , K ^{3/2} V_n ^2 + K ^{1/2}  V_u \> .
\end{align*}
Then  
\begin{align}\label{eq:key1}
  F_{211}  +F_{111} \le  -2^{-1}   \frac{1+\varepsilon ^2+2\sqrt{1+K}\varepsilon}{ \sqrt{1+K} + \sqrt{ K} + \varepsilon }      \< \sech ^2\(     \varepsilon \kappa  \cdot  \)   ,      K V_n ^2+V_u^2+V_\phi  ^{\prime 2} + V_\phi  ^{  2}\> .
\end{align}
So we conclude  the following, completing  the proof of  Lemma  \ref {lem:2stV1}, 
\begin{align*}
  \mathcal{J}_1   \lesssim & -    \|         ( V_n , V_u, V_\phi  ^{\prime  })   \| _{  \widetilde{\Sigma}    }^2   +  \varepsilon _1 \|  \widetilde{{V }}   \|_{  \widetilde{\Sigma}    }^2 +  \varepsilon _1 ^{-1}\|  V _\phi   \|_{  \widetilde{\Sigma}    }^2   .
\end{align*}
Collecting the above estimates, we obtain 
\begin{align*}
 \|         ( V_n , V_u, V_\phi  ^{\prime  })   \| _{  \widetilde{\Sigma}    }^2   & \le C   \(  -   \frac{d}{dt} \mathcal{J}  +   \frac{d}{dt} \<   \psi  \nabla_{ \widetilde{{U}}} e(S_c ) ,   { \widetilde{V}}\>    +B^{-1 }  \|V \| _{   { \Sigma }_{1A A_1}  } ^2 \right .\\ & \left .   +B^{-1 }  \|V \| _{   { \Sigma }_{2A A_1}  } ^2 +\varepsilon _1 \|         ( V_n , V_u, V_\phi  ^{\prime  })   \| _{  \widetilde{\Sigma}    }^2   +  \|  {V }_\phi   \|_{  \widetilde{\Sigma}    }^2\)        
\end{align*}
which bootstrapping yields  \eqref{eq:lem:2stV11}.

\qed 

\textit{Proof of Lemma \ref {prop:1virial}.} Using $ \| \varphi _{iAA_1} \| _{L^\infty(\R )} \lesssim A $      by   \eqref  {eq:h1errorbounds} we obtain
\begin{equation*}
  \| \< \varphi _{iAA_1}    , e(S_c+ V) -e(S_c) \>  \|_{ L^\infty( I   )}  + \|  \<   \varphi _{iAA_1}\nabla_{ \widetilde{{U}}} e(S_c ) ,   { \widetilde{V}}\>  \|_{ L^\infty( I   )}    \lesssim AB^{1/2}\delta.
\end{equation*}
Then,  integrating in    time \eqref{eq:lem:1stV11} and  \eqref{eq:lem:1stV12}
and bootstrapping we obtain
\begin{align}\label{eq:key2}
   \|  V  \|_{ L^2( I ,  \Sigma  _{1A A_1}) }  ^2&\lesssim  AB^{1/2}\delta +
   B ^{-1}
\|V \| _{  L^2( I ,    \Sigma  _{2A A_1} ) } ^2  + \|   V    \|_{   L^2( I ,\widetilde{\Sigma})    }^2 \text{ and}
\\
   \|  V  \|_{ L^2( I ,{ \Sigma }_{2A A_1})} ^2&\lesssim AB^{1/2}\delta +
   A_1 ^{-2}
\|V \| _{  L^2( I ,  { \Sigma }_{1A A_1} ) } ^2  .\nonumber
\end{align}
Next, we use 
\begin{equation*}
  \| \< \psi     , e(S_c+ V) -e(S_c) \>  \|_{ L^\infty( I   )}  + \|  \<   \psi \nabla_{ \widetilde{{U}}} e(S_c ) ,   { \widetilde{V}}\>  \|_{ L^\infty( I   )}    \lesssim   \varepsilon ^ {-1}    B^{1/2}\delta.
\end{equation*}
 The proof of  Lemma \ref {prop:1virial}  is completed inserting in \eqref{eq:key2} the following consequence of    \eqref{eq:lem:2stV11},
\begin{align*}
  \|      ( V_n , V_u, V_\phi  ^{\prime  })   \|_{   L^2( I ,\widetilde{\Sigma})    }^2  \lesssim \varepsilon ^ {-1}    B^{1/2}\delta   +  B ^{-1}
\|V \| _{  L^2( I ,    \Sigma  _{1A A_1} ) } ^2 +  B ^{-1}
\|V \| _{  L^2( I ,    \Sigma  _{2A A_1} ) } ^2 +  \|     V_\phi    \|_{   L^2( I ,\widetilde{\Sigma})    }^2  ,
\end{align*}
which yields   \eqref{eq:sec:1virial11}  by bootstrap.

\qed

\begin{remark} \label{rem:Kpositive}Here $K>0$ is   crucial.
Firstly, it makes the equation dispersive.
Indeed, linearizing \eqref{EP} around $0$, we have $\ddot{n}=-Kn''+n-(-\partial_x^2+1)^{-1}n$,  which by $K>0$ resembles a Klein-Gordon equation.
Secondarily,  the virial functionals  are akin $I=\<\varphi,e(V)\>$. Positivity is produced
  because $e(V)\sim \frac{1}{2}(u^2+n^2)$ and, when computing $\dot{I}$, the contribution of the $cV'$ term is like $c\<\varphi,e'(V)V'\>=c\<\varphi,\partial_x(e(V))\>=-c\<\varphi',e(V)\>\sim -c(u^2+Kn^2)$.
With  $K=0$ we  lose control  of  $n$. 

\end{remark}

\section{Dispersion and Kato smoothing  for the linearized equation }\label{sec:disp}

We consider the operators
\begin{align}\label{eq:projections}
 P_{[c]}  := \sum _{i=1,2}   \xi _i [c]    \< \eta _i [c],  \cdot \>
 \text{ and } Q_{[c]}=1- P_{[c]}  
\end{align}
  associated to the spectral decomposition  
\begin{align}
  \label{eq:specdec} L^2 _a = N_g\( \mathcal{L}_c \) \oplus N_g^\perp \( \mathcal{L}_c^* \)
\end{align}
where   perpendicularity    with $ N_g  \( \mathcal{L}_c ^*  \) \subseteq  L^2 _{-a}$ is  in terms
of the $L^2$ inner product.  

 \begin{theorem}[Dispersive estimates]\label{thm:pwdisp}For any fixed   $c \in(0, \sqrt{\frac{2\mathsf{V} }{3}})$,  let $a = c  \sqrt{\varepsilon}$.
 Then there exist  constants $C_1( a ) >0$ and $C_2( a ) >0$  s.t.
   \begin{align}\label{eq:pwdisp1}
      \| e^{  t  \mathcal{L}_c }Q_{[c]}   (n_0, u_0) ^\intercal  \| _{L^2_a}\le C_1( a )   e^{-C_1( a ) t}\|    (n_0, u_0) ^\intercal\| _{L^2_a}\text{ for all } (n_0, u_0) ^\intercal\in L^2_a \text{ and }t\ge 0.
   \end{align}

 \end{theorem}
For the proof  see   \cite[Theorem 1.3]{BK22ARMA}.  
The following  is similar to  \cite[Proposition 6.2]{CM25D1}.
\begin{proposition}[Kato smoothing]\label{prop:smooth}There is a constant $C(  c)>0$ such that
  \begin{align}\label{eq:smooth1} &
      \left \| \int _0 ^t    e^{  (t-s) \mathcal{L}_c  } Q_{[c]} f(s) ds  \right \| _{L^2 \( \R _+ ,    \widetilde{\Sigma}\) }\le C(  c)  \|     f \| _{ L^{ 1, 1} \( \R   ,    L^{ 2}\( \R _+\)\)}\text{ for all }f\in L^{ 1, 1} \( \R   ,    L^{ 2}\( \R _+\)\)   .
   \end{align}
\end{proposition}
\begin{proof}
  Bae and Kwon \cite {BK22ARMA}  show  that     $e^{  t \mathcal{L}_c  }$  for $t\ge 0$  is a $C_0$ semigroup  in $ L^2(\R, \R^2 )$
  and also in  $L^2_a$.  Then  
  \begin{equation*}
    R_{\mathcal{L} _c}(z) Q_{[c]}  w_0 =-\int _{0}^{+\infty} e^{-z t} e^{  t \mathcal{L}_c } Q _{[c]}  w_0 dt \text{ for any }\Re z > - C_1( a ), 
  \end{equation*}  
  defines the resolvent in  $L^2_a$.
   In particular by Theorem    \ref{thm:pwdisp} and by Plancherel identity we have 
  \begin{align*}
  \frac{1}{2\pi}  \|  R_{\mathcal{L} _c}(\im \lambda ) Q w_0  \| _{L^2( \R  , L^2_a) }   =   \| e^{  t \mathcal{L} _c  } Q_{[c]}   w_0  \| _{L^2( \R _+, L^2_a) }\le C(a ) \|  w_0 \| _{ L^2_a} 
  \end{align*}
  for some $C(a )>0$.
   Since  for $a \in (0, \varepsilon \kappa )$
we have $L ^{ 2}   _{a}  \subset \widetilde{\Sigma}$ 
we conclude also
\begin{align*}
  \frac{1}{2\pi}  \|  R_{\mathcal{L} _c }(\im \lambda ) Q_{[c]}   w_0  \| _{L^2( \R  , \widetilde{\Sigma}a) }   =   \| e^{  t \mathcal{L} _c } Q_{[c]}    w_0  \| _{L^2( \R _+, \widetilde{\Sigma}) }\le C(a ) \|  w_0 \| _{ L^2_a}.
  \end{align*}
  Then by Plancherel,   initially for $f(t,x)$ smooth and exponentially  decaying to 0 as $|t|+|x|\to  +\infty$ and then extending by density,  we obtain  
  \begin{equation}\label{eq:fourtran}
     \begin{aligned}
        &
      \left \| \int _0 ^t    e^{  (t-s) \mathcal{L}_c } Q_{[c]}   f(s) ds  \right \| _{L^2 \( \R _+ ,    \widetilde{\Sigma}\) } =  \left \|     R_{\mathcal{L}  _c}  (\im \lambda )   Q_{[c]}  \widehat{f}(\lambda)   \right \| _{L^2 \( \R  ,    \widetilde{\Sigma}\) }\\&\le   \left \|     R_{\mathcal{L}_c }  (\im \lambda )   Q_{[c]}   \right \| _{L^\infty \( \R  ,    \mathcal{L} \( L^{ 1, 1}\( \R\),   \widetilde{\Sigma} \)  \) }   \left \|         \widehat{{f}}    \right \| _{L^2 \( \R _  ,    L^{ 1, 1}\( \R\)\) }\\&\le   \left \|     R_{\mathcal{L}_c }  (\im \lambda )  Q    \right \| _{L^\infty \( \R  ,    \mathcal{L} \( L^{ 1, 1}\( \R\),   \widetilde{\Sigma} \)  \) }  \|     f \| _{ L^{ 1, 1} \( \R   ,    L^{ 2}\( \R _+\)\) } .
     \end{aligned}
  \end{equation}
   The proof  of \eqref{eq:smooth1} is completed by means of Proposition \ref{prop:smooth2}, see below.
\end{proof}

The following is proved in \S \ref{sec:proofsmooth2} using the information on the Jost solutions derived in \S \ref{sec:jost}.
\begin{proposition}\label{prop:smooth2} For   $\varepsilon >0$ small enough    there is a constant $C(  \varepsilon)>0$  such that
  \begin{align}\label{eq:smooth21}&
     \sup  _{\lambda \in \R \backslash \{  0 \}  } \left \|     R_{\mathcal{L} _c} (\im \lambda )  Q_{[c]}   \right \| _{ \mathcal{L}\( L^{ 1,1}\( \R\),   \widetilde{\Sigma} \)   }\le C(  \varepsilon). 
   \end{align}
\end{proposition}

\section{Bounding $\| V _\phi \| _{\widetilde{\Sigma}}$  by Kato smoothing}\label{sec:smooth}

We prove Lemma \ref{eq:sec:smooth11} assuming Proposition \ref{prop:smooth2}.
From the second line in \eqref{EP2} we have $V_\phi = R _{h_c}(-1) V_n +  R _{h_c} (-1)O(V_\phi ^2)$. Then by \eqref{eq:resolvat1} \small
\begin{align}  \nonumber
  \| R _{h_c} (-1)O(V_\phi ^2) \| _{\widetilde{\Sigma}}&\lesssim  \| V_\phi ^2  \| _{\widetilde{\Sigma}}\lesssim \| V_\phi    \| _{L^\infty } (  \|V  \|_{  { \Sigma }_{1A A_1}} + \| V  \|_{  { \Sigma }_{2A A_1}}  )\\& \lesssim \| \widetilde{V}     \| _{L^2 } (  \|V  \|_{  { \Sigma }_{1A A_1}} + \| V  \|_{  { \Sigma }_{2A A_1}}  ) \lesssim B ^{1/2}\delta (  \|V  \|_{  { \Sigma }_{1A A_1}} + \| V  \|_{  { \Sigma }_{2A A_1}}  )\label{eq:smooth---1}
\end{align}
\normalsize
where we used again \eqref{EP2} and \eqref{eq:h1errorbounds}.  To prove Lemma \ref{prop:smooth11} we need to bound 
$\| R _{h_c}(-1)  \widetilde{V} \| _{L^2(I, \widetilde{\Sigma}   )}$.  
Multiplying the first equation in \eqref{EP2}  by $\zeta_{B}$ gives
\begin{align}     \nonumber
    \partial _t\( \zeta  _B {\widetilde{V}}\)& =  \mathcal{L}  _{c_0} \zeta  _B {\widetilde{V}} + [ \zeta  _B,\mathcal{L}  _{c_0} ] \widetilde{V}   + \zeta  _B \partial _x\( 0,  (-\partial _x ^2 + e^{\phi _{c }} ) ^{-1}  e^{\phi _{c }} \(  e^{V _\phi}-1-V _\phi \)  \) ^\intercal  \nonumber
\\&
- \zeta  _B  \partial _x \sigma _1 \( \nabla_{\widetilde{U}} E(S_c+V)      -  \nabla_{\widetilde{U}} E(S_c ) - \nabla ^2_{\widetilde{U}U} E(S_c )V \) \label{eq:prekato1}\\&+ ( \mathcal{L}  _{c } - \mathcal{L}  _{c_0}) \zeta  _B {\widetilde{V}}
+  \zeta  _B    (\dot D-c) \( \widetilde{S}_c+ \widetilde{V} \)'- \zeta  _B  \dot c \partial _c \widetilde{S}_c . \nonumber
\end{align}
Next, for $ P=P_{[c_0]}$ and $Q=   Q_{[c_0]} $ we consider the spectral decomposition
 \begin{align}\label{eq:deftildw}
    \zeta_{B}\widetilde{V}=P \zeta_{B}\widetilde{V}+ \widetilde{W} \text { where }  \widetilde{W}:= Q \zeta_{B}\widetilde{V} .
 \end{align}
The first and preliminary observation is the following.

\begin{lemma}
  \label{lem:disczbv} We have
\begin{align}
   \|  P \zeta_{B}\widetilde{V}  \| _{L^2}\lesssim B ^{-1}  (  \|V  \|_{  { \Sigma }_{1A A_1}} + \| V  \|_{  { \Sigma }_{2A A_1}}  ) .\label{eq:disczbv}
\end{align}
\end{lemma}
\begin{proof} The proof is a simplified version of  similar   one in \cite{CM25D1}. 
   We have
\begin{align*}
  P \zeta_{B}\widetilde{V} =   P_{[c]}\zeta_{B}\widetilde{V} + \(  P -  P_{[c ]}  \) \zeta_{B}\widetilde{V} .
\end{align*}
Since  $\< \eta _1 [c],  \zeta_{B}\widetilde{V} \>=0$  by  \eqref{61}, 
by Theorem \ref{thm:solwave} and by \eqref{eq:defeta2} we obtain  \small
\begin{align*}
  P_{[c]}\zeta_{B}\widetilde{V}  =   \xi _2 [c]    \< \eta _2 [c],  \zeta_{B}\widetilde{V} \> =   \xi _2 [c]    \< \eta _2 [c],  ( \zeta_{B} -1)(\vartheta _{1A_1} + \vartheta _{2A_1}) \widetilde{V}  \> =O\(  B ^{-1} (  \| V  \|_{  { \Sigma }_{1A A_1}} + \| V \|_{  { \Sigma }_{2A A_1}}  ) \) .
\end{align*}\normalsize
 Next we focus on
\begin{align*}
   \(  P_{[c ]} -  P  \) \zeta_{B}\widetilde{V}   = \sum _{i=1,2}\left [  {(\xi _i [c]  -\xi _i[c_0] )  \< \eta _i [c], \zeta_{B}\widetilde{V}\>}
+\xi _i [c_0] \< \eta _i [c]-\eta _i[c_0]   , \zeta_{B}\widetilde{V}\>\right ] .
\end{align*}
Here    \eqref{61} implies   $\< \eta _1 [c], \zeta_{B}\widetilde{V}\> =0$ and \small
\begin{align}\label{eq:mod111}
   |\< \eta _2 [c], \zeta_{B}\widetilde{V}\>|\lesssim   B ^{-1}  \|   \zeta _A \widetilde{V} \| _{ L^2}  =  B ^{-1}  \|   \zeta _A (\vartheta _{1A_1} + \vartheta _{2A_1}) \widetilde{V} \| _{ L^2} \le  B ^{-1}  (  \| \widetilde{V}  \|_{  { \Sigma }_{1A A_1}} + \| \widetilde{V}  \|_{  { \Sigma }_{2A A_1}}  ) .
\end{align} \normalsize
It is elementary, dropping the $[c_0]$ and writing $\xi _j= \xi _j[c_0]$,   that
\begin{align}\label {eq:disczbv1}
   \| \xi _2   \< \eta _2 [c]-\eta _2 [c_0] , \zeta_{B}\widetilde{V}\> \| _{L^2}&\lesssim |\< \eta _2 [c]-\eta _2[c_0]  , \zeta_{B}\widetilde{V}\>|
\\& \lesssim |c-c_0| \| \zeta_{B} \widetilde{V} \| _{L^2}\lesssim \delta  (  \| V  \|_{  { \Sigma }_{1A A_1}} + \| V  \|_{  { \Sigma }_{2A A_1}}  ) . \nonumber
\end{align}
We have the following, where we exploit  \eqref{61},
\begin{align}\label {eq:disczbv2}
   \| \xi _1  \< \eta _1 [c]-\eta _1 [c_0]   , \zeta_{B}\widetilde{V}\> \| _{L^2}\lesssim |\< \eta _1 [c]-\eta _1  [c_0], \zeta_{B}\widetilde{V}\>|  =|\<  \eta _1 [c_0] , \zeta_{B}\widetilde{V}\>| .
\end{align}
Up to this point the proof is verbatim the same of  \cite{CM25D1}. Next   by \eqref{eq:ker2}
\begin{align*}
  \<  \eta _1[c_0] , \zeta_{B}\widetilde{V}\> &= \theta _1[c_0] \<  \int _{-\infty}^x  \left .   \partial _c (u_c(x'),n_c(x') )^\intercal \right |_{c=c_0}   dx'    , \zeta_{B}\widetilde{V}\>  \\& +\frac{\theta _2 [c_0]}{\theta _3 [c_0]} \< \eta _2[c_0] ,   \zeta_{B}\widetilde{V}\> =: \theta _1[c_0] I+II.
\end{align*}
Then like in \eqref{eq:mod111}  and  \eqref{eq:disczbv1} \small
\begin{align}\label {eq:disczbv3}
  |II|\lesssim  |\< \eta _2 [c], \zeta_{B}\widetilde{V}\>|+|\< \eta _2[c_0]   -  \eta _2 [c] ,   \zeta_{B}\widetilde{V}\>| \lesssim B ^{-1}  (  \| V  \|_{  { \Sigma }_{1A A_1}} + \| V  \|_{  { \Sigma }_{2A A_1}}  ) .
\end{align}\normalsize
We have 
\begin{align*}
  I&=  \<  \int _{-\infty} ^{x } \partial _c (u_c(x'),n_c(x') )^\intercal   dx'       , \zeta_{B}\widetilde{V} \> \\& +  \<  \int _{-\infty} ^{x }   \left [ \left .   \partial _c (u_c(x'),n_c(x') )^\intercal \right |_{c=c_0} -  \partial _c (u_c(x'),n_c(x') )^\intercal   \right ]   , \zeta_{B}\widetilde{V}\> =:I_1+I_2.  
\end{align*}
By   $\< \eta _1 [c], \zeta_{B}\widetilde{V}\> =0$, \eqref{eq:ker2} and  \eqref{eq:mod111},  $|I_1|$ is  bounded above by the right hand side in \eqref{eq:disczbv} and
\begin{align*}
  I_2&= (c-c_0)   \int _0 ^1   \<  \int _{-\infty} ^{x }   \left .   \partial _c (u_c(x'),n_c(x') )^\intercal \right |_{c=c_0+s(c-c_0)} dx' ,    \zeta_{B}\widetilde{V}\> ds
\end{align*}
so that by  inequality  \eqref{eq:solwave3} we have 
\begin{align*}
  |I _2| \lesssim  |c-c_0|   \left \|     \frac{\zeta_{B}}{\zeta_{A}}    \right  \| _{L^2}   \left \|       \zeta_{A} \widetilde{V}  (\vartheta _{1A_1} + \vartheta _{2A_1}) \widetilde{V}   \right  \| _{L^2}\lesssim \delta B ^{1 /2 }   (  \| V  \|_{  { \Sigma }_{1A A_1}} + \| V  \|_{  { \Sigma }_{2A A_1}}  ).
\end{align*}
 
\end{proof}

We now focus on the term $\widetilde{W}$   in \eqref{eq:deftildw}. Applying the projection $Q $   to equation \eqref{eq:prekato1}
we obtain
\begin{align}
    \partial _t\widetilde{W}  &=  \mathcal{L}  _{c_0} \widetilde{W}  + [ \zeta  _B,\mathcal{L}  _{c_0} ]\widetilde{W} \nonumber
\\& +  Q\zeta  _B    (\dot D-c) \( \widetilde{S}_c+ \widetilde{V} \)'- Q\zeta  _B  \dot c \partial _c \widetilde{S}_c   \label{eq:gKdV22}
\\&
+   Q  \zeta_{B} \(   \mathcal{L}_{c } - \mathcal{L} _{c _0}  \) \widetilde{V}   \label{eq:gKdV23}  \\&  - Q \zeta  _B \partial _x \sigma _1 \( \nabla_{\widetilde{U}} E(S_c+V)      -  \nabla_{\widetilde{U}} E(S_c ) - \nabla ^2_{\widetilde{U}U} E(S_c )V \)       \label{eq:gKdV24}   \\&   + Q \zeta  _B \partial _x\( 0,  R_{h_c} (-1)  e^{\phi _c} \(  e^{V _\phi}-1-V _\phi \)  \) ^\intercal    \label{eq:gKdV25}
\end{align}
where we use \eqref{eq:defhc}.
We next consider
\begin{align}\label{eq:expv1}
  \widetilde{W} &=e^{  t \mathcal{L} _{c _0}}  \widetilde{W}(0) +   \int _{0}^t e^{ (t-s) \mathcal{L} _{c _0}}Q [\zeta_{B}, \mathcal{L} _{c _0} ]\widetilde{V }ds \\& \quad+     \int _{0}^t e^{  (t-s)\mathcal{L}_{c _0} } \text{lines \eqref{eq:gKdV22}--\eqref{eq:gKdV25}} ds.\label{eq:expv2}
\end{align}
  The proof of   Lemmas \ref{lem:w0} and \ref{lem:gKdV22}   is  the same  of  the analogous ones in   \cite{CM25D1}. In both cases the first inequality follows from $ \| R _{h_c}(-1)\| _{\mathcal{L}(\widetilde{\Sigma} )}\lesssim 1$.

\begin{lemma}\label{lem:w0} We have
   \begin{align}\label{eq:w0} \| R _{h_c}(-1) e^{  t \mathcal{L}_{c _0} }   \widetilde{W}(0)\| _{L^2 (\R _+, \widetilde{\Sigma})}\lesssim
      \| e^{  t \mathcal{L}_{c _0} }   \widetilde{W}(0)\| _{L^2 (\R _+, \widetilde{\Sigma})}\le  \sqrt{\delta}.
   \end{align}
\end{lemma}
\qed

\begin{lemma}\label{lem:gKdV22} We have
   \begin{align}\label{eq:gKdV221} \left \| R _{h_c}(-1)  \int _{0}^t e^{  (t-s)\mathcal{L}_{ c _0 } }   \text{line  \eqref{eq:gKdV22}} ds\right  \| _{L^2 (I, \widetilde{\Sigma})} &\lesssim
     \left  \|  \int _{0}^t e^{  (t-s)\mathcal{L}_{ c _0 } }   \text{line  \eqref{eq:gKdV22}} ds\right \| _{L^2 (I, \widetilde{\Sigma})}\\& \lesssim   \|  \dot D -c \| _{L^2 \( I  \) }+ \|  \dot  c \| _{L^2 \( I  \) }     .\nonumber
   \end{align}
\end{lemma}
\qed

The following  is more complicated and weaker than the corresponding lemma in   \cite{CM25D1}, and this is the reason why in \S \ref{sec:virial} here we had a 3rd virial inequality.

\begin{lemma}\label{lem:gKdV23} We have
   \begin{align}\label{eq:gKdV231}\left \| R _{h_c}(-1)  \int _{0}^t e^{  (t-s)\mathcal{L}_{ c _0 } }   \text{line  \eqref{eq:gKdV23}} ds\right  \| _{L^2 (I, \widetilde{\Sigma})} & \lesssim  \sqrt{ \delta}  (\| V \| _{L^2(I,  { \Sigma }_{1A A_1}   )}+\| V \| _{L^2(I,  { \Sigma }_{2A A_1}   )}). 
   \end{align}
\end{lemma}
 \begin{proof} We have
\begin{align*}
   &  \(   \mathcal{L}_{c } - \mathcal{L} _{c _0}  \) \widetilde{V} = \partial _x ( {L}_{c _0 } -  {L} _{c }  ) \widetilde{V}+\partial _x  (R _{h_{c _0 }} (-1) -  R _{h_{c   }}) (-1)
   \begin{pmatrix}
                                                    0 & 0 \\
                                                   1 & 0 \\
                                                  \end{pmatrix}
                                                 \widetilde{V} =\mathcal{A} +\mathcal{B}.
                                           \end{align*}
Omitting the irrelevant constant matrix,   the 2nd resolvent identity and \eqref{eq:smooth1} give \small
\begin{align*}  &
    \left  \| Q\zeta  _B \mathcal{B}\right \| _{L^2 (I, \widetilde{\Sigma})}    \lesssim \| \zeta  _B  \partial _x  R _{h_{c   }} (-1) \(  e^{\phi _c}- e^{\phi _{c _0}}   \) R _{h_{c _0 }} (-1) \widetilde{V}  \| _{ L^{ 1, 1} \( \R   ,    L^{ 2}\( I\)\)} \\& \lesssim B  \|   \(   {\phi _c}-  {\phi _{c _0}}   \) R _{h_{c _0 }} (-1) \widetilde{V}  \| _{ L^{ 1 } \( \R   ,    L^{ 2}\( I\)\)}\le  B  \|  \zeta _A ^{-1}   \(   {\phi _c}-  {\phi _{c _0}}   \)\| _{ L^{ 2} \( \R   ,    L^{ \infty }\( I\)\)}    \|   \zeta _A \widetilde{V}  \| _{ L^{ 2 } \( \R   ,    L^{ 2}\( I\)\)} \\& \lesssim \| c-c_0\| _{  L^{ \infty }\( I\)}  (\| V \| _{L^2(I,  { \Sigma }_{1A A_1}   )}+\| V \| _{L^2(I,  { \Sigma }_{2A A_1}   )}) 
\end{align*} \normalsize
 which takes care of the contribution to \eqref{eq:gKdV231} of $\mathcal{B}$.  The contribution  from $\mathcal{A}$ can  be bounded above by 
 \begin{align*}& 
    \left \| R _{h_c}(-1)  \int _{0}^t e^{  (t-s)\mathcal{L}_{ c _0 } } Q\zeta _B ' ({L}_{c _0 } -  {L} _{c }) \widetilde{V} ds\right  \| _{L^2 (I, \widetilde{\Sigma})}\\& +  \left \| R _{h_c}(-1)  \int _{0}^t e^{  (t-s)\mathcal{L}_{ c _0 } } Q    R _{h_{c   }} (-1)
   \begin{pmatrix}
                                                    0 & 0 \\
                                                   1 & 0 \\
                                                  \end{pmatrix}\partial _x
                                                 \zeta _B{L}_{c _0 }^{-1}({L}_{c _0 } -  {L} _{c }) \widetilde{V} ds\right  \| _{L^2 (I, \widetilde{\Sigma})} \\& +  \left \| R _{h_c}(-1) \mathcal{L}_{ c _0 } \int _{0}^t e^{  (t-s)\mathcal{L}_{ c _0 } } Q\zeta _B  {L}_{c _0 }^{-1} ({L}_{c _0 } -  {L} _{c }) \widetilde{V} ds\right  \| _{L^2 (I, \widetilde{\Sigma})} =:\sum_{i=1,2,3}A_i.
 \end{align*}
 By \eqref{eq:smooth1}, \eqref{eq:h1errorbounds}, \eqref{eq:resolvat1}     \small
 \begin{align*}
   A_1+A_2&\lesssim \|  (c-c_0) \zeta _B  \widetilde{V}    \| _{ L^{ 1, 1}  \( \R   ,    L^{ 2}\( I\) \) } \lesssim  \| c-c_0 \| _{ L^{ \infty}\( I\)}    \left  \|  \frac{\zeta _B}{\zeta _A} \right  \| _{  L^{ 1, 1} \( \R \)}   \(       \| V \| _{L^2(I,    \Sigma  _{1A A_1}   )} +   \| V \| _{L^2(I,   \Sigma  _{2A A_1}   )} \) \\& \lesssim B^2 \delta   \(       \| V \| _{L^2(I,    \Sigma  _{1A A_1}   )} +   \| V \| _{L^2(I,   \Sigma  _{2A A_1}   )} \) .
 \end{align*}
 \normalsize 
 Finally, we similarly have   
\begin{align*}
   A_3&\lesssim   \left \|   \int _{0}^t e^{  (t-s)\mathcal{L}_{ c _0 } } Q\zeta _B  {L}_{c _0 }^{-1} ({L}_{c _0 } -  {L} _{c }) \widetilde{V} ds\right  \| _{L^2 (I, \widetilde{\Sigma})} \\&  \lesssim B^2 \delta   \(       \| V \| _{L^2(I,    \Sigma  _{1A A_1}   )} +   \| V \| _{L^2(I,   \Sigma  _{2A A_1}   )} \) .
\end{align*}

 \end{proof}

\begin{lemma}\label{lem:gKdV24} We have
   \begin{align}\label{eq:gKdV241}  \left  \| R _{h_c}(-1)  \int _{0}^t e^{  (t-s)\mathcal{L} _{c _0}} \text{line  \eqref{eq:gKdV24}} ds\right \| _{L^2 (I, \widetilde{\Sigma})}& \lesssim 
     \left  \|  \int _{0}^t e^{  (t-s)\mathcal{L} _{c _0}} \text{line  \eqref{eq:gKdV24}} ds\right \| _{L^2 (I, \widetilde{\Sigma})}\\& \lesssim   \delta ^{   \frac{1}{2}   }  (\| V \| _{L^2(I,  { \Sigma }_{1A A_1}   )}+\| V \| _{L^2(I,  { \Sigma }_{2A A_1}   )}). \nonumber
   \end{align}
\end{lemma}
\begin{proof} Again, the first line follows from \eqref{eq:resolvat1}. 
    Schematically, using Proposition \ref{prop:smooth},  the  term on the right in line \eqref{eq:gKdV241}  is 
\begin{align*}&
   \left  \|  \int _{0}^t e^{  (t-s)\mathcal{L} _{ c _0 }} \partial _x O( \zeta _B V    ^2 )     \right \| _{L^2 (I, \widetilde{\Sigma})} \lesssim      \left \|      \partial _x O( \zeta _B V    ^2 ) \right \| _{ L^{ 1, 1} \( \R   ,    L^{ 2}\( I \)\)}    \\& \lesssim \left \|   \frac{\zeta _B}{\zeta _A} \right \|   _{L^{ \infty , 1} \( \R    \)} 
 \left \|  V  \right \|   _{L^{ \infty } \( I, H^1    \)}   (\| V \| _{L^2(I,  { \Sigma }_{1A A_1}   )}+\| V \| _{L^2(I,  { \Sigma }_{2A A_1}   )})\\& \lesssim B ^{3/2}\delta  (\| V \| _{L^2(I,  { \Sigma }_{1A A_1}   )}+\| V \| _{L^2(I,  { \Sigma }_{2A A_1}   )}) .
\end{align*}

 \end{proof}

By \eqref{eq:resolvat1} and   $ e^{V _\phi}-1-V _\phi = O(V _\phi ^2)$  it is clear that the following  is true.

\begin{lemma}\label{lem:gKdV25} We have
   \begin{align}\label{eq:gKdV251}   \left  \|  \int _{0}^t e^{  (t-s)\mathcal{L} _{c _0}} \text{line  \eqref{eq:gKdV25}} ds\right \| _{L^2 (I, \widetilde{\Sigma})}&\lesssim
     \left  \|  \int _{0}^t e^{  (t-s)\mathcal{L} _{c _0}} \text{line  \eqref{eq:gKdV25}} ds\right \| _{L^2 (I, \widetilde{\Sigma})}\\& \lesssim   \delta ^{   \frac{1}{2}   }  (\| V \| _{L^2(I,  { \Sigma }_{1A A_1}   )}+\| V \| _{L^2(I,  { \Sigma }_{2A A_1}   )}).\nonumber 
   \end{align}
\end{lemma}
\qed

The most delicate contributor in the formula \eqref{eq:gKdV22}--\eqref{eq:gKdV25} is the commutator.
We have
\begin{align}&\label{eq:comm}
   [ \zeta  _B,\mathcal{L}  _{c_0} ] \widetilde{V}=  \zeta_{B}' 
\begin{pmatrix}
   -c & 1 \\
     K & -c
 \end{pmatrix} \widetilde{V}  +  \zeta_{B}'
\begin{pmatrix}
   u_c & n_c \\
    - \frac{Kn_c}{1+n_c} &  u_c
 \end{pmatrix} \widetilde{V}    \\&  - \partial _x R_{h_c} (-1)  \(  \zeta _B ''+2 \zeta _B ' \partial _x   \)  R_{h_c} (-1)  \begin{pmatrix}
 0 & 0 \\ 1 & 0
 \end{pmatrix}
        \widetilde{V}   =: D_1+D_2+D_3.\nonumber 
\end{align}
The easiest term is $D_2$ for which we have for some fixed  $\alpha>0$ and by  Proposition \ref{prop:smooth} \small
\begin{align}\label{eq:estD2}
    &\left  \|  \int _{0}^t e^{  (t-s)\mathcal{L} _{c _0}} D_2 ds\right \| _{L^2 (I, \widetilde{\Sigma})} \lesssim    
     \| \< x \>  e^{-\alpha\sqrt{ \varepsilon}|x|}   \zeta _B ' \widetilde{V} \| _{ L^{ 1 } \( \R   ,    L^{ 2}\( I\)\)}\\&  \le C_\varepsilon B^{-1}  \left  \|  \frac{\zeta _B }{\zeta _A }  \right \| _{L^\infty }   \left  \|    \zeta _A (\vartheta _{1A_1} + \vartheta _{2A_1})  \widetilde{V} \right \| _{L^2 (I\times \R) }\lesssim   B^{-1}(\| V \| _{L^2(I,  { \Sigma }_{1A A_1}   )}+\| V \| _{L^2(I,  { \Sigma }_{2A A_1}   )})  \nonumber
\end{align}\normalsize
where here and below   by    $ \| R _{h_c}(-1)  \| _{ \mathcal{L}( \widetilde{\Sigma}   )}\lesssim 1$  we can ignore $R _{h_c}(-1)$.
We write  \small
\begin{align*}
 D_3&=  - \partial _x  R_{h_c} (-1)  \(  \zeta _B ''+2 \zeta _B ' \partial _x   \) \( 0,      V_\phi \)^\intercal \\& - \partial _x R_{h_c} (-1)  \(  \zeta _B ''+2 \zeta _B ' \partial _x   \)  \( 0,  R_{h_c} (-1)   e^{\phi _c} \(  e^{V _\phi}-1-V _\phi \)  \) ^\intercal =:D_{31}+D_{32}.
\end{align*}
\normalsize
By   Proposition \ref{prop:smooth} and \eqref{eq:resolvat1}
\begin{align*}
    &\left  \|  \int _{0}^t e^{  (t-s)\mathcal{L} _{c _0}} D_{32} ds\right \| _{L^2 (I, \widetilde{\Sigma})} \lesssim
     \|  \< x \>    D_{32} \| _{ L^{ 1 } \( \R   ,    L^{ 2}\( I\)\)}\\&  \nonumber \lesssim   \left  \|  \< x \> ^2 e^{-|x|} *   \( \zeta _B ''    e^{-|x|} *  e^{\phi _c} O \(  V _\phi ^2  \) \)    \right \| _{ L^{ 1 } \( \R   ,    L^{ 2}\( I\)\)}  \\&  +    \left  \|  \< x \> ^2 e^{-|x|} *  \(\zeta _B '  \< x \>  e^{-|x|} * e^{\phi _c} O \(  V _\phi ^2  \)\)     \right \| _{ L^{ 1 } \( \R   ,    L^{ 2}\( I\)\)}  =: D_{321} +D_{322}   . \nonumber
\end{align*}
Then \small
\begin{align} \nonumber
  D_{321}&\lesssim  B^{-2}  \left  \|  \< x \> ^2  \zeta  _{2B}     \zeta _{2B}^{-1}  e^{-|x|} *     \zeta _{2B}\zeta _{2B}     e^{-|x|} *   \zeta _{2B}^{-1}  \zeta _{2B} \|   O \(  V _\phi ^2  \) \|_{L^{ 2}\( I\)}    \right \| _{ L^{ 1 } \( \R   \)}\\& \lesssim   B^{-2}  \left  \|  \< x \> ^2  \zeta  _{2B}  \right \| _{ L^{ 2 } \( \R   \)}   \left  \|    \zeta _{2B}^{-1}  e^{-|x|} *    \( \zeta _{2B} \sqcup \) \right \| _{\mathcal{ L}\( L^{ 2 } \( \R   \) \)} ^2  \left  \|     \frac{\zeta  _{2B}}{\zeta  _{A}} \zeta  _{A}(\vartheta _{1A_1} + \vartheta _{2A_1})    V _\phi ^2   \right \| _{ L^{ 2 } \( I\times \R   \)}\nonumber\\& \lesssim B^{1/2}  \left  \|      V _\phi   \right \| _{ L^{ \infty } \( I\times \R   \)}   (\| V \| _{L^2(I,  { \Sigma }_{1A A_1}   )}+\| V \| _{L^2(I,  { \Sigma }_{2A A_1}   )})\nonumber\\&  \lesssim B^{-1}(\| V \| _{L^2(I,  { \Sigma }_{1A A_1}   )}+\| V \| _{L^2(I,  { \Sigma }_{2A A_1}   )}) \label{eq:estD321}
\end{align}
\normalsize
and 
\small
\begin{align} \nonumber
 & D_{322} \lesssim  B^{-1}  \left  \|  \< x \> ^2  \zeta  _{2B}     \zeta _{2B}^{-1}  e^{-|x|} *     \< x \> \zeta _{2B}\zeta _{2B}     e^{-|x|} *   \< x \>^{-1}  \< x \> \zeta _{2B}^{-1}  \zeta _{2B} \|   O \(  V _\phi ^2  \) \|_{L^{ 2}\( I\)}    \right \| _{ L^{ 1 } \( \R   \)}\\& \lesssim   B^{ -1}  \left  \|  \< x \> ^2  \zeta  _{2B}  \right \| _{ L^{ 2 } \( \R   \)}   \left  \|   \< x \>  \zeta _{2B}  e^{-|x|} *   \(  \< x \> ^{-1}\zeta _{2B}^{-1} \sqcup \) \right \| _{\mathcal{ L}\( L^{ 2 } \( \R   \) \)}    \left  \|   \< x \>   \frac{\zeta  _{2B}}{\zeta  _{\frac{A}{2}}} \zeta  _{\frac{A}{2}}     V _\phi ^2   \right \| _{ L^{ 2 } \( I\times \R   \)}\nonumber\\& \lesssim  B^{ \frac{3}{2}}  \left  \|     \zeta _{A}   \right \| _{ L^{ 2 } \(   \R   \)}   \left  \|      V _\phi   \right \| _{ L^{ \infty } \( I\times \R   \)}  \left  \|     \zeta _{A} (\vartheta _{1A_1} + \vartheta _{2A_1})    V _\phi  \right \| _{ L^{ 2 } \( I\times \R   \)} 
  \nonumber\\&  \lesssim B^{-1}(\| V \| _{L^2(I,  { \Sigma }_{1A A_1}   )}+\| V \| _{L^2(I,  { \Sigma }_{2A A_1}   )}) . \label{eq:estD322}
\end{align}
\normalsize
The estimate of the term $ D_{31}$, which is linear in $V$, requires more care.
We write 
\begin{align*}
  D_{31}&= - \partial _x  R_{h_c} (-1)    \zeta _B ''      V_\phi  - \partial _x  R_{h_c} (-1)      \zeta _B ' \partial _x        V_\phi =:  D_{311} + D_{312} 
\end{align*} with a  harmless abuse of notation.  By $ \partial _x  R_{h_c} (-1)=[ \partial _x , R_{h_c} (-1)]+ R_{h_c} (-1)\partial _x$
we have 
\begin{align*}
   D_{311}&=  R_{h_c} (-1)   \(   e^{\phi _c} \) ' R_{h_c} (-1) \zeta _B '' V_\phi   - R_{h_c} (-1)   \zeta _B '''      V_\phi +R_{h_c} (-1)   \zeta _B ''      V_\phi ' .
\end{align*}
Then, schematically,
\begin{align}\nonumber
    &\left  \|  \int _{0}^t e^{  (t-s)\mathcal{L} _{c _0}} D_{311} ds\right \| _{L^2 (I, \widetilde{\Sigma})} \lesssim
     \|  \< x \>    D_{311} \| _{ L^{ 1 } \( \R   ,    L^{ 2}\( I\)\)}\\&  \nonumber \lesssim  B ^{-2} \left  \|  \< x \>   e^{-|x|} *   \(   \< x \> ^{-1}  \< x \>\zeta _B      e^{-|x|} *     \(  V _\phi   + V _\phi '  \) \)    \right \| _{ L^{ 1 } \( \R   ,    L^{ 2}\( I\)\)}  \\&  \nonumber    \lesssim   B ^{-2} \left  \|    \< x \>\zeta _B      e^{-|x|} * \(  \< x \>^{-1}\zeta _B ^{-1} \< x \>\zeta _B      \(  V _\phi   + V _\phi '  \) \)    \right \| _{ L^{ 1 } \( \R   ,    L^{ 2}\( I\)\)}\\& \nonumber \lesssim B ^{-2} \left  \|     \< x \>\zeta _B      \(  V _\phi   + V _\phi '  \)     \right \| _{ L^{ 1 } \( \R   ,    L^{ 2}\( I\)\)}\\& \lesssim  B ^{-2} \left  \|     \< x \>\zeta  _{2B}        \right \| _{     L^{ 2}\( I\) }   \left  \|   \frac{ \zeta  _{2B}}{\zeta  _{A}}  \zeta  _{A} (\vartheta _{1A_1} + \vartheta _{2A_1}) \(  V _\phi   + V _\phi '  \)     \right \| _{ L^{ 2 } \( I\times \R   \)}\nonumber \\& \lesssim B ^{-1/2} (\| V \| _{L^2(I,  { \Sigma }_{1A A_1}   )}+\| V \| _{L^2(I,  { \Sigma }_{2A A_1}   )}) .\label{eq:estD311}
\end{align}
By an elementary computation  
\begin{align*}
   D_{312}&=  R_{h_c} (-1)       e^{\phi _c}\phi _c 'R_{h_c} (-1) \zeta _B '' V_\phi    + R_{h_c} (-1)   e^{\phi _c}\phi _c ' R_{h_c} (-1) \( \zeta _B '   V_\phi \) ' \\& + R_{h_c} (-1)  \( \zeta _B ''   V_\phi \) '  - R_{h_c} (-1)\( e^{\phi _c}-1   \) \zeta _B '   V_\phi\\&  -R_{h_c} (-1) \zeta _B '   V_\phi   + \zeta _B '   V_\phi =:\sum _{k=1}^{6}D_{312k}.
\end{align*}
It is easy to show 
 \begin{align}
    &\left  \|  \int _{0}^t e^{  (t-s)\mathcal{L} _{c _0}}D_{312k} ds\right \| _{L^2 (I, \widetilde{\Sigma})} \lesssim
      B ^{-1/2} (\| V \| _{L^2(I,  { \Sigma }_{1A A_1}   )}+\| V \| _{L^2(I,  { \Sigma }_{2A A_1}   )}) \text{ for all }k\le 4.\label{eq:est312k}
\end{align}
We will prove the following lemma in   \S \ref{sec:gKdV21hard}.
\begin{lemma}\label{lem:gKdV21hard} For $D =D_1$ in \eqref{eq:comm} and for $D= D_{3125},     D_{3126}$      we have
   \begin{align}\label{eq:gKdV21hard1}
     \left  \|  \int _{0}^t e^{  (t-s)\mathcal{L}_{c _0}} Q  D      ds\right \| _{L^2 (I, \widetilde{\Sigma})}\lesssim   A_ 1^{3/2}B ^{-1 }    \| V \| _{L^2(I,  { \Sigma }_{1A A_1}   )}+B ^{ 1/2 } \| V \| _{L^2(I,  { \Sigma }_{2A A_1}   )} .
   \end{align}
 \end{lemma}
\qed

\textit{Proof of Lemma \ref{prop:smooth11}.} From   Lemmas \ref{lem:disczbv}--\ref{lem:gKdV21hard}  and Lemma \ref{lem:lemdscrt}
we have
\begin{align*}
  \| \widetilde{W} \| _{L^2(I, \widetilde{\Sigma}   )} \lesssim \( A_ 1^{3/2}B ^{-1 }  +B ^{-1/2 }  \)  \| V \| _{L^2(I,  { \Sigma }_{1A A_1}   )}+B ^{ 1/2 } \| V \| _{L^2(I,  { \Sigma }_{2A A_1}   )} .
\end{align*}
 Using  \eqref{eq:relABg}, decomposition \eqref{eq:defw} and inequality \eqref{eq:disczbv}, we have the following
\begin{align*}
   \| \widetilde{V} \| _{L^2(I, \widetilde{\Sigma}   )} &\le      \| (1-\zeta_{B}) \widetilde{V}   \| _{L^2(I, \widetilde{\Sigma}   )}  + \| \zeta_{B} \widetilde{V}  \| _{L^2(I, \widetilde{\Sigma}   )}   \\& \lesssim B ^{-1} \(  \| V \| _{L^2(I,  { \Sigma }_{1A A_1}   )} + \| V \| _{L^2(I,  { \Sigma }_{2A A_1}   )}    \) + \| P \zeta_{B}\widetilde{V}  \| _{L^2(I, \widetilde{\Sigma}   )} +\|   \widetilde{W} \| _{L^2(I, \widetilde{\Sigma}   )}  \\& \lesssim  A_ 1^{3/2}B ^{-1 }    \| V \| _{L^2(I,  { \Sigma }_{1A A_1}   )}+B ^{ 1/2 } \| V \| _{L^2(I,  { \Sigma }_{2A A_1}   )}  .
\end{align*}
From   \eqref{EP2} and    \eqref{eq:resolvat1}    we have  the following, which yields \eqref{eq:sec:smooth11},
\begin{align*}
    \|  {V}_\phi  \| _{L^2(I, \widetilde{\Sigma}   )} & \lesssim    \| (-\partial _x ^2 + e^{\phi _c} ) ^{-1}  {V}_n  \| _{L^2(I, \widetilde{\Sigma}   )} +  \| (-\partial _x ^2 + e^{\phi _c} ) ^{-1}  {V}_\phi ^2  \| _{L^2(I, \widetilde{\Sigma}   )} \\&   \lesssim    \| {V}_n \| _{L^2(I, \widetilde{\Sigma}   )} +  \|    {V}_\phi    \| _{L^\infty (I\times \R   )}  \|    {V}_\phi    \| _{L^2(I, \widetilde{\Sigma}   )}.
\end{align*}

\qed

\section{The Jost functions}\label{sec:jost}

For $\mathcal{L} =\mathcal{L} _c$,    we consider here the equation
\begin{align}\label{eq:resoleq}
(\lambda  - \mathcal{L} ) \widetilde{U}  =\widetilde{F}   \text{ in } L^2(\R , \C ^2) \text{ with }  \widetilde{U}=(\dot n , \dot u ) ^\intercal .
\end{align}
We set
  \begin{equation}\label{Poisson_ODE}
\partial_x\dot{\phi} =: \dot{\psi} , \quad \partial_x\dot{\psi}  =  e^{\phi_c}\dot{\phi} - \dot{n}.
\end{equation}
 Then,    letting $\mathbf{U}:=(\dot{n},\dot{u},\dot{\phi},\dot{\psi})^\intercal $, we obtain the system
 \begin{equation}\label{ODE_LinEP}
\( \frac{d }{dx} - A(x,\lambda,\varepsilon) \)
\mathbf{U}  = \mathbf{F}   := ( L _c^{-1}\widetilde{F},0,0)^\intercal  
\end{equation}
 with the coefficient matrix, for $L=L_c$,
\begin{equation}\label{ODE_LinEP1}
A=A(x,\lambda,\varepsilon):=
 \left(
\begin{array}{c|c}
L^{-1} & \begin{array}{cc}
						0 & 0 \\
						0 & 0
					\end{array} \\
		\hline
\begin{array}{cc}
0 & 0 \\
0 & 0
\end{array}
&
\begin{array}{cc}
1 & 0 \\
0 & 1
\end{array}
\end{array}
\right)
\left(
\begin{array}{c|c}
-\lambda I_2 - \partial_xL & \begin{array}{cc}
						0 & 0 \\
						0 & -1
					\end{array} \\
		\hline
\begin{array}{cc}
0 & 0 \\
-1 & 0
\end{array}
&
\begin{array}{cc}
0 & 1 \\
e^{\phi_c} & 0
\end{array}
\end{array}
\right)
\end{equation}
where $A(x,\lambda,\varepsilon ) =A_1(x,\varepsilon ) + \lambda A_2(x,\varepsilon )$, with
\begin{subequations}\label{A_Decompose}
\begin{align*}
 A_1 & :=   
 \begin{pmatrix}
\frac{(c-u_c)\partial_xu_c}{J} - \frac{K\partial_x n_c}{J(1+n_c)} & \frac{(c-u_c)\partial_xn_c}{J} + \frac{(1+n_c)\partial_xu_c}{J} & 0 & \frac{1+n_c}{J} \\
\frac{K\partial_x u_c}{J(1+n_c)} -\frac{K(c-u_c)\partial_xn_c}{J(1+n_c)^2} & \frac{K\partial_xn_c}{J(1+n_c)} + \frac{(c-u_c)\partial_xu_c}{J} & 0 & \frac{c-u_c}{J} \\
0 & 0 & 0 & 1 \\
-1 & 0 & e^{\phi_c} & 0
\end{pmatrix},
\\
A_2& :=
\left(
\begin{array}{c|c}
-L^{-1} & \mathbf{0}_2 \\
\hline
\mathbf{0}_2 & \mathbf{0}_2
\end{array}
\right) = \scriptsize\frac{1}{J}\left(
\begin{array}{c|c}
\begin{array}{cc}
c-u_c & 1+n_c \\
\dfrac{K}{1+n_c} & c-u_c
\end{array} & \mathbf{0}_2 \\
\hline
\mathbf{0}_2 & \mathbf{0}_2
\end{array}
\right)
\end{align*}
\end{subequations}
with  $\mathbf{0}_2$ the $2\times 2$ zero matrix and with \begin{equation}\label{Useful_Id3}
J=J(x,\varepsilon):=(c-u_c(x))^2-K>0.
\end{equation}
We will focus now   on
\begin{align}
  \label{eq:lineq1} \( \frac{d }{dx} - A(x,\lambda,\varepsilon) \)
\mathbf{U}  = 0 \text{ for } \lambda \in \overline{\C_+}    .
\end{align}
Notice that if we set for $\psi _c:=\phi '_c$
\begin{align}\label{eq:bigxi}
  \Xi _1[c]:= ( \xi _1 [c]^{\intercal} , \phi '_c , \psi '_c) ^\intercal \text{ and } \Xi _2[c]:= ( \xi _2 [c]^{\intercal} , \partial _c\phi  _c , \partial _c\psi  _c)^\intercal
\end{align}
then
\begin{align}\label{eq:bigxi1}
 \( \frac{d }{dx} - A_1(x, \varepsilon) \)
 \Xi _1[c]  = 0 \text{ and  }   \( \frac{d }{dx} - A_1(x, \varepsilon) \)
 \Xi _2[c]  = -A_2(x, \varepsilon)\Xi _1[c] .
\end{align}
Similarly,  proceeding like   Bae and Kwon \cite[p. 272]{BK22ARMA}  from \eqref{eq:dualkernel}  we obtain
\begin{align}\label{eq:bigeta1}
 \( \frac{d }{dx} + A_1(x, \varepsilon) ^\intercal  \)
 E _1[c]  = 0 \text{ and  }   \( \frac{d }{dx} + A_1(x, \varepsilon)  ^\intercal \)
 E _2[c]  =  A_2(x, \varepsilon) ^\intercal E _1[c]  ,
\end{align}
where for $j=1,2$  
\begin{align}  \label{eq:bigeta12}   E _j[c] = \left(
                                                 \begin{array}{c}
                                                    L _c    ^{\intercal}     \eta _j[c]  \\
                                                    \widetilde{\phi} _j \\
                                                    \widetilde{\psi} _j \\
                                                 \end{array}
                                               \right)  \text{ with } \partial _x  \left(
                                                 \begin{array}{c} 
                                                    \widetilde{\phi} _j \\
                                                    \widetilde{\psi} _j \\
                                                 \end{array}
                                               \right) = \left(
                                                           \begin{array}{cccc}
                                                             0 & 0 & 0 & -e^{\phi _c} \\
                                                             0 & 1 & -1 & 0 \\
                                                           \end{array}
                                                         \right)   \left(
                                                 \begin{array}{c}
                                                       \eta _j[c]  \\
                                                    \widetilde{\phi} _j \\
                                                    \widetilde{\psi} _j \\
                                                 \end{array}
                                               \right) .
\end{align}
We can consider $\displaystyle A^{\infty}( \lambda,\varepsilon):=\lim _{x\to \infty}A(x,\lambda,\varepsilon)$. Then   \begin{equation}\label{A_Asymptotic}
A^\infty(\lambda,\varepsilon )
:=
\begin{pmatrix}
\dfrac{c\lambda}{c^2-K} & \dfrac{\lambda}{c^2-K} & 0 & \dfrac{1}{c^2-K} \\
\dfrac{K \lambda}{c^2-K} & \dfrac{c\lambda}{c^2-K} & 0 & \dfrac{c}{c^2-K} \\
0 & 0 & 0 & 1 \\
-1 & 0 & 1 & 0
\end{pmatrix}  \ .
\end{equation}
We will write
\begin{align}
  \label{eq:pertA} & A(x,\lambda,\varepsilon)=  A^{\infty}( \lambda,\varepsilon) +   q (x,\lambda , \varepsilon)  \text{ where} \\&   A^{\infty}  ( \lambda,\varepsilon)=  A^{\infty} _1(  \varepsilon) +\lambda A^{\infty} _2(  \varepsilon)     \text{ and} \nonumber  \\&   q (x,\lambda , \varepsilon)=  q _1(x  , \varepsilon) +\lambda q _2(x  , \varepsilon)  \nonumber
\end{align}
and where
\small
\begin{align}
  \label{eq:estqj}\text{$ |\partial ^{l}_x  q _j(x  , \varepsilon)|\le C_l \varepsilon^{1+\frac{j  }{2}} e^{-\alpha \sqrt{\varepsilon} |x|}$  with fixed constants $C_l$ for all $l\ge 0$ and for a fixed $ \alpha >0$.}
\end{align}\normalsize
The   eigenvalues $\mu$ of $A^\infty$ are the zeros of the characteristic polynomial
\begin{equation}\label{dispersEP}
\begin{split}
d(\mu) = d(\mu,\lambda,\varepsilon)
& := \textnormal{det}\,\left(\mu I-A^\infty(\lambda,\varepsilon) \right) \\
& =(c^2-K)^{-1}\left( (\mu^2-1)\left[(\lambda - c\mu)^2 - K \mu^2 \right] + \mu^2 \right).
\end{split}
\end{equation}
This is equivalent to the fact that  they satisfy   one of the equations
\begin{equation}\label{Charact1}
d_{\pm}(\mu)=d_{\pm}(\mu,\varepsilon ):= \mu \left(c \pm \sqrt{\frac{1}{1-\mu^2}+ K}\, \right) = \lambda.
\end{equation}
For each $\varepsilon > 0$ and $\lambda \in \mathbb{C}$, $d(\mu)$ has four zeros $\mu_j$ $(j=1,2,3,4)$.  For any
$\mu_j$ we set
\begin{equation}\label{LRVec_A}
\mathbf{v}_j  := \left(1,\; \frac{c\mu_j - \lambda}{\mu_j} , \; \frac{1}{1-\mu_j^2}, \;\frac{\mu_j}{1-\mu_j^2}  \right) ^\intercal, \quad \mathbf{w}_j  := \frac{\boldsymbol{\pi}_j}{\< \boldsymbol{\pi}_j ,  \mathbf{v}_j \> _{\C ^4}},
\end{equation}
where $\< \mathbf{w}_i ,  \mathbf{v}_j \> _{\C ^4}=\delta _{i,j}$  and
\begin{subequations}\label{eigenvec_A}
\begin{align}
 \boldsymbol{\pi}_j & :=  \left(\left( \frac{c\lambda}{\mu_j} - (c^2-K)\right)(1-\mu_j^2), \; -\lambda\frac{1-\mu_j^2}{\mu_j}, \; 1, \;\mu_j \right) ^\intercal , \label{eigenvec_A1} \\
 \< \boldsymbol{\pi}_j ,  \mathbf{v}_j \> _{\C ^4} &  =  \frac{\lambda^2(1-\mu_j^2)}{\mu_j^2} - (c^2-K)(1-\mu_j^2) +\frac{1+\mu_j^2}{1-\mu_j^2}. \label{eigenvec_A2}
\end{align}
\end{subequations}Using the column vectors in \eqref{LRVec_A}  we consider the matrices
\begin{align}
  \label{eq:defVW}\mathbf{V}(\lambda,\varepsilon):= (\mathbf{v}_{ 1}, \mathbf{v}_{ 2}, \mathbf{v}_{ 3}, \mathbf{v}_{ 4})(\lambda,\varepsilon) \text{ and }\mathbf{W}(\lambda,\varepsilon):= (\mathbf{w}_{ 1}, \mathbf{w}_{ 2}, \mathbf{w}_{ 3}, \mathbf{w}_{ 4})(\lambda,\varepsilon).
\end{align}
We define also 
\begin{align}\label{eq:boldmu}&
  \boldsymbol{\mu}(\lambda,\varepsilon) := \diag(\mu_{ 1}, \mu_{ 2}, \mu_{ 3},\mu_{ 4})(\lambda,\varepsilon) \text{ and}\\ & \mathbf{M}(\lambda,\varepsilon) := \diag(m_{ 1}, m_{ 2}, m_{ 3},m_{ 4})(\lambda,\varepsilon) \text{ with } \label{mj} \\ & m_j=m_j(\lambda,\varepsilon ) := (c+(-1)^j\sqrt{K}) |1-\mu_j^2|^{\frac{1}{2}}, \quad (j=1,2,3,4).\nonumber
\end{align}
Setting also \begin{equation}\label{V0W0}
\widetilde{\mathbf{V}}  := \mathbf{V}  \mathbf{M}, \quad
 \widetilde{\mathbf{W}}  := \mathbf{W} \mathbf{M}^{-1} 
\end{equation}
 we have 
\begin{equation}\label{W0V01}
\widetilde{\mathbf{W}}^{\intercal}\widetilde{\mathbf{V}}=1, \quad \widetilde{\mathbf{W}}^{\intercal} A^\infty  \widetilde{\mathbf{V}} = \boldsymbol{\mu},
\end{equation}
and,  using the $q$ in \eqref{eq:pertA},
\begin{equation}\label{DecompW0(A-Ainf)V0}
\widetilde{\mathbf{W}}^{\intercal}  q     \widetilde{\mathbf{V}} = \lambda \widetilde{\mathbf{W}}^{\intercal}  q _2 \widetilde{\mathbf{V}} + \widetilde{\mathbf{W}} ^{\intercal}q _1\widetilde{\mathbf{V}}.
\end{equation} Bae and Kwon \cite{BK22ARMA} show that it is possible to take \begin{subequations} \label{EigenSpliting1}
\begin{align}
& \Re\mu_1 < 0 = \Re\mu_2 = \Re\mu_3 < \Re \mu_4, & \text{when} \;  \Re\lambda =0, \label{EigenSpliting1_1} \\
& \Re\mu_1 < 0 < \Re\mu_j, \quad  (j=2,3,4), & \text{when} \; \Re\lambda >0. \label{EigenSpliting1_2}
\end{align}
\end{subequations}
Elementary computations show that $\lambda $ near 0 we can choose
\begin{align}\label{EigenSpliting1_3}
  \mu_2 = \frac{\lambda}{c+\mathsf{V}}(1+o_{\lambda}(1)) \text{ and } \mu_3 = \frac{\lambda}{ \varepsilon \mathsf{V}}(1+o_{\lambda}(1))
\end{align}
and with the above choices
\begin{align}\label{EigenSpliting1_4}
 0<\Re  \mu_2  <  \Re  \mu_3  \text{ for } \Re\lambda >0.
\end{align}
We also have
\begin{align}
   \label{eq:sum_of_mus}
   & \mu_1    + \mu_2   + \mu_3 + \mu_4  =\dfrac{2c\lambda}{c^2-K}   .
\end{align}
We will need the facts and  formulas stated in the  following elementary lemma.

\begin{lemma}\label{lem:elest}
\begin{enumerate}[label = (\roman*)]
\item We have
\begin{align}
  & \mu_4 (0, \varepsilon) =\sqrt{2\mathsf{V} \varepsilon }      \(1- \frac{3+4K}{4\mathsf{V}}  \varepsilon +     O(\varepsilon ^2)\) , \label{eq:mu40}\\&  \partial _{\lambda}\mu_4 (0, \varepsilon) =- \frac{1}{2\varepsilon}   (1+O(\varepsilon)) . \label{eq:mu41}
\end{align} 
 
  \item  For $\lambda \in \im \R \backslash \{ 0   \}$, the four eigenvalues in \eqref{EigenSpliting1_1} are distinct. 
  \item We have
  \begin{equation}\label{mu1_mu4}
   \mu  _1 \(  - \lambda , \varepsilon   \)=-\mu  _4 \(   \lambda ,\varepsilon \)  \quad \text{for } \lambda \in  \im \R.
 \end{equation}
\end{enumerate}

\end{lemma}

\begin{proof}
The starting point for  \eqref{eq:mu40} is the   equation     $d_{-}(\mu,\varepsilon )=\lambda$, see  Bae and Kwon \cite[formula (3.19)]{BK22ARMA}, that $\mu_4 (\lambda, \varepsilon)$ needs to satisfy,  so that for $\lambda =0$ and after elementary computations
\begin{align*}                 
\mu_4 ^2  = 1 - \frac{1}{c^2-K} = 1 -  \frac1{  1+2\mathsf{V}   \varepsilon   + \varepsilon  ^2}  =      2\mathsf{V}   \varepsilon   +  \( 1-4\mathsf{V}^2 \)  \varepsilon  ^2  +O( \epsilon ^3)
\end{align*}
so that taking square root we obtain  \eqref{eq:mu40}. 
 We obtain \eqref{eq:mu41} by implicit differentiation of $d_{-}(\mu,\varepsilon )=\lambda$  at $\lambda =0$ and   by
  \begin{equation*}
    1=  -\frac{\mu_4^2\partial_\lambda \mu_4}{c(1-\mu_4^2)^2} =  \mu _4^2 (0)   \partial _{\lambda}\mu_4   (0 )      \( -\frac{1}{ \mathsf{V}} + O(\varepsilon) + O( \mu^2_4(0)) \) .
  \end{equation*}
To prove the second statement, notice that for $\im \tau _j(\tau )=    \mu _j(\in \im \R  ) $  and for the choice in    \eqref{EigenSpliting1_3}  we have 
\begin{align*}&
  \tau _2 \(   c + \sqrt{K+ \frac{1}{1+\tau _2 ^2}} \) =\tau  \text{ and}\\& \tau _3 \(   c - \sqrt{K+ \frac{1}{1+\tau _3 ^2}} \) =\tau
\end{align*}
so that subtracting, if $\tau _2=\tau _3$  we have the following, possible only for $\tau _2=\tau _3=0$,  
\begin{align*}
  \tau _2     \sqrt{K+ \frac{1}{1+\tau _2 ^2}} + \tau _3   \sqrt{K+   \frac{1}{1+\tau _3 ^2}}=   2  \tau _2     \sqrt{K+ \frac{1}{1+\tau _2 ^2}} =0 .
\end{align*}
Next, for $\lambda \in \im \R $, we know from \eqref{Charact1} that  $\mu_1 (  -\lambda )$ and $-\mu  _4 (   \lambda   )$ are both solutions of
  \begin{align*}
    \mu  \left(c - \sqrt{\frac{1}{1-\mu^2 }+ K}\, \right) = -\lambda .
  \end{align*}
  In fact, $\mu_1,\mu_3,\mu_4$ satisfy $d_-(\mu,\varepsilon)=\lambda$ (see \cite{BK22ARMA}, Proof of Lemma 3.3) and \eqref{EigenSpliting1_1} holds. Hence, $\mu_1(-\lambda) \neq \mu_3(-\lambda)$ and $\mu_1(-\lambda) \neq \mu_4(-\lambda)$, and we obtain  \eqref{mu1_mu4}. 
\end{proof}

An elementary   computation shows that for $\mathbf{V}$ the matrix in \eqref{eq:defVW}
\begin{align}\label{eq:detv}
  \det \mathbf{V}  = -4 \mu_1^2  (1+O(\varepsilon)) \text{ at }\lambda =0.
\end{align}  Since   by  Lemma \ref{lem:elest} the  four eigenvalues are distinct  for  $\lambda \in \im \R \backslash \{ 0   \}$, it follows
\begin{equation}\label{eq:detvlambda}
   \det  \mathbf{V}  \neq 0  \text{ for all  }\lambda \in \im \R .
\end{equation}
It is elementary that solutions of
\begin{align}
  \label{eq:linequn1} \( \frac{d }{dx} - A^{\infty} ( \lambda,\varepsilon) \)
\mathbf{Y} ^{(0)}  = 0 \text{ for } \lambda \in \overline{\C_+}
\end{align}
form the vector space $\Span \{ \mathbf{Y} ^{(0)}_{j}: 1\le j \le 4   \} $  with
\begin{align*}
  \mathbf{Y} ^{(0)}_{j} (x) := e^{\mu _j (\lambda , \varepsilon )x} \mathbf{v}_j (\lambda , \varepsilon )
\end{align*}
and that  the
solutions of
\begin{align}
  \label{eq:linequn21} \( \frac{d }{dx} + \( A^{\infty} ( \lambda,\varepsilon) \) ^\intercal \)
\mathbf{Z} ^{(0)}  = 0 \text{ for } \lambda \in \overline{\C_+}
\end{align}
form the vector space $\Span \{ \mathbf{Z} ^{(0)}_{j}: 1\le j \le 4   \} $  with
\begin{align*}
  \mathbf{Z} ^{(0)}_{j} (x) := e^{-\mu _j (\lambda , \varepsilon )x} w_j (\lambda , \varepsilon ).
\end{align*}
We look for Jost functions of
\begin{equation}\label{eq:lineq1}
  \( \frac{d }{dx} - A(x,\lambda,\varepsilon) \)
\mathbf{Y}  =0
\end{equation}
  that is for appropriate solutions of   \eqref{eq:lineq1} of the form
$  f_j (x,\lambda , \varepsilon ) := e^{\mu _j (\lambda , \varepsilon )x}   \mathbf{m}_j (x,\lambda , \varepsilon ) $
with 
\begin{align}
   \label{eq:limm1}   & \lim _{x\to +\infty}\mathbf{m}_1 (x,\lambda , \varepsilon )=\mathbf{v}_1(\lambda , \varepsilon ) \text{ and with}\\&  \label{eq:limmj} \lim _{x\to -\infty}\mathbf{m}_j (x,\lambda , \varepsilon )=\mathbf{v}_j(\lambda , \varepsilon ) \text{ for }j\ge 2. 
\end{align}

\begin{lemma}\label{lem:estimate_f1}
    There exists a  solution  $f_1(x , \lambda , \varepsilon ) = e^{\mu_1 x} \mathbf{m}_1(x , \lambda , \varepsilon )$    of  \eqref{eq:lineq1}
    such that for any $m,n\ge 0$  there is a constant $C_{mn}(\varepsilon)$ such that   for all $\lambda \in \overline{\C _+} $  with $|\lambda |\le 1$    and $x\in \R$
    \begin{equation}\label{eq:m1_estimate_partial}
        \left |     \partial_\lambda^n \partial_x^{m} \(   \mathbf{{m}}_1(x , \lambda , \varepsilon  ) -\mathbf{v}_1 ( \lambda , \varepsilon  ) \)  \right |  \leq C_{mn}  (\varepsilon)      \sum _{i=1,2}    \int _{x}^{+\infty } |q_i  (x', \varepsilon ) | dx' \text{ . }
    \end{equation}

\end{lemma}
\

We skip the elementary and standard proof of Lemma \ref{lem:estimate_f1}, referring to \cite[Sect. 7]{Bealsbook1988}, to the classical \cite{DT1979} and to \cite[Lemma 8.1]{CM25D1}. We only point out
 that   $\mathbf{m}_1$ solves the
   Volterra equation
   \begin{equation}\label{eq:integral_eq_m1}
          \mathbf{m}_1(x , \lambda , \varepsilon  )  = \mathbf{v}_1 ( \lambda , \varepsilon  )  - \sum _{j=1}^{4}  \mathbf{v}_j  ( \lambda , \varepsilon  )    \int_x^{+\infty}  e^{(\mu_{ {j}}   -\mu_{1} ) (x-y)}\< q(y,\lambda , \varepsilon  ) \mathbf{m}_1(y , \lambda , \varepsilon  ) , \mathbf{w} _j   \> _{\C ^4}    d y .
    \end{equation}

  \begin{lemma}\label{lem:estimate_f4}
    There exists a  solution  $f_4(x , \lambda , \varepsilon ) = e^{\mu_4 x} \mathbf{m}_4(x , \lambda , \varepsilon )$    of  \eqref{eq:lineq1}
    such that for any $m,n \ge 0$  there is a constant $C_{mn}(\varepsilon)$ such that   for all $\lambda \in \overline{\C _+} $  with $|\lambda |\le 1$   and $x\in \R$
    \begin{equation}\label{eq:m4_estimate_partial}
        \left |     \partial_x^{n}    \partial_\lambda^m \(   \mathbf{m}_4(x , \lambda , \varepsilon  ) -\mathbf{v}_4 ( \lambda , \varepsilon  )  \)  \right |  \leq C_{mn} (\varepsilon) \sum _{i=1,2}       \int _{-\infty }^{x }|q_i  (x', \varepsilon ) | dx' \text{  .}
    \end{equation}

\end{lemma}
It is  completely analogous to Lemma \ref{lem:estimate_f1}  but with
 \begin{equation}\label{eq:integral_eq_m4}
          \mathbf{m}_4(x , \lambda , \varepsilon  )  = \mathbf{v}_4 ( \lambda , \varepsilon  )  + \sum _{j=1}^{4}  \mathbf{v}_j  ( \lambda , \varepsilon  )    \int_{-\infty} ^x  e^{(\mu_{ {j}}   -\mu_{4} ) (x-y)}\< q(y,\lambda , \varepsilon  ) \mathbf{m}_4(y , \lambda , \varepsilon  ) , \mathbf{w} _j   \> _{\C ^4}    d y .
    \end{equation}

    \begin{lemma}\label{lem:symjost1} Consider the diagonal matrix $D_{3,1}=\diag (1,1,1,-1)$. Then for $\lambda \in \im \R$  we have
    \begin{equation}\label{eq:symjost11}
        f_1(x , -\lambda , \varepsilon )=D_{3,1} f_4(-x , \lambda , \varepsilon ).
    \end{equation}

    \end{lemma}
    \begin{proof}
     We omit the $\varepsilon$ for simplicity. Utilizing the fact that $( n_c,u_c,\phi _c)$ is even in $x$, so that $L_c $ is even and $L' _c $ is odd, it is elementary to see that
      \begin{align*}
        \frac{d}{dx} \(  f_4(-x ,  \lambda    )  \) &=\left(
\begin{array}{c|c}
-L^{-1}_c(x)L'_c(x) &  L^{-1}_c(x) \left (\begin{array}{cc}
						0 & 0 \\
						0 & 1
					\end{array}\right ) \\
		\hline
\begin{array}{cc}
0 & 0 \\
1 & 0
\end{array}
& \begin{array}{cc}
0 & -1 \\
-e^{\phi_c(x)} & 0
\end{array}
\end{array}
\right) f_4(-x ,  \lambda    )\\&  -  \lambda\left(
\begin{array}{c|c}
-L^{-1}_c(x) & \mathbf{0}_2 \\
\hline
\mathbf{0}_2 & \mathbf{0}_2
\end{array}
\right) f_4(-x ,  \lambda    )
      \end{align*}
  which leads to,  see also \cite{BK22ARMA}, Remark 3  p.275,
 \begin{align*}
     \frac{d}{dx}\(  D_{3,1}f_4(-x ,  \lambda    )  \) = \big(A_1(x) -\lambda A_2(x) \big)D_{3,1}f_4(-x ,  \lambda    ).
 \end{align*}
By \eqref{mu1_mu4} and by comparing the asymptotic behaviors of $f_1$ and $f_4$, we obtain \eqref{eq:symjost11}.

    \end{proof}

    \begin{lemma} \label{lem:jostlambda0} There are $\alpha ( \varepsilon )\neq 0$  and $\beta ( \varepsilon ) $
    such that for $j=1$ and 4 we have   
   \begin{align}
       \label{eq:jostlambda01}&  f_j(x,0,\varepsilon )=(-1) ^{j}\alpha ( \varepsilon )   \Xi _1[c]   \text{ and }\\  \label{eq:jostlambda02}&  \partial _\lambda  f_j(x,0,\varepsilon )=(-1) ^{j+1}\alpha ( \varepsilon )  \Xi _2[c]     +\beta ( \varepsilon ) \Xi _1[c] .
   \end{align} 
\end{lemma}

  \begin{proof}
  We omit the $\varepsilon$ for simplicity. Identity \eqref{eq:jostlambda01}  for $j=1$   for some $\alpha ( \varepsilon ) \neq 0$
    follows immediately  {comparing asymptotic behaviors.}  Identity \eqref{eq:jostlambda01}  for $j=4$  is obtained    from the case  $j=1$  using \eqref{eq:symjost11}. 
     To prove \eqref{eq:jostlambda02}, we differentiate \eqref{eq:lineq1} in $\lambda$, and we see that
  \begin{equation}\label{eq:jostlambda03}
     \( \frac{d }{dx} - A_1(x  ) \)\partial _\lambda  f_4(x,0  ) = A_2(x  ) f_4(x,0  ) = A_2(x  )\alpha \Xi_1[c],
  \end{equation}
  where we have used \eqref{eq:jostlambda01}.
  Multiplying the second identity in  \eqref{eq:bigxi1} by $\alpha$, and then adding the resulting identity and \eqref{eq:jostlambda03}, we have
  \[
  \( \frac{d }{dx} - A_1(x  ) \) \big( \partial _\lambda   f_4(x,0  ) + \alpha \Xi_2[c] \big) = 0.
  \]
  By comparing the asymptotic behavior, we conclude
   \eqref{eq:jostlambda02} for $j=4$ for some $\beta (\varepsilon )$. From \eqref{eq:symjost11} we have
\begin{align*}
  \partial _\lambda  f_1(x,0  )
  =  - D_{3,1}  \partial _\lambda  f_4(-x,0  )
  =  \alpha D_{3,1} 
  \begin{pmatrix} 
  \partial _cn_c  (-x) \\
  \partial _cu_c  (-x) \\
  \partial _c\phi_c  (-x) \\
  \partial _c\psi_c  (-x)
  \end{pmatrix}
 + \beta D_{3,1} 
 \begin{pmatrix}
   n_c ' (-x) \\
   u_c'  (-x) \\
   \phi_c'  (-x) \\
   \psi_c ' (-x)
 \end{pmatrix}
\end{align*}
   which yields  \eqref{eq:jostlambda02} for $j=1$ by the fact that  $ \partial _c ^{n}( n_c,u_c,\phi _c)$ are for $n=0,1$  even in $x$ while $\partial _c ^{n}\psi _c$ is odd and that $  ( n_c',u_c',\phi _c')$ is odd while $\psi_c '$  is even.  

  \end{proof}

We have now an  analogue of \cite[Lemma 8.3]{CM25D1}. The following functions do not satisfy  good bounds  for $x\in \R$ and they are different from the functions in  \eqref{eq:limmj}, but they will be nonetheless useful.
\begin{lemma}\label{lem:estimate_tildef2} For $k=2,3$ there exist  solutions  $  \widehat{{f}}_k(x , \lambda , \varepsilon )= e^{\mu_kx}  \widehat{{\mathbf{m}}}_k(x , \lambda , \varepsilon )) $  and $ x_0(\varepsilon)\in \R $ such that  for all for all $\lambda \in \overline{\C _+} $  with $|\lambda |\le 1$   we have
\begin{align}\label{eq:estimate_tildef21}&
\|   \widehat{\mathbf{m}} _k(\cdot  ,\lambda , \varepsilon  )\| _{L^\infty\( (  -\infty,  x_0(\varepsilon) )  \) } \le 2
\text{ and }\\&
\lim_{x\to-\infty} \widehat{\mathbf{m}}_k(x ,\lambda , \varepsilon   ) =\mathbf{v}_k \ .\label{eq:estimate_tildef22}
\end{align}
There are  constants $  C_{\alpha\beta} (\varepsilon ,x_0)$ such that for $\alpha +\beta \ge 1  $
\begin{align} \label{eq:estimate_tildef232}
  \|  \partial_x^{n} \partial_\lambda^m   \widehat{ {\mathbf{m}}}_k(\cdot  , \lambda , \varepsilon )\| _{L^\infty\( (  -\infty, x_0(\varepsilon)  )  \) } \le C_{mn}(\varepsilon)     \text{  } .
\end{align}
Finally, there are constants  $  C_{mn }  (\varepsilon ,x_0) $ such that for $x\le x_0$
\begin{align}\label{eq:estimate_tildef233} &  |   \partial_x^{n} \partial_\lambda^m \( \widehat{ {\mathbf{m}}}_k(x  , \lambda , \varepsilon ) -\mathbf{v}_k \)    |\le C_{ mn }(\varepsilon)    e^{-  \sqrt{\varepsilon} \alpha  |x|}  .
\end{align}

\end{lemma}

We skip the proof,   similar to \cite[Lemma 8.3]{CM25D1}. 
 For $\Re \lambda >0$, by the  general theory  in   \cite{Bealsbook1988}, it is possible to prove the existence of the Jost functions in   \eqref{eq:limmj} with   $\mathsf{m}_k (\cdot , \lambda , \varepsilon ) \in L^\infty (\R )$ for $k=2,3$. 
However, for our purposes we prefer to follow the discussion in  \cite[Sect. 8]{CM25D1}.
In order to do this we introduce some other Jost functions.  The proof of the following   is standard.

\begin{lemma}\label{lem:extrajost} For all $j=1,2,3,4$  there exist solutions  $ F_j(x , \lambda , \varepsilon )= e^{\mu_j x} M_j(x , \lambda  , \varepsilon)  $   of \eqref{eq:lineq1}   such that for any  $\alpha, \beta \ge 0$  there is a constant $C_{\alpha\beta  }(\varepsilon)$ such that  for  $\Re \lambda \ge 0  $ and $|\lambda |\le 1$  we have
\begin{align} \label{lem:extrajost1}&
   |    \partial_\lambda^\beta \partial_x^{\alpha} \( M_1(x ,  \lambda  , \varepsilon )- \mathbf{v}_1(  \lambda , \varepsilon ) \)  |  \leq  C_{\alpha, \beta}(\varepsilon)    \sum _{i=1,2}       \int _{-\infty }^{x }|q_i  (x', \varepsilon ) | dx' \text{  and}\\& |    \partial_\lambda^\beta \partial_x^{\alpha} \( M_j(x , \lambda  , \varepsilon  )-\mathbf{v}_j(  \lambda , \varepsilon )\)  |  \leq  C_{\alpha, \beta}(\varepsilon)    \sum _{i=1,2}       \int ^{+\infty }_{x }|q_i  (x', \varepsilon ) | dx' \text{  for  all $j=2,3,4$.}  \label{lem:extrajost234}
\end{align}
\end{lemma}
\qed

There are also Jost solutions analogous to those of Lemmas \ref{lem:estimate_f1},  \ref{lem:estimate_f4}, \ref{lem:estimate_tildef2} and \ref{lem:extrajost}    for    system \begin{align}
  \label{eq:lineq1ad} \( \frac{d }{dx} + A^\intercal (x,\lambda,\varepsilon) \)
\mathbf{U}  = 0 \text{ for } \lambda \in \overline{\C_+} .
\end{align} 
We single out here solutions to  \eqref{eq:lineq1ad} which behave like 
\begin{align} &
  \label{eq:g1}   g_1(x, \lambda , \varepsilon )\sim e^{-\mu _1(  \lambda , \varepsilon ) x} \mathbf{w}_1(  \lambda , \varepsilon ) \text{ as  }   x\to -\infty   \text{ and}\\ &
  \label{eq:g4}   g_4(x, \lambda , \varepsilon )\sim e^{-\mu _4(  \lambda , \varepsilon ) x} \mathbf{w}_4(  \lambda , \varepsilon ) \text{ as  }   x\to +\infty .
\end{align}
It is elementary that   $x\to \< \mathbf{u}(x) , \mathbf{v}(x)\> _{\C ^4} $ is constant if $\mathbf{u}$ solves  \eqref{eq:lineq1}  and $\mathbf{v}$ solves  \eqref{eq:lineq1ad}.
 Then  
\begin{align}
  \label{eq:evans} D(\lambda , \varepsilon  ): = \<  f_1(x, \lambda , \varepsilon ),g_1(x, \lambda , \varepsilon )\> _{\C^4}
\end{align}
is constant in $x$ and, since $f_1(x, \lambda , \varepsilon )\sim e^{ \mu _1  x} \mathbf{v}_1 $ as $x\to -\infty$ with $  \<  \mathbf{v}_1,\mathbf{w}_1\>=1$, it
 is the Evans function of  \eqref{eq:lineq1},   Pego and Weinstein \cite [Definition 1.8] {PegoWei1}.  By   Bae and Kwon \cite {BK22ARMA}   for $\varepsilon$
 small enough
 \begin{align}
  \label{eq:derevans0}D(0, \varepsilon)= \partial_{\lambda} D  (0, \varepsilon) =0 \text{ and }\partial_{\lambda} ^2D (0, \varepsilon)\neq 0
\end{align}
and      that $D(\lambda , \varepsilon)\neq 0$ for any nonzero $\lambda \in \overline{\C _+}$.
By Pego and Weinstein \cite[ Formula (1.28)]{PegoWei1}
\begin{align}\label{eq:limm1-infty}
   \lim _{x\to -\infty}  e^{-\mu_ {1}x}  f_1(x, \lambda , \varepsilon ) =D(\lambda , \varepsilon )  \mathbf{v}_1(\lambda , \varepsilon).
\end{align}
Now, from $ \widehat{{f}}_2(\cdot ,\lambda  , \varepsilon ), \widehat{f}_3(\cdot  ,\lambda , \varepsilon) , {f}_4(\cdot  ,\lambda , \varepsilon) \in \mathrm{Span}\{f_1(\cdot  ,\lambda , \varepsilon),F_2(\cdot  ,\lambda , \varepsilon),F_3(\cdot  ,\lambda , \varepsilon) , F_4(\cdot  ,\lambda , \varepsilon)\}$  we have
\begin{align}
      \widehat{{f}}_j(x,\lambda  , \varepsilon )&=   c_{j1}(\lambda   , \varepsilon )f_1(x,\lambda  , \varepsilon) + \sum _{k=2}^{4} c_{jk}(\lambda   , \varepsilon )F_k(x,\lambda  , \varepsilon) \text{ for $j=2,3$ and}\label{eq:connect2}\\
    f_4 (x,\lambda  , \varepsilon )&=   c_{41}(\lambda   , \varepsilon )f_1(x,\lambda  , \varepsilon) + \sum _{k=2}^{4} c_{4k}(\lambda   , \varepsilon )F_k(x,\lambda  , \varepsilon).\label{eq:connect3}
\end{align}

\begin{lemma}\label{lem:c23vanish} The
$c_{jk}(\lambda  , \varepsilon )$ ($j=2,3,4$, $k=1,2,3,4$)   in \eqref{eq:connect2} are analytic in $\lambda$ and jointly continuous in $( \lambda  , \varepsilon)$.
Furthermore, we have $c_{44}(\lambda   , \varepsilon )=D(-\lambda  , \varepsilon)$ for $\lambda \in \im \R$ and $c_{j4}(0 , \varepsilon )=0$  for $j=2,3$.
\end{lemma}

\begin{proof}
     We skip the standard  proof of the regularity  of  $c_{jk}(\lambda  , \varepsilon )$, see \cite{CM25D1}.  By the symmetry \eqref{eq:symjost11}, \eqref{mu1_mu4},  and by \eqref{eq:limm1-infty}, we have that for
$  \lambda \in \im \R $,
\begin{align*}  
  c_{44}(\lambda   , \varepsilon )\mathbf{v}_4(  \lambda , \varepsilon )&=\lim  _{x\to +\infty}  \mathbf{m}_4(x,\lambda  , \varepsilon )  =    \lim  _{x\to -\infty}  D_{3,1} \mathbf{m}_1(x,-\lambda    , \varepsilon ) \\& \nonumber =D(-\lambda  , \varepsilon  )D_{3,1} \mathbf{v}_1( - \lambda , \varepsilon ) =  D(-\lambda  , \varepsilon  ) \mathbf{v}_4(  \lambda , \varepsilon ).
\end{align*}
This yields  $c_{44}(\lambda   , \varepsilon )=D(-\lambda  , \varepsilon)$. Next, for $j=2,3$  we have 
\begin{align*}
   \<  \widehat{f}_j(x, \lambda , \varepsilon ),g_4(x, \lambda , \varepsilon )\> _{\C^4} =\lim _{x\to + \infty}  \<  \widehat{f}_j(x, \lambda , \varepsilon ),g_4(x, \lambda , \varepsilon )\> _{\C^4} = c_{j4}(\lambda   , \varepsilon ) .
\end{align*}
On the other hand,   $g_4(x, 0 , \varepsilon )   $    is proportional to $E_2[c]$.  This follows  from \eqref{eq:bigeta1} and the fact that $E_2 [c]$ is up to constant factor  the only solution of \eqref{eq:lineq1ad} with $E_2 [c]\xrightarrow{x \to - \infty}0$ (see also the comment under \eqref{eq:adjker} and the fact that there are solutions of \eqref{eq:lineq1ad} blowing up as $x \to - \infty$  similar to those in    Lemma  \ref{lem:extrajost}). 
Then  we also have 
\begin{align*}
   c_{j4}(0 , \varepsilon ) =  \lim _{x\to -\infty}  \<  \widehat{f}_j(x, 0 , \varepsilon ),g_4(x, 0, \varepsilon )\> _{\C^4} =0.
\end{align*}
\end{proof}
Next, we set 
\begin{align}\label{eq:defc0}
    c_j(\lambda , \varepsilon):=\frac{\lambda c_{j4}(\lambda , \varepsilon )}{D(-\lambda , \varepsilon)}  \text{ for }j=2,3
\end{align}
and we define 
\begin{align}
  \label{eq:deff2}
  f_{j}(x,\lambda  , \varepsilon ) := \lambda  \widehat {f}_{j}(x,\lambda    , \varepsilon  ) - c_j(\lambda , \varepsilon) {f}_{4}(x,\lambda  , \varepsilon ) \text{ for $j=2,3$.}
\end{align}
Then, we claim  we have
\begin{equation}\label{eq:f23adjusted}
  f_j(\cdot ,\lambda  , \varepsilon )\in \mathrm{Span}\{f_1(\cdot ,\lambda  , \varepsilon ),F_2(\cdot ,\lambda  , \varepsilon ) ,F_3(\cdot ,\lambda  , \varepsilon ) \}
\end{equation}
Indeed, from \eqref{eq:connect2}, \eqref{eq:connect3}, \eqref{eq:defc0} and Lemma \ref{lem:c23vanish},
\begin{align*}
    f_j(x,\lambda , \varepsilon )=&\left(\lambda c_{j1}(\lambda , \varepsilon )-c_j(\lambda , \varepsilon )c_{41}(\lambda , \varepsilon)\right) f_1(x,
    \lambda , \varepsilon )+   \sum _{k=2,3}  \left(\lambda c_{jk}(\lambda)-c_j(\lambda)c_{4k}(\lambda)\right) F_k(x,\lambda , \varepsilon)\\&
    +\cancel{\left(\lambda c_{j4}(\lambda , \varepsilon   )-c_j(\lambda   , \varepsilon  )D(-\lambda , \varepsilon    )\right)} F_4(x,\lambda , \varepsilon    ).
\end{align*}
 For $  \mathbf{m}_j(x,\lambda , \varepsilon ) := e^{-x\mu _j (\lambda  , \varepsilon)} f_j(x,\lambda , \varepsilon ) $ with $j=2,3$ we have the following.

\begin{lemma}
  \label{lem:bdm2}  For any preassigned $a_0>0$  there are constants     $C_{\alpha\beta}(\varepsilon)>0$ such that \begin{align} \label{eq:estm21}
   |  \partial_x^\alpha  \partial_\lambda^\beta   \mathbf{m}_j(x  , \lambda  , \varepsilon) |  \le C _{\alpha\beta}  (\varepsilon)   \text{ for all $x\in \R$ and all $\lambda \in \(  -\im a_0, \im a_0\) $ and for $j=2,3$}.
\end{align}
\end{lemma}
\begin{proof}
   For $x\le 0$ this follows from Lemmas \ref{lem:estimate_f4}  and \ref{lem:estimate_tildef2} and from \eqref{eq:f23adjusted} and Lemmas \ref{lem:estimate_f1} and  \ref{lem:extrajost}  for $x\ge 0$.
\end{proof}
Let us set  now
 \begin{align}
   \label{eq:deftildew}\widehat{W}(x  , \lambda , \varepsilon   ):= \det \(    {f}_1(x  , \lambda , \varepsilon   ), \widehat{f}_2(x  , \lambda , \varepsilon   ), \widehat{f}_3(x  ,\lambda , \varepsilon  ) ,{f}_4(x  , \lambda , \varepsilon   ) \) .
 \end{align}
 For the $J$ in  \eqref{Useful_Id3} and the $\mathbf{V}$  in \eqref{eq:defVW}, we claim
 \begin{align}\label{eq:wronsk1}  \widehat{W}(x  , \lambda , \varepsilon   ) =    J(x) e^{2\lambda \int _0 ^x  \frac{c-u_c(y)}{J(y)} dy}     D(\lambda , \varepsilon ) \frac{\det \mathbf{V}}{c^2-K} .
\end{align}
 To see this notice that for some $ C(\lambda , \varepsilon )$ constant in  $x$ we have
 \begin{equation*}
   \widehat{W}(x  , \lambda , \varepsilon   )  =   C(\lambda , \varepsilon )  J(x) e^{2\lambda \int _0 ^x  \frac{c-u_c(y)}{J(y)} dy} .
 \end{equation*}
For $*$ the Hodge star operator in  \eqref{eq:hodgestar}, for $j=1$ we have    \small
 \begin{equation*}
   \frac{d}{dx}  \<  {f}_j(x  , \lambda , \varepsilon   ), Z_4(x, \lambda , \varepsilon   )\> _{\C^4} \equiv 0 \text{ for } Z_i(x, \lambda , \varepsilon   ):=\frac{*(  \widehat{f}_2(x, \lambda , \varepsilon   )\wedge \widehat{f} _3(x, \lambda , \varepsilon   )\wedge f_i(x, \lambda , \varepsilon   ))}{ J(x) e^{2\lambda \int _0 ^x  \frac{c-u_c(y)}{J(y)} dy}}  \ , \ i=1,4.
 \end{equation*}\normalsize
Since the above equality is also true for $j\neq 1$ we conclude that
  $Z_4(x, \lambda , \varepsilon   )$ solves \eqref{eq:lineq1ad}.  By \eqref{eq:sum_of_mus} the following asymptotics is elementary for $x\to -\infty$,
\small
\begin{equation} \label{eq:oscz4}
   Z_4(x, \lambda , \varepsilon   ) \sim     \frac{ *( \mathbf{v}_2\wedge \mathbf{v}_3\wedge \mathbf{v}_4)}{-(c^2-K)} e^{\(\mu _1(\lambda , \varepsilon   )+\mu _2(\lambda , \varepsilon   )+\mu _2(\lambda , \varepsilon   ) -\frac{2\lambda c}{c^2-K} \)x }  =  \frac{ *( \mathbf{v}_2\wedge \mathbf{v}_3\wedge \mathbf{v}_4)}{ -(c^2-K) }e^{-\mu _1(  \lambda , \varepsilon ) x} .
\end{equation}\normalsize
We claim that for    small $\lambda$ and   $(i_1,i_2,i_3,i_4)$ a permutation of $( 1, 2, 3, 4)$, we have
\begin{equation}\label{eq:strar3vec}
  *( \mathbf{v} _{i_1}\wedge \mathbf{v}_{i_2}\wedge \mathbf{v}_{i_3}) =   \det (\mathbf{v}_{i_1}, \mathbf{v}_{i_2}, \mathbf{v}_{i_3}, \mathbf{v}_{i_4}) \mathbf{w}_{i_4} .
\end{equation}
Indeed,  by $** =- 1$  in $\wedge ^3 (\C ^4)$
\begin{align*}&
  \<    \mathbf{v}_{i }, *( \mathbf{v} _{i_1}\wedge \mathbf{v}_{i_2}\wedge \mathbf{v}_{i_3}) \> _{\C ^4}=  - \det ( \mathbf{v}_{i } ,\mathbf{v}_{i_1}, \mathbf{v}_{i_2}, \mathbf{v}_{i_3}  ) 
= \left\{
\begin{array}{ll}
 0, & \hbox{ if } i=i_1,i_2,i_3 \\
\det (\mathbf{v}_{i_1}, \mathbf{v}_{i_2}, \mathbf{v}_{i_3}, \mathbf{v}_{i_4}) , & \hbox{if }i=i_4.\end{array}
\right.
\end{align*}
Since   $  ( \mathbf{w}_i ) _{i=1}^{4}$ for small $\lambda $ is a basis and in view of \eqref{LRVec_A}--\eqref{eigenvec_A2}  we conclude \eqref{eq:strar3vec}.
Hence the following, which implies  \eqref{eq:wronsk1},
\begin{equation}\label{eq:z4g1}
   Z_4(x, \lambda , \varepsilon   ) =  \frac{\det \mathbf{V}}{c^2-K} g_1(x, \lambda , \varepsilon ) .
\end{equation}
Notice that  $Z_1$  similarly satisfies  \eqref{eq:lineq1ad}.  We have  for some $\widehat{{\alpha}}(\varepsilon) \neq 0$
\begin{align}\label{eq:Z14at0}
   Z_1(x, 0 , \varepsilon   )=  -Z_4(x, 0 , \varepsilon   )  = - \widehat{{\alpha}}(\varepsilon) E_2 [c] .
\end{align}
Differentiating in $\lambda$  in \eqref{eq:lineq1ad} for $i=1,4$  we get
\begin{align*}
    \( \frac{d }{dx} + A^\intercal  _1(x, \varepsilon) \) \partial _{\lambda}  Z_i(x, 0 , \varepsilon   ) =-  A^\intercal  _2\(x, \varepsilon  \)    Z_i(x, 0 , \varepsilon   ) = -(-1)^{i} \widehat{{\alpha}}(\varepsilon)  A^\intercal  _2\(x, \varepsilon  \)   E_2 [c]  .
\end{align*}
Since by the definition of $Z_i$ and by Lemmas  \ref{lem:jostlambda0} and \ref {lem:estimate_tildef2}  we have $\partial _{\lambda}  Z_i(x, 0 , \varepsilon   )\xrightarrow{x\to -\infty}0$, it follows that there exist constants $\widehat{\beta}_i(\varepsilon)$ such that
\begin{align}\label{eq:partlambdaz}
  \partial _{\lambda}  Z_i(x, 0 , \varepsilon   ) = (-1)^{i} \widehat{\alpha}(\varepsilon)  E _1[c] +   \widehat{\beta}_i(\varepsilon) E _2[c].
\end{align}
We claim now that,  for ${\alpha}(\varepsilon)$  and ${\beta}(\varepsilon)$ the constants in \eqref{eq:jostlambda02},  we have
\begin{align} \label{eq:imprelat}
  -2\beta (\varepsilon)   \widehat{\alpha}(\varepsilon) +   {\alpha}(\varepsilon) \(   \widehat{\beta}_1(\varepsilon)  +\widehat{\beta}_4(\varepsilon)  \) =0.
\end{align}
This follows using \eqref{eq:jostlambda01}, \eqref{eq:jostlambda02},     \eqref{eq:Z14at0}  and   \eqref{eq:partlambdaz},  
\begin{align*} & \(   \widehat{\beta}_1(\varepsilon)  +\widehat{\beta}_4(\varepsilon)  \)  E _2[c] =   \partial _{\lambda}  Z_1(x, 0 , \varepsilon   ) + \partial _{\lambda}  Z_4(x, 0 , \varepsilon   )\\& = \frac{*(  \widehat{f}_2(x, 0 , \varepsilon   )\wedge \widehat{f} _3(x, 0 , \varepsilon   )\wedge  \partial _\lambda (f_1(x, 0 , \varepsilon   ) +f_4(x, 0 , \varepsilon   ) )   )}{ J(x)  } \\& = 2\frac{\beta (\varepsilon)}{{\alpha}(\varepsilon) }   \frac{*(  \widehat{f}_2(x, 0 , \varepsilon   )\wedge \widehat{f} _3(x, 0 , \varepsilon   )\wedge  f_4(x, 0 , \varepsilon   )     )}{ J(x)  } =2\frac{\beta (\varepsilon)}{{\alpha}(\varepsilon) }  Z_4(x, 0 , \varepsilon   ) =  2\frac{\beta (\varepsilon)   \widehat{{\alpha}}(\varepsilon)}{{\alpha}(\varepsilon) }    E_2 [c] .
\end{align*}
The proof of Proposition \ref{prop:Proposition 1.1} is completed by the following claim.
\begin{claim} \label{claim:noeig}For $\varepsilon >0$
 small enough if  $\lambda \in \im \R $ is a nonzero eigenvalue then $D(\lambda , \varepsilon   )=0$.
\end{claim}
\begin{proof}
 If $\lambda \in \im \R $ is an eigenvalue, then we have a solution of \eqref{eq:resoleq} for $\widetilde{F}=0$  with $\widetilde{U}\in  L^2(\R , \C ^2)$. For $ \widetilde{U}=(\dot n , \dot u ) ^\intercal$
and defining $U^\intercal := (\widetilde{U} ^\intercal , \dot \phi  , \dot \psi ) $ using \eqref{Poisson_ODE}, we obtain a solution $ {U}\in  L^2(\R , \C ^4)$  of \eqref{eq:lineq1}. If a nonzero $U$ of this type, then  the left hand side in \eqref{eq:wronsk1} needs to equal 0. In Bae and Kwon \cite[p. 273]{BK22ARMA} it is remarked
that when  the $(\mu _j) _{j=1}^4$ are distinct the $(\mathbf{v} _j) _{j=1}^4$  are a basis in $\C ^4$. From Bae and Kwon \cite[Lemma 3.3]{BK22ARMA} the only two possible coincident
roots are $\mu _2(\lambda , \varepsilon   )$ and $\mu _3(\lambda , \varepsilon   )$.  From $d _+ (\mu _2)= \lambda = d _- (\mu _3)$  it is easily seen that  $\mu _2(\lambda , \varepsilon   )\neq \mu _3(\lambda , \varepsilon   )$ for any nonzero  $\lambda \in \im \R $. Hence from   \eqref{eq:wronsk1} we conclude that if  $\lambda \in \im \R $ is an eigenvalue then  $ D(\lambda , \varepsilon   )=0$.
\end{proof}

Set now for $i=2,3$
\begin{align}\label{eq:defhatZ12}
  \widehat{Z}_i(x, \lambda , \varepsilon   )&:=\frac{*(  \widehat{f}_i(x, \lambda , \varepsilon   )\wedge  {f} _1(x, \lambda , \varepsilon   )\wedge f_4(x, \lambda , \varepsilon   ))}{ J(x) e^{2\lambda \int _0 ^x  \frac{c-u_c(y)}{J(y)} dy}}   \text{ and }   \\  {Z}_i(x, \lambda , \varepsilon   )&:=\frac{*(   {f}_i(x, \lambda , \varepsilon   )\wedge  {f} _1(x, \lambda , \varepsilon   )\wedge f_4(x, \lambda , \varepsilon   ))}{ J(x) e^{2\lambda \int _0 ^x  \frac{c-u_c(y)}{J(y)} dy}} .\label{eq:defZ12}
\end{align}

\begin{lemma}\label{lem:b13at0}
  The following facts are true.
\begin{description}
  \item[i]   We have $\partial ^{a}_\lambda  \widehat{Z}_i(x, 0 , \varepsilon   )=0$ for $a=0,1 $   and for both $i=2,3$.
  \item[ii] For any      $n, m \ge 0$  there is a constant $C_{n,m}$ such that   for all   $\lambda \in \im \R $    and $x\in \R$
    \begin{equation}\label{eq:b13_estimate_partial}
        \left |     \partial_x^{n}    \partial_\lambda^m \(    e^{ \mu _3 (\lambda , \varepsilon )x }    {Z}_2(x, \lambda , \varepsilon   )   \)  \right |  + \left |     \partial_x^{n}    \partial_\lambda^m \(    e^{ \mu _2 (\lambda , \varepsilon )x }    {Z}_3(x, \lambda  , \varepsilon   )   \)  \right |  \leq C_{n,m}      .
    \end{equation}
  \item[iii]   There exists    constants $a_0^{(i)} , b_0^{(i)} ,b_1^{(i)}\in \C$  such that  for $i=2,3$ and for   $a_0^{(i)}=c  b_0^{(i)}\neq 0$ 
\begin{equation}\label{eq::b13at03}
  \partial ^{2}_\lambda   Z_i(x, 0 , \varepsilon   ) = \left(
                                                 \begin{array}{c}
                                                    L _c    ^{\intercal}      (a_0^{(i)},b_0^{(i)} ) ^\intercal  \\
                                                    b_0^{(i)} \\
                                                    0 \\
                                                 \end{array}
                                               \right)  + b_1^{(i)} E _2[c] ,
\end{equation} 
\end{description}
\end{lemma}
\begin{proof}
 Case $a=0$   in  \textbf{i} follows immediately from \eqref{eq:jostlambda01}. Taking the $\lambda $ derivative at $\lambda =0$ we get
\begin{align*}
  & \partial _\lambda \( \frac { \widehat{f}_i \wedge  {f} _1 \wedge f_4  }{J(x) e^{2\lambda \int _0 ^x  \frac{c-u_c(y)}{J(y)} dy}}  \)   = J ^{-1}  \widehat{f}_i \wedge  \partial _\lambda ({f} _1 \wedge f_4)
\end{align*}
where   at $\lambda =0$,  by Lemma \ref{lem:jostlambda0} and by the antisymmetry of the wedge product,
\begin{align*}
  \partial _\lambda ({f} _1 \wedge f_4) & =   \partial _\lambda {f} _1 \wedge f_4  +   {f} _1 \wedge      \partial _\lambda f_4 =  \(- \alpha  \Xi _2     +\cancel{\beta   \Xi _1 }   \)   \wedge \alpha  \Xi _1  - \alpha  \Xi _1    \wedge    \(  \alpha  \Xi _2     +\cancel{\beta   \Xi _1 }    \) =0.
\end{align*}
 The estimates for ${Z}_3$  in \eqref{eq:b13_estimate_partial} follow from the following, \small
\begin{align*}
  e^{ \mu _2 (\lambda , \varepsilon )x }  \frac{   ({f}_3 \wedge  {f} _1 \wedge f_4)(x, \lambda , \varepsilon   ) }{ J(x) e^{2\lambda \int _0 ^x  \frac{c-u_c(y)}{J(y)} dy}}  =  \frac{    \exp \(2\lambda \int _0 ^x \left [ \dfrac{ c }{c^2-K} -  \dfrac{c-u_c(y)}{J(y)}  \right ] dy\)   }{J(x)}     ( \mathbf{m} _3 \wedge   \mathbf{m}  _1 \wedge \mathbf{m}_4)(x, \lambda , \varepsilon   ) ,
\end{align*}\normalsize
which follows from \eqref{eq:sum_of_mus}, and from \eqref{eq:m1_estimate_partial},   \eqref{eq:m4_estimate_partial}  and  \eqref{eq:estm21}.  The proof of \eqref{eq:b13_estimate_partial}    for ${Z}_2$ is similar.
  By \textbf{i} we have  $\partial ^{2}_\lambda  Z_3 | _{\lambda =0}=0= \partial ^{2}_\lambda   \(e^{ \mu _2  x }Z_3 \)  | _{\lambda =0} $ is an $L^\infty (\R , \C^4)$ solution of  the adjoint equation
\eqref{eq:lineq1ad} for $\lambda =0$. Therefore  there exist $a_0^{(i)} , b_0^{(i)} ,b_1^{(i)}\in \C $ such that
\begin{align*}
  \partial ^{2}_\lambda   \(e^{ \mu _2  x }Z_3 \) | _{\lambda =0} =    \left(
                                                 \begin{array}{c}
                                                    L _c    ^{\intercal}      (a_0^{(3)} , b_0^{(3)}) ^\intercal  \\
                                                    b_0^{(3)} \\
                                                    0 \\
                                                 \end{array}
                                               \right)  +b_1^{(3)}E _2[c].
\end{align*}
Now,   we have, not writing the $\varepsilon$ and for a moment assuming without discussion the last equality,
\begin{align} \label{lem:b13at01}
   &     \left(
                                                 \begin{array}{c}
                                                   \left(
                  \begin{array}{cc}
                    -c & K \\
                    1 & -c \\
                  \end{array}
                \right)  \left(
                           \begin{array}{c}
                             a_0^{(3)} \\
                             b_0^{(3)} \\
                           \end{array}
                         \right)
 \\
                                                    b_0^{(3)} \\
                                                    0 \\
                                                 \end{array}
                                               \right) =     \lim _{x\to -\infty} \partial ^{2}_\lambda   \(e^{ \mu _2  x }Z_3(x,\lambda ) \) |_{\lambda =0}  =  \partial ^{2}_\lambda  \lim _{x\to -\infty}   (e^{ \mu _2  x }{Z}_3(x,\lambda ))|_{\lambda =0}
\end{align}
where by \eqref{lem:estimate_f1} and \eqref{eq:limm1-infty}, by      \eqref{lem:estimate_f4}, \eqref{eq:estimate_tildef22},  \eqref {eq:deff2} and \eqref {eq:strar3vec}
\begin{align*}
  \lim _{x\to -\infty} e^{ \mu _2  x } Z_3(x,\lambda ) = \frac{D(\lambda )   }{c^2-K}   *( \mathbf{v} _{3}\wedge \mathbf{v}_{1}\wedge \mathbf{v}_{4}) = \frac{D(\lambda )   }{c^2-K}     \det (\mathbf{v}_{3}, \mathbf{v}_{1}, \mathbf{v}_{4}, \mathbf{v}_{2}) \mathbf{w}_{2},
\end{align*}
so that
\begin{align*}
    \left(
                                                 \begin{array}{c}
                                                   \left(
                  \begin{array}{cc}
                    -c & K \\
                    1 & -c \\
                  \end{array}
                \right)  \left(
                           \begin{array}{c}
                             a_0^{(3)} \\
                             b_0^{(3)} \\
                           \end{array}
                         \right)
 \\
                                                    b_0^{(3)} \\
                                                    0 \\
                                                 \end{array}
                                               \right) =   -    \frac{\partial ^2 _\lambda D (0, \varepsilon )   }{c^2-K}      \det (\mathbf{v}_{1} (0 , \varepsilon ), \mathbf{v}_{2}(0 , \varepsilon ), \mathbf{v}_{3}(0 , \varepsilon ), \mathbf{v}_{4}(0 , \varepsilon ))  \mathbf{w}_{2}(0 , \varepsilon )  
\end{align*}
so that
 \begin{align*}&
    b_0^{(3)} =  -   \frac{\partial ^2 _\lambda D (0, \varepsilon )   }{(c^2-K) \< \boldsymbol{\pi}_2 ,  \mathbf{v}_2 \> _{\C ^4} | _{\lambda =0} }      \det (\mathbf{v}_{1} (0 , \varepsilon ), \mathbf{v}_{2}(0 , \varepsilon ), \mathbf{v}_{3}(0 , \varepsilon ), \mathbf{v}_{4}(0 , \varepsilon ))  \text{ and } a_0^{(3)}=c b_0^{(3)} .
 \end{align*}
A similar result is  obtained for $ a_0^{(2)}$ and $ b_0^{(2)}$. Now we need to justify the last equality in  \eqref{lem:b13at01}. It is an elementary consequence of \eqref{eq:m4_estimate_partial}, \eqref{eq:estimate_tildef232} and   the fact, elementary to check from \eqref {eq:integral_eq_m1}, that
\begin{align*}
   \partial ^{n}_\lambda  \lim _{x\to -\infty} \mathbf{m}_1(x , \lambda , \varepsilon  ) =   \lim _{x\to -\infty}  \partial ^{n}_\lambda \mathbf{m}_1(x , \lambda , \varepsilon  ) .
\end{align*}


\end{proof}

\section{Proof of Proposition \ref{prop:smooth2} }\label{sec:proofsmooth2}

It is enough to show using the notation in \eqref{ODE_LinEP} that   there exists a constant $C(  \varepsilon)$ such that for $P_{[c]} \widetilde{F}=0$
\begin{align}\label{eq:pfsmooth21}&
       \left \|         \mathbf{U}  \right \| _{       \widetilde{\Sigma}    }\le C(  \varepsilon)  \| \widetilde{F} \| _{   L ^{1,1}\( \R\)    }  \text{ for all $\lambda \in \im \R \backslash \{  0 \}$}.
   \end{align}
 We set
 \begin{align}\label{eqvarpartilde}
    \mathbf{U}=\widehat{R}(\lambda) \widetilde{F} = \sum _{j=1}^{4}\widehat{R}_j(\lambda) \widetilde{F}   \end{align}
  where for  $ \mathbf{F}   := ( L ^{-1}\widetilde{F},0,0)^\intercal$ like in \eqref{ODE_LinEP},
    \begin{align}& \widehat{R}_1(\lambda) \widetilde{F}=
    f_1(x,\lambda , \varepsilon ) \int_  {-\infty} ^x  \frac{ \det \( \mathbf{F} ,   \widehat{f}_2(y,\lambda , \varepsilon ) ,  \widehat{f}_3(y,\lambda , \varepsilon ) ,   {f}_4(y,\lambda , \varepsilon )   \)}{\widehat{W}(y  , \lambda , \varepsilon   )} dy \label{eq:r1}
 \end{align}
 and for $j\ge 2$
\begin{align}& \widehat{R}_2(\lambda) \widetilde{F}= -
    \widehat{f}_2(x,\lambda , \varepsilon ) \int_  {x}  ^{+\infty}  \frac{ \det \( {f}_1(y,\lambda , \varepsilon )  ,  \mathbf{F} ,  \widehat{f}_3(y,\lambda , \varepsilon ) ,   f _4(y,\lambda , \varepsilon )   \)}{\widehat{W}(y  , \lambda , \varepsilon   )} dy   \label{eq:r2}  \\&  \widehat{R}_3(\lambda) \widetilde{F}= -
    \widehat{f}_3(x,\lambda , \varepsilon ) \int_  {x}  ^{+\infty}  \frac{ \det \( {f}_1(y,\lambda , \varepsilon )  ,  \widehat{f}_2(y,\lambda , \varepsilon ) ,  \mathbf{F}   ,   {f}_4(y,\lambda , \varepsilon )   \)}{\widehat{W}(y  , \lambda , \varepsilon   )} dy \label{eq:r3} \\&  \widehat{R}_4(\lambda) \widetilde{F}= -
     {f}_4(x,\lambda , \varepsilon ) \int_  {x}  ^{+\infty}  \frac{ \det \( {f}_1(y,\lambda , \varepsilon )  ,  \widehat{f}_2(y,\lambda , \varepsilon ) ,  \widehat{f}_3(y,\lambda , \varepsilon ) , \mathbf{F}       \)}{\widehat{W}(y  , \lambda , \varepsilon   )} dy . \label{eq:r4}
 \end{align}
 Formula \eqref{eqvarpartilde} follows from the variation of the parameters formula.
 For $\mathbf{G} := \det (\mathbf{v}_1, \mathbf{v}_2, \mathbf{v}_3, \mathbf{v}_4) $
 \begin{align*}&
    \widehat{R}_1(\lambda) \widetilde{F}=  \frac{ (c^2-K) }{ D(\lambda , \varepsilon ) \mathbf{G} }  f_1(x,\lambda , \varepsilon )  \int_  {-\infty} ^x  \<  \mathbf{F}(y), Z_4(y, \lambda , \varepsilon   )\> _{\C^4} dy  \text{ and}\\&   \widehat{R}_4(\lambda) \widetilde{F}=  \frac{ (c^2-K) }{ D(\lambda , \varepsilon ) \mathbf{G} }  f_4(x,\lambda , \varepsilon )  \int_  {x}  ^{+\infty}  \<  \mathbf{F}(y), Z_1(y, \lambda , \varepsilon   )\> _{\C^4} dy .
 \end{align*}
 Similarly, by \eqref{eq:defhatZ12}
  \begin{align*}&
    \widehat{R}_2(\lambda) \widetilde{F}=  \frac{ (c^2-K) }{ D(\lambda , \varepsilon ) \mathbf{G} }  \widehat{f}_2(x,\lambda , \varepsilon )  \int ^  {+\infty} _x  \<  \mathbf{F}(y), \widehat{Z}_2(y, \lambda , \varepsilon   )\> _{\C^4} dy  \text{ and}\\&   \widehat{R}_3(\lambda) \widetilde{F}=  -  \frac{ (c^2-K) }{ D(\lambda , \varepsilon ) \mathbf{G} }  \widehat{f}_3(x,\lambda , \varepsilon )  \int_  {x}  ^{+\infty}  \<  \mathbf{F}(y), \widehat{Z}_3(y, \lambda , \varepsilon   )\> _{\C^4} dy .
 \end{align*}
We  distinguish between $\lambda$ close to 0 and away from 0.  We start with the  small $\lambda$  case.
Setting \begin{align}
   \label{eq:defw} {W}(x  , \lambda , \varepsilon   ):= \det \(    {f}_1(x  , \lambda , \varepsilon   ),  {f}_2(x  , \lambda , \varepsilon   ),  {f}_3(x  ,\lambda , \varepsilon  ) ,{f}_4(x  , \lambda , \varepsilon   ) \) ,
 \end{align}
then $ {W}(x  , \lambda , \varepsilon   )= \lambda ^2 \widehat{{W}}(x  , \lambda , \varepsilon   )$. Then, still from the variation of parameter formula we have
\begin{align}\label{eqvarpar}
    \mathbf{U}= {R}(\lambda) \widetilde{F} = \sum _{j=1}^{4} {R}_j(\lambda) \widetilde{F}   \end{align}
with the ${R}_j(\lambda)$ defined like the in \eqref{eq:r1}--\eqref{eq:r4} but with $\widehat{f}_2$, $\widehat{f}_3$ and $\widehat{W}$ replaced by
$f_2$, $f_3$ and $W$.

\begin{lemma}\label{lem:smooth2small} There exists an $a _4(\varepsilon )>0$  and a  $C( \varepsilon)>0$ such that
  \begin{align}\label{eq:smooth21small}
     \sup  _{\tau  \in  (-a _4(\varepsilon ),a _4(\varepsilon )  )\backslash \{  0 \} } \left \|    \widehat{R}  (\im  \tau )  Q _{[c]}    \right \| _{ \mathcal{L}\( L ^{1,1}\( \R\),   \widetilde{\Sigma} \)   }\le C( \varepsilon).
   \end{align}

\end{lemma}
\begin{proof}
We introduce
 \begin{align}& \label{eq:rhat10}
    \widehat{R}_1^{(0)}(\lambda) \widetilde{F}=  \frac{  c^2-K  }{ D(\lambda , \varepsilon ) \mathbf{G} }\sum _{a=0}^{1}  \lambda ^a  \int_  {-\infty} ^x     \partial ^a _\lambda \(  f_1(x,\lambda , \varepsilon ) \<  \mathbf{F}(y), Z_4(y, \lambda , \varepsilon   )\> _{\C^4}\) | _{\lambda =0} dy  \text{,}\\&   \widehat{R}_4^{(0)}(\lambda) \widetilde{F}=  \frac{ c^2-K  }{ D(\lambda , \varepsilon ) \mathbf{G} }  \sum _{a=0}^{1}  \lambda ^a    \int_  {x}  ^{+\infty}  \partial ^a _\lambda \(   f_4(x,\lambda , \varepsilon )   \<  \mathbf{F}(y), Z_1(y, \lambda , \varepsilon   )\> _{\C^4}   \) | _{\lambda =0} dy .\label{eq:rhat40}
 \end{align}
Similarly we define for $i=2,3$
\begin{align*}
   \widehat{R}_i^{(0)}(\lambda) \widetilde{F}= (-1) ^i  \frac{ c^2-K  }{ D(\lambda , \varepsilon ) \mathbf{G} }  \sum _{a=0}^{1}  \lambda ^a    \int_  {x}  ^{+\infty}  \partial ^a _\lambda \(   \widehat{f}_i(x,\lambda , \varepsilon )   \<  \mathbf{F}(y), Z_i(y, \lambda , \varepsilon   )\> _{\C^4}   \) | _{\lambda =0} dy .
\end{align*}
Notice that   $ \widehat{R}_i^{(0)}(\lambda) \widetilde{F}= 0$     for     $i=2,3$        by Lemma \ref{lem:b13at0}.  Next, we claim that, if we assume \eqref{eq:imprelat}, we have
\begin{equation}\label{eq:smooth21canc}
 \(   \widehat{R}_1^{(0)}(\lambda)   +   \widehat{R}_4^{(0)}(\lambda)\) \circ Q_{[c]}=0.
\end{equation}
We have
\begin{align*}
   &  \frac{ D(\lambda , \varepsilon ) \mathbf{G} } { c^2-K  }   \(  \widehat{R}_1^{(0)}(\lambda)  +    \widehat{R}_4^{(0)}(\lambda)  \)     \widetilde{F}   =A+ \lambda B \text{ where }
   A=
   f_1(x,0 , \varepsilon )   \widehat{{\alpha}}(\varepsilon)\<  \mathbf{F} , E_2 [c]\> _{L^2(\R , \C^4 ) } \text{ and }    \\&  B  = \int_  {-\infty} ^x     \partial   _\lambda    f_1(x,0 , \varepsilon ) \<  \mathbf{F}(y), Z_4(y, 0 , \varepsilon   )\> _{\C^4} dy    +  \int_  {x}  ^{+\infty}   \partial   _\lambda    f_4(x,0 , \varepsilon ) \<  \mathbf{F}(y), Z_1(y, 0 , \varepsilon   )\> _{\C^4} dy  \\&   + \int_  {-\infty} ^x        f_1(x,0 , \varepsilon ) \<  \mathbf{F}(y),  \partial   _\lambda  Z_4(y, 0 , \varepsilon   )\> _{\C^4} dy    +  \int_  {x}  ^{+\infty}        f_4(x,0 , \varepsilon ) \<  \mathbf{F}(y),  \partial   _\lambda Z_1(y, 0 , \varepsilon   )\> _{\C^4} dy
\end{align*}
where to get $A$ in the first line we used $L_1(x,0 , \varepsilon ) =- L_4(x,0 , \varepsilon )$  for $L=f,Z$  and $Z_4(\cdot , 0 , \varepsilon   )  =  \widehat{{\alpha}}(\varepsilon)E_2 [c]$.  Now we claim $A=0$. This follows from   \eqref{eq:bigeta12}
\begin{align}& \label{eq:smooth21canc1} \<  \mathbf{F} , E_i [c]\> _{L^2(\R , \C^4 ) }  =   \<    L _c  ^{-1}  \widetilde{F} ,   L _c   ^{\intercal}   \eta _i[c] \> _{L^2(\R , \C^2 ) }    =   \<       \widetilde{F} ,      \eta _i[c] \> _{L^2(\R , \C^2 ) }  =0 \text{ for }i=1,2
\end{align}
by $ P_{[c]}  \widetilde{F}=0$.   Next we claim $B=0$.  We have
\begin{align*}
  B&=         \widehat{{\alpha}}   \( -\alpha  \Xi _2[c]     +\beta   \Xi _1[c]    \) \int_  {-\infty} ^x      \<  \mathbf{F}(y),   E_2 [c]\> _{\C^4} dy    -  \widehat{{\alpha}}    \(  \alpha  \Xi _2[c]     +\beta   \Xi _1[c]    \) \int_  {x}  ^{+\infty}    \<  \mathbf{F}(y),E_2 [c] \> _{\C^4} dy  \\&    - \alpha    \Xi _1[c ]\int_  {-\infty} ^x       \<  \mathbf{F}(y), \widehat{\alpha}   E _1[c] +   \widehat{\beta}_4  E _2[c] \> _{\C^4} dy    +  \alpha    \Xi _1[c ]  \int_  {x}  ^{+\infty}      \<  \mathbf{F}(y),  -\widehat{\alpha}   E _1[c] +   \widehat{\beta}_1  E _2[c]\> _{\C^4} dy \\& = -  \alpha   \widehat{{\alpha}}   \( \Xi _1[c]  +\Xi _2[c] \)  \<  \mathbf{F} , E_1 [c]\> _{L^2(\R , \C^4 ) }  + \( \widehat{{\alpha}}  \beta   - {{\alpha}}  \widehat{\beta}_4  \)  \Xi _1[c]    \<  \mathbf{F} , E_2 [c]\> _{L^2(\R , \C^4 ) }   \\& +\( \alpha (\widehat{\beta}_1 +\widehat{\beta}_2 )   -2 \widehat{{\alpha}} \beta  \)  \Xi _1[c] \int_  {x}  ^{+\infty}      \<  \mathbf{F}(y),    E _2[c]\> _{\C^4} dy
\end{align*}
where  $B=0$ because the last term  is null by   \eqref{eq:imprelat} and \eqref{eq:smooth21canc1}.
 Let us now write
 \begin{align*}
    \widehat{R} (\lambda) \widetilde{F}  = \frac{\widehat{N}(\lambda)\widetilde{F}}{D(\lambda , \varepsilon ) \mathbf{G}} = \frac{ {N}(\lambda)\widetilde{F}}{\lambda ^2D(\lambda , \varepsilon ) \mathbf{G}}= {R} (\lambda) \widetilde{F}.
 \end{align*}
 All is left to bound is
 \begin{align*}
    \widehat{R} ^{(1)}(\lambda) \widetilde{F} :&=   \widehat{R} (\lambda) \widetilde{F} - \frac{ \sum _{a=0} ^{1} \lambda ^a  \partial _\lambda ^a \widehat{N}(0)\widetilde{F}}{D(\lambda , \varepsilon ) \mathbf{G}}  =   {R} (\lambda) \widetilde{F} - \frac{ \sum _{a=0} ^{3} \frac{\lambda ^a }{a!} \partial _\lambda ^a N(0)\widetilde{F}}{\lambda ^2 D(\lambda , \varepsilon ) \mathbf{G}}  
 \end{align*}
so that it is enough to bound \small
  \begin{align*}& {R}_1^{(1)}(\lambda) \widetilde{F}:=\frac{ (c^2-K) }{\lambda ^2 D(\lambda , \varepsilon ) \mathbf{G} }
  \int _0 ^\lambda dz \frac{(\lambda -z )^3}{2}  \int_  {-\infty} ^x   \partial _z ^4\(   f_1(x,z, \varepsilon )\frac{ \det \( \mathbf{F} ,    {f}_2(y,z , \varepsilon ) ,  {f}_3(y,z , \varepsilon ) ,   {f}_4(y,z , \varepsilon )   \)}{ J(y) e^{2z \int _0 ^y  \frac{c-u_c(y')}{J(y')} dy'}  } \) dy
 \end{align*}\normalsize
 and similarly  for $j\ge 2$   \small
\begin{align*}&{R}_2^{(1)}(\lambda) \widetilde{F}:=\frac{ -(c^2-K) }{\lambda ^2 D(\lambda , \varepsilon ) \mathbf{G} }
  \int _0 ^\lambda dz \frac{(\lambda -z )^3}{2}  \int_  {x}  ^{+\infty}  \partial _z ^4\(
     {f}_2(x,z , \varepsilon )  \frac{ \det \( {f}_1(y,z , \varepsilon )  ,  \mathbf{F} ,   {f}_3(y,z , \varepsilon ) ,   f _4(y,z , \varepsilon )   \)}{ J(y) e^{2z \int _0 ^y  \frac{c-u_c(y')}{J(y')} dy'}  }  \) dy   \\&   {R}_3^{(1)}(\lambda) \widetilde{F}:=\frac{ -(c^2-K) }{\lambda ^2 D(\lambda , \varepsilon ) \mathbf{G} }
  \int _0 ^\lambda dz \frac{(\lambda -z )^3}{2}  \int_  {x}  ^{+\infty}  \partial _z ^4\(
     {f}_3(x,z , \varepsilon )   \frac{ \det \( {f}_1(y,z , \varepsilon )  ,   {f}_2(y,z , \varepsilon ) ,  \mathbf{F}   ,   {f}_4(y,z , \varepsilon )   \)}{ J(y) e^{2z \int _0 ^y  \frac{c-u_c(y')}{J(y')} dy'}  }  \) dy   \\&   {R}_4^{(1)}((\lambda) \widetilde{F}:=\frac{ -(c^2-K) }{\lambda ^2 D(\lambda , \varepsilon ) \mathbf{G} }
  \int _0 ^\lambda dz \frac{(\lambda -z )^3}{2}  \int_  {x}  ^{+\infty}  \partial _z ^4\(
     {f}_4(x,z , \varepsilon )   \frac{ \det \( {f}_1(y,z, \varepsilon )  ,   {f}_2(y,z, \varepsilon ) ,   {f}_3(y,z , \varepsilon ) , \mathbf{F}       \)}{ J(y) e^{2z \int _0 ^y  \frac{c-u_c(y')}{J(y')} dy'}  }  \) dy .
 \end{align*}\normalsize
Here the terms ${R}_1^{(1)}(\lambda) \widetilde{F}$ and ${R}_4^{(1)}(\lambda) \widetilde{F}$ are the simpler ones, while the others two are subtler to treat. As we  noticed already    in \eqref {eq:rhat10}
we have  \small
\begin{align*}
  {R}_1^{(1)}(\lambda) \widetilde{F} =\frac{ -1 }{\lambda ^2 D(\lambda , \varepsilon )  }
  \int _0 ^\lambda dz \frac{(\lambda -z )^3}{2}  \int_  {x}  ^{+\infty}  \partial _z ^4\(
      e  ^{\mu _1(z , \varepsilon ) (x-y) }  \mathbf{m}_1(x,z , \varepsilon )  \<  \mathbf{F}(y), \mathbf{n}_1(y,z , \varepsilon )   )\> _{\C^4}\)  dy
\end{align*}
where we used the function  $g_1$ introduced right after \eqref{eq:lineq1ad}
\begin{align*}
  Z_4(y, \lambda , \varepsilon ) = \frac{\mathbf{G}}{c^2-K}g_1(y, \lambda , \varepsilon )=  \frac{\mathbf{G}}{c^2-K}    e^{-\mu _1(  \lambda , \varepsilon ) y}  \mathbf{n}_1 (y, \lambda , \varepsilon )
\end{align*}
 and where $\mathbf{n}_1$ satisfies estimates similar to \eqref{lem:estimate_f4} (but with $\mathbf{v}_4$ replaced by  $\mathbf{w}_1$).  Since for $\lambda$ and $z$ sufficiently close to 0 we can assume
$\Re  \mu _1(z , \varepsilon )\le -\sqrt{ \mathsf{V} \varepsilon } $
from \eqref{EigenSpliting1_3}, \eqref{eq:sum_of_mus}  and \eqref{eq:mu40}, we conclude that
\begin{align}\label{eq:estR1(1)}
  \left |{R}_1^{(1)}(\lambda) \widetilde{F} (x) \right | \lesssim     \int_  {x}  ^{+\infty}    (|x|+|y|) ^4    \exp \(- \sqrt{ \mathsf{V} \varepsilon }   |x-y|\)  |\widetilde{F} (y) | dy
\end{align}
so that  for $j=1$ the following is true,
\begin{align}\label{eq:estRj(1)}
   \| {R}_j^{(1)}(\lambda) \widetilde{F}  \| _{\widetilde{\Sigma}}\lesssim  C _{\varepsilon} \| \widetilde{F}\| _{L^1(\R )}  \text{ for $j=1$ and $j=4$}.
\end{align}
It is possible to prove similarly the  estimate for $j=4$. We focus now on the remaining terms   ${R}_2^{(1)}(\lambda) \widetilde{F}$  and ${R}_3^{(1)}(\lambda) \widetilde{F}$.  We can write  schematically
\begin{align*}&
    {R}_2^{(1)}(\lambda) \widetilde{F}  =  \frac{  -(c^2-K) }{\lambda ^2 D(\lambda , \varepsilon ) \mathbf{G} }  \int_  {x}  ^{+\infty}
   \mathbf{a}(x,y, \lambda )  : \mathbf{F}(y)   dy \text{ where}
\\& \mathbf{a}(x,y, \lambda )  =
     {f}_2(x,\lambda  , \varepsilon )     Z_3(y,z , \varepsilon )  - \frac{\lambda ^3}{3!} c_2(0,\varepsilon )
     {f}_3(x,0  , \varepsilon )      \partial _z ^3 Z_3(y,z , \varepsilon )| _{z=0}     .
\end{align*}
We have  \small
\begin{align*} &
  \mathbf{a}(x,y, \lambda )=    A(\lambda , x,y) +B(\lambda , x,y)   \text{ where}    \\&  A(\lambda , x,y) =
  e^{-\mu _2(\lambda , \varepsilon ) y}  \(     {f}_2(x,\lambda  , \varepsilon )    e^{ \mu _2(\lambda , \varepsilon ) y}  Z_3(y,\lambda , \varepsilon ) -  \frac{\lambda ^3}{3!}
     {f}_2(x,0  , \varepsilon )      \partial _z ^3\(   e^{ \mu _2(z , \varepsilon ) y}  Z_3(y,z , \varepsilon )   \) | _{z=0} \)
\\& B(\lambda , x,y) = \frac{\lambda ^3}{3!} c_2(0,\varepsilon ) \( e^{-\mu _2(\lambda  , \varepsilon ) y} -1  \)    {f}_4(x,0  , \varepsilon )       \partial _z ^3 Z_3(y,z , \varepsilon )| _{z=0}  .
\end{align*}
\normalsize
 Since
\begin{align*}
   A(\lambda , x,y) =  \int _0 ^\lambda dz \frac{(\lambda -z )^3}{2}  \partial _z ^4\(   e^{ \mu _2(z , \varepsilon ) x}  \mathbf{m } _2(x,z  , \varepsilon )    e^{ \mu _2(z , \varepsilon ) y}  Z_3(y,z , \varepsilon ) \)
\end{align*}
by \eqref{eq:estm21} and   \eqref{eq:b13_estimate_partial}  we have
\begin{align*}
  |A(\lambda , x,y)|\lesssim   |\lambda |^4  (1+|x|)^4
\end{align*}
so that, if we set 
\begin{align*}&
    {R} _{21}^{(1)}(\lambda) \widetilde{F}  :=  \frac{  -(c^2-K) }{\lambda ^2 D(\lambda , \varepsilon ) \mathbf{G} }  \int_  {x}  ^{+\infty}
   A(\lambda , x,y)  : \mathbf{F}(y)    dy
\end{align*}
we have 
\begin{align}\label{eq:R121}
  \left \|  {R} _{21}^{(1)}(\lambda) \widetilde{F}  \right \| _{\widetilde{\Sigma}}   \lesssim   \|    \widetilde{{F}} \| _{L^1(\R )}   .
\end{align}
By    \eqref{eq:jostlambda01}, \eqref{eq:deff2},  \eqref{eq:defhatZ12},  \eqref{eq:defZ12} and  \eqref{eq::b13at03}   we have
\begin{align*}&
  B(\lambda , x,y) = B_1(\lambda , x,y) +B_2(\lambda , x,y)    \text{ with }\\& B_1(\lambda , x,y) = \frac{3\lambda ^3}{2} c_2(0,\varepsilon )  b_1^{(3)} \alpha ( \varepsilon )  \( e^{-\mu _2(\lambda  , \varepsilon ) y} -1  \)    \Xi _1[x,c]   E _2[y,c] \text{ and} \\& B_2(\lambda , x,y) = \frac{3\lambda ^3}{2} c_2(0,\varepsilon )    \alpha ( \varepsilon )  \( e^{-\mu _2(\lambda  , \varepsilon ) y} -1  \)    \Xi _1[x,c]  \left ( \begin{array}{c}
                                            L _c    ^{\intercal}      (a_0^{(3)},b_0^{(3)} ) ^\intercal  \\
                                                    b_0^{(3)} \\
                                                    0 \\
                                                 \end{array}
                                               \right)  .
\end{align*}
 Then
\begin{align*}
  \left |  \frac{  -(c^2-K) }{\lambda ^2 D(\lambda , \varepsilon ) \mathbf{G} }  \int_  {x}  ^{+\infty}B_1(\lambda , x,y) : \mathbf{F}(y)   dy  \right | &\lesssim   \left |  \frac{   \lambda  \mu _2(\lambda  , \varepsilon ) }{  D(\lambda , \varepsilon ) } \Xi _1[x,c]   \int_  {x}  ^{+\infty}  |\< y\>    E _2[y,c] \mathbf{F}(y)  |  dy  \right |
\\& \lesssim   \int_  {x}  ^{+\infty}    |   \widetilde{{F}} (y) |  dy 
\end{align*} 
so that, if we set
\begin{align*}&
    {R} _{22}^{(1)}(\lambda) \widetilde{F}  :=  \frac{  -(c^2-K) }{\lambda ^2 D(\lambda , \varepsilon ) \mathbf{G} }  \int_  {x}  ^{+\infty}
   B_1(\lambda , x,y) : \mathbf{F}(y)    dy
\end{align*}
we have
\begin{align}\label{eq:R122}
  \left \|  {R} _{22}^{(1)}(\lambda) \widetilde{F}  \right \| _{\widetilde{\Sigma}}   \lesssim   \|    \widetilde{{F}} \| _{L^1(\R )}   .
\end{align}
 By Lemma \ref{lem:b13at0} and by $ \mathbf{F}   := ( L ^{-1}\widetilde{F},0,0)^\intercal$ we have \small
\begin{align*}
  \frac{  -(c^2-K) }{\lambda ^2 D(\lambda , \varepsilon ) \mathbf{G} }  B_2(\lambda , x,y) : \mathbf{F}(y) = \frac{  -(c^2-K) \lambda }{ 2 D(\lambda , \varepsilon ) \mathbf{G} } c_2(0,\varepsilon )    \alpha ( \varepsilon )  \( e^{-\mu _2(\lambda  , \varepsilon ) y} -1  \)   b_0^{(3)}(\varepsilon )  \Xi _1[x,c]   \< (c,1) ^\intercal , \widetilde{{F}} (y) \> .
\end{align*} \normalsize
Then we set
\begin{align}\label{eq:R123}&
    {R} _{23}^{(1)}(\lambda) \widetilde{F}  :=  \frac{  -(c^2-K) }{\lambda ^2 D(\lambda , \varepsilon ) \det \mathbf{V}(\lambda , \varepsilon ) }  \int_  {x}  ^{+\infty}
   B_1(\lambda , x,y) : \mathbf{F}(y)    dy.
\end{align}
Let now $\widetilde{\phi} (x) :=    \frac{ \int _{-\infty} ^{x }    \psi_{\mathrm{KdV}}   ( y)  dy}{ \int _{\R }      \psi_{\mathrm{KdV}}   ( y)  dy} $
and write 
\begin{align*}
   {R} _{23}^{(1)}(\lambda) \widetilde{F}& =   \frac{  -(c^2-K) }{\lambda ^2 D(\lambda , \varepsilon ) \det \mathbf{V}(\lambda , \varepsilon ) }  \int_  {x}  ^{+\infty}
  (1- \widetilde{\phi} (y) )  B_1(\lambda , x,y) : \mathbf{F}(y)    dy \\&   + \frac{   (c^2-K) }{\lambda ^2 D(\lambda , \varepsilon ) \det \mathbf{V}(\lambda , \varepsilon ) }  \int_  {-\infty} ^x   \widetilde{\phi} (y)  B_1(\lambda , x,y) : \mathbf{F}(y)    dy       \\& -  \frac{   (c^2-K) }{\lambda ^2 D(\lambda , \varepsilon ) \det \mathbf{V}(\lambda , \varepsilon ) }    \int_  {\R }   \widetilde{\phi} (y)  B_1(\lambda , x,y) : \mathbf{F}(y)    dy =\sum _{j=1}^{3}  {R} _{23j}^{(1)}(\lambda) \widetilde{F}  .
\end{align*}
  Then, using the rapid decay to 0 of $1- \widetilde{\phi}$ for $y\to +\infty$,  we have 
\begin{align}
    \left \|  {R} _{231}^{(1)}(\lambda) \widetilde{F}  \right \| _{\widetilde{\Sigma}}&  \lesssim  \left |  \frac{   \lambda  \mu _2(\lambda  , \varepsilon ) }{  D(\lambda , \varepsilon ) } \right | \| \< x \>  \Xi _1[x,c] \| _{\widetilde{\Sigma}} \left \|   \< x \> ^{-1}  \int_  {x}  ^{+\infty} \< y\>  (1- \widetilde{\phi} (y) )  |\mathbf{F}(y)| dy \right \| _{L^\infty (\R )} \nonumber \\& \lesssim   \|  \widetilde{{F}} \| _{L^1(\R )}   , \label{eq:r2311}
\end{align}
and, using the rapid decay to 0 of $ \widetilde{\phi}$ for $y\to -\infty$,
\begin{align}
    \left \|  {R} _{232}^{(1)}(\lambda) \widetilde{F}  \right \| _{\widetilde{\Sigma}}&  \lesssim  \left |  \frac{   \lambda  \mu _2(\lambda  , \varepsilon ) }{  D(\lambda , \varepsilon ) } \right | \| \< x \>  \Xi _1[x,c] \| _{\widetilde{\Sigma}} \left \|   \< x \> ^{-1}  \int_  {-\infty} ^x  \< y\>    \widetilde{\phi} (y)    |\mathbf{F}(y)| dy \right \| _{L^\infty (\R )} \nonumber \\& \lesssim   \|      \widetilde{{F}} \| _{L^1(\R )}   , \label{eq:r2321}
\end{align}
while 
\begin{align}
    \left \|  {R} _{233}^{(1)}(\lambda) \widetilde{F}  \right \| _{\widetilde{\Sigma}}&  \lesssim  \left |  \frac{   \lambda  \mu _2(\lambda  , \varepsilon ) }{  D(\lambda , \varepsilon ) } \right | \|    \Xi _1[x,c] \| _{\widetilde{\Sigma}} \left \|   \< x \> ^{-1}  \int_  {\R }    \< y \>    \widetilde{\phi} (y)    |\mathbf{F}(y)| dy \right \| _{L^\infty (\R )} \nonumber \\& \lesssim   \|      \widetilde{{F}} \| _{L^1(\R _- )}  +   \| \< y \>      \widetilde{{F}} \| _{L^1(\R _+ )}   . \label{eq:r2331}
\end{align}
Notice that we have proved
\begin{align}\label{eq:estR1good}
  \left \| \( {R} _{2 }^{(1)}(\lambda)    -  {R} _{233}^{(1)}(\lambda) \) \widetilde{F} ) \right \| _{\widetilde{\Sigma}}   \lesssim   \|    \widetilde{{F}} \| _{L^1(\R )}   .
\end{align}
The discussion for ${R}_3^{(1)}(\lambda) \widetilde{F}$ is similar to that for ${R}_2^{(1)}(\lambda) \widetilde{F}$.
We have    schematically
\begin{align*}&
    {R}_3^{(1)}(\lambda) \widetilde{F}  =  \frac{  -(c^2-K) }{\lambda ^2 D(\lambda , \varepsilon ) \mathbf{G} }  \int_  {x}  ^{+\infty}
   \mathbf{b}(x,y, \lambda )  : \mathbf{F}(y)   dy \text{ where}
\\& \mathbf{b}(x,y, \lambda )  =
     {f}_3(x,\lambda  , \varepsilon )     Z_2(y,z , \varepsilon )  - \frac{\lambda ^3}{3!} c_3(0,\varepsilon )
     {f}_3(x,0  , \varepsilon )      \partial _z ^3 Z_2(y,z , \varepsilon )| _{z=0}      
\end{align*}
where we write  
\small
\begin{align*} &
  \mathbf{b}(x,y, \lambda )=    \widetilde{A}(\lambda , x,y) +\widetilde{B}(\lambda , x,y)   \text{ where}    \\&  \widetilde{A}(\lambda , x,y) =
  e^{-\mu _3(\lambda , \varepsilon ) y}  \(     {f}_3(x,\lambda  , \varepsilon )    e^{ \mu _3(\lambda , \varepsilon ) y}  Z_2(y,\lambda , \varepsilon ) -  \frac{\lambda ^3}{3!}
     {f}_3(x,0  , \varepsilon )      \partial _z ^3\(   e^{ \mu _3(z , \varepsilon ) y}  Z_2(y,z , \varepsilon )   \) | _{z=0} \)
\\& \widetilde{B}(\lambda , x,y) = \frac{\lambda ^3}{3!} c_3(0,\varepsilon ) \( e^{-\mu _3(\lambda  , \varepsilon ) y} -1  \)    {f}_4(x,0  , \varepsilon )       \partial _z ^3 Z_2(y,z , \varepsilon )| _{z=0}  .
\end{align*}
\normalsize Exactly as before, if we set
\begin{align*}&
    {R} _{31}^{(1)}(\lambda) \widetilde{F}  :=  \frac{  -(c^2-K) }{\lambda ^2 D(\lambda , \varepsilon ) \mathbf{G} }  \int_  {x}  ^{+\infty}
   \widetilde{A}(\lambda , x,y)  : \mathbf{F}(y)    dy
\end{align*}
we have
\begin{align}\label{eq:R131}
  \left \|  {R} _{31}^{(1)}(\lambda) \widetilde{F}  \right \| _{\widetilde{\Sigma}}   \lesssim   \|    \widetilde{{F}} \| _{L^1(\R )}   .
\end{align}
By    \eqref{eq:jostlambda01}, \eqref{eq:deff2},  \eqref{eq:defhatZ12},  \eqref{eq:defZ12} and  \eqref{eq::b13at03}   we have
\begin{align*}&
  \widetilde{B}(\lambda , x,y) = \widetilde{B}_1(\lambda , x,y) +\widetilde{B}_2(\lambda , x,y)    \text{ with }\\& \widetilde{B}_1(\lambda , x,y) = \frac{3\lambda ^3}{2} c_3(0,\varepsilon )  b_1^{(2)} \alpha ( \varepsilon )  \( e^{-\mu _2(\lambda  , \varepsilon ) y} -1  \)    \Xi _1[x,c]   E _2[y,c] \text{ and} \\&  \widetilde{B}_2(\lambda , x,y) = \frac{3\lambda ^3}{2} c_2(0,\varepsilon )    \alpha ( \varepsilon )  \( e^{-\mu _2(\lambda  , \varepsilon ) y} -1  \)    \Xi _1[x,c]  \left ( \begin{array}{c}
                                            L _c    ^{\intercal}      (a_0^{(2)},b_0^{(2)} ) ^\intercal  \\
                                                    b_0^{(2)} \\
                                                    0 \\
                                                 \end{array}
                                               \right)   
\end{align*}
and for \begin{align*}&
    {R} _{32}^{(1)}(\lambda) \widetilde{F}  :=  \frac{  -(c^2-K) }{\lambda ^2 D(\lambda , \varepsilon ) \mathbf{G} }  \int_  {x}  ^{+\infty}
   \widetilde{B}_1(\lambda , x,y) : \mathbf{F}(y)    dy
\end{align*}
we have
\begin{align}\label{eq:R132}
  \left \|  {R} _{32}^{(1)}(\lambda) \widetilde{F}  \right \| _{\widetilde{\Sigma}}   \lesssim   \|    \widetilde{{F}} \| _{L^1(\R )}   .
\end{align}
 By Lemma \ref{lem:b13at0} and by $ \mathbf{F}   := ( L ^{-1}\widetilde{F},0,0)^\intercal$ we have \small
\begin{align*}
  \frac{  -(c^2-K) }{\lambda ^2 D(\lambda , \varepsilon ) \mathbf{G} }  \widetilde{B}_2(\lambda , x,y) : \mathbf{F}(y) = \frac{  -(c^2-K) \lambda }{ 2 D(\lambda , \varepsilon ) \mathbf{G} } c_3(0,\varepsilon )    \alpha ( \varepsilon )  \( e^{-\mu _3(\lambda  , \varepsilon ) y} -1  \)   b_0^{(2)}(\varepsilon )  \Xi _1[x,c]   \< (c,1) ^\intercal , \widetilde{{F}} (y) \> .
\end{align*} \normalsize
 Then we set
\begin{align}\label{eq:R133}&
    {R} _{33}^{(1)}(\lambda) \widetilde{F}  :=  \frac{  -(c^2-K) }{\lambda ^2 D(\lambda , \varepsilon ) \det \mathbf{V}(\lambda , \varepsilon ) }  \int_  {x}  ^{+\infty}
   \widetilde{B}_2(\lambda , x,y) : \mathbf{F}(y)    dy
\end{align}
  and we write
\begin{align*}
   {R} _{33}^{(1)}(\lambda) \widetilde{F}& =   \frac{  -(c^2-K) }{\lambda ^2 D(\lambda , \varepsilon ) \det \mathbf{V}(\lambda , \varepsilon ) }  \int_  {x}  ^{+\infty}
  (1- \widetilde{\phi} (y) )  \widetilde{B}_2(\lambda , x,y) : \mathbf{F}(y)    dy \\&   + \frac{   (c^2-K) }{\lambda ^2 D(\lambda , \varepsilon ) \det \mathbf{V}(\lambda , \varepsilon ) }  \int_  {-\infty} ^x   \widetilde{\phi} (y)  \widetilde{B}_2(\lambda , x,y) : \mathbf{F}(y)    dy       \\& -  \frac{   (c^2-K) }{\lambda ^2 D(\lambda , \varepsilon ) \det \mathbf{V}(\lambda , \varepsilon ) }    \int_  {\R }   \widetilde{\phi} (y)  \widetilde{B}_2(\lambda , x,y) : \mathbf{F}(y)    dy =\sum _{j=1}^{3}  {R} _{33j}^{(1)}(\lambda) \widetilde{F}  .
\end{align*}
Then like for \eqref{eq:r2311}--\eqref{eq:r2331} we have 
\begin{align}
    \left \|  {R} _{331}^{(1)}(\lambda) \widetilde{F}  \right \| _{\widetilde{\Sigma}}&  \lesssim   \|  \widetilde{{F}} \| _{L^1(\R )}   , \label{eq:r3311}  \\
    \left \|  {R} _{332}^{(1)}(\lambda) \widetilde{F}  \right \| _{\widetilde{\Sigma}}&  \lesssim    \|      \widetilde{{F}} \| _{L^1(\R )}   \text{ and}  \label{eq:r3321}\\
    \left \|  {R} _{333}^{(1)}(\lambda) \widetilde{F}  \right \| _{\widetilde{\Sigma}}&  \lesssim      \|      \widetilde{{F}} \| _{L^1(\R _- )}  +   \| \< y \>      \widetilde{{F}} \| _{L^1(\R _+ )}   . \label{eq:r3331}
\end{align}
Notice that we have proved
\begin{align}\label{eq:estR3good}
  \left \| \( {R} _{3 }^{(1)}(\lambda)    -  {R} _{333}^{(1)}(\lambda) \) \widetilde{F} ) \right \| _{\widetilde{\Sigma}}   \lesssim   \|    \widetilde{{F}} \| _{L^1(\R )}   .
\end{align}
Collecting the above results we conclude that we have also proved the following improved estimate
\begin{align}\label{eq:estRgood}
 \sup  _{\tau  \in  (-a _4(\varepsilon ),a _4(\varepsilon )  )\backslash \{  0 \} } \left \|   \( \widehat{R}  (\im  \tau )-     {R} _{233}^{(1)}(\im  \tau )  -  {R} _{333}^{(1)}(\im  \tau) \)   Q _{[c]}     \right \| _{ \mathcal{L}\( L ^{1 }\( \R\),   \widetilde{\Sigma} \)   }\le C( \varepsilon).
\end{align}

\end{proof}

In the next section we prove the following.

\begin{lemma}\label{lem:smooth2large} There exists a  $C(   \varepsilon  )>0$ such that
  \begin{align}\label{eq:smooth21large}
     \sup  _{ \tau \in \R  \backslash (-a _4(\varepsilon ),a _4(\varepsilon )  )  } \left \|     R   (\im  \tau  )      \right \| _{ \mathcal{L}\( L^1\( \R\),   \widetilde{\Sigma} \)   }\le   C(    \varepsilon )   .
   \end{align}

\end{lemma}

\section{Proof of Lemma \ref{lem:smooth2large}}\label{sec:smooth2large}

For the proof we need to distinguish between $|\lambda |\in [a _4(\varepsilon ), 1]$ and $|\lambda |\in [1, +\infty )$.
We start by assuming $|\lambda |\in [a _4(\varepsilon ), 1]$. The  equation
\eqref{ODE_LinEP} can be written as
\begin{align}\label{eq:resolveq}
  (1-R_0 (\lambda ) q ) \mathbf{U} = R_0 (\lambda )\mathbf{F}
\end{align}
where
\begin{align}\label{eq:resolveq2}
  R_0 (\lambda )g(x) := \mathbf{v}_1  \int _{-\infty}^x e^{\mu _1(x-y) } \< g,\mathbf{w}_1  \> _{\C ^4} dy - \sum _{j=2}^{4}\mathbf{v}_j  \int _{x} ^{+\infty}  e^{\mu _j(x-y) } \< g,\mathbf{w}_j  \> _{\C ^4} dy .
\end{align}
We have for some $C _{\varepsilon}>0$
\begin{align}\label{eq:resolveq3}
  \sup _{ \tau \in [a _4(\varepsilon ), 1]} \|  R_0 (\pm \im \tau )\| _{\mathcal{L}(L^1(\R ),L^\infty(\R ) ) )} \le C _{\varepsilon} .
\end{align}
The operator $R_0 (\lambda ) q$ is compact from $L^\infty(\R ) $ into itself, so equation \eqref {eq:resolveq} is not solvable
only if there is a nonzero $\mathbf{U}\in L^\infty(\R ) $ with $ \mathbf{U} = R_0 (\lambda ) q\mathbf{U}$. It is then easy to see that $U\in L^2 _a (\R )$ for some appropriate small $a>0$. That would make this $\lambda \in \im \R \backslash \{ 0  \}$ an eigenvalue for the operator $\mathcal{L}_c$  in $  L^2 _a (\R )$, which is shown not to be possible by Bae and Kwon \cite{BK22ARMA}. This completes the proof of Lemma   \ref{lem:smooth2large} for    $|\lambda |\in [a _4(\varepsilon ), 1]$

\noindent We turn now to the case when $|\lambda |\ge a_0$   for $a_0=1$. Here the lower bound $a_0 $ plays no particular role. Indeed 
the above argument proves the invertibility of $ 1-R_0 (\lambda ) q  $ for a
any $\lambda \neq 0$, but it does not provide bounds for the operators.  In   \cite{CM25D1} the behavior of these operators for $\lambda \to \infty$ is trivial, thanks to the particularly simple formulas for the $R_0(\lambda )$ for the KdV. However
the operator \eqref{eq:resolveq2} is  more complicated, in particular because of the asymptotic formulas (5.27a)--(5.27d), that is \begin{subequations}\label{mujlarge}
\begin{align}
 \text{for } j& =1,4, & \text{for } j& =2,3, \nonumber \\
\mu_j & = (-1)^{j} + O(|\lambda|^{-2}), &  \mu_j & = \frac{\lambda}{c+ (-1)^j\sqrt{K}}\left(1+ O(|\lambda|^{-2})\right),  \label{mujlarge1} \\
 \frac{c\mu_j-\lambda}{\mu_j} & = (-1)^{j+1}\lambda\left(1+O(|\lambda|^{-1})\right),  &   \frac{c\mu_j-\lambda}{\mu_j} & = (-1)^{j+1}\sqrt{K}\left(1+O(|\lambda|^{-2})\right), \label{mujlarge2}\\
1-\mu_j^2 & = \lambda^{-2}\left( 1+O(|\lambda|^{-1}) \right) , &  1-\mu_j^2 & = \frac{-\lambda^2}{(c+(-1)^j\sqrt{K})^2}\left(1+O(|\lambda|^{-2})\right),  \label{mujlarge3} \\
 \< \boldsymbol{\pi}_j  ,\mathbf{v}_j \> & =2\lambda^2\left(1+O(|\lambda|^{-1}) \right), & \< \boldsymbol{\pi}_j  ,\mathbf{v}_j \> & =  -\lambda^2 \left( \frac{2(-1)^{j}\sqrt{K}}{(c+(-1)^{j}\sqrt{K})}  + O(|\lambda|^{-1}) \right) \label{muj large4}
\end{align}
\end{subequations}
as $ \lambda  \to \infty$ uniformly in $\varepsilon$ which makes complicated a direct estimate of  \eqref{eq:resolveq2}.

\noindent To proceed we   rely  on Sect. 5.3 and 5.5  and Appendix 9.4  in Bae and Kwon \cite{BK22ARMA}.
 A direct computation in Bae and Kwon  \cite[Sect. 5.3]{BK22ARMA}
shows for the matrices in \eqref{DecompW0(A-Ainf)V0}  that
\begin{equation}\label{BdW0R2V0_1}
|(\widetilde{\mathbf{W}}^\intercal q _1(x) \widetilde{\mathbf{V}}) | \leq C \varepsilon  e^{-C\varepsilon ^{1/2}|x|}   \text{ for a fixed $C$ and for $\varepsilon \in (0, \varepsilon _0)$ with $\varepsilon _0>0$ small enough}
\end{equation}
and \begin{equation}\label{DecompW0R1V0}
\lambda \widetilde{\mathbf{W}}q _2(x) \widetilde{\mathbf{V}} = \frac{\lambda}{2\sqrt{K}} S_1 + \widetilde{q}_{ 2 },
\end{equation}
where $S_1=S_1(x,\varepsilon )$ is a symmetric and semi--positive matrix defined by
\begin{equation}\label{S1sym}
S_1:=
\begin{pmatrix}
0 & 0 & 0 & 0 \\
0 & 2\sqrt{K} (q _2)_{11}-K (q _2)_{12}-(q _2)_{21}  & K (q _2)_{12} - (q _2)_{21} & 0  \\
0 & K (q _2)_{12} - (q _2)_{21}  & 2\sqrt{K} (q _2)_{11}+K (q _2)_{12}+ (q _2)_{21} & 0 \\
0 & 0 & 0 & 0
\end{pmatrix}
\end{equation}
and $\widetilde{q}_{ 2 }$ is a matrix such that like in \eqref{BdW0R2V0_1}
\begin{equation}\label{bdR1jk}
|\widetilde{q}_{ 2 }(x , \lambda , \varepsilon)|\leq C \varepsilon  e^{-C\varepsilon ^{1/2}|x|} .
\end{equation}
Applying the matrices  in \eqref{V0W0} to  equation
\eqref{ODE_LinEP} we obtain, for $\boldsymbol{\mu}$ defined in \eqref{eq:boldmu},
 \begin{equation*}
\( \frac{d }{dx} -\boldsymbol{\mu}       -\frac{\lambda}{2\sqrt{K}} S_1\)
\widetilde{\mathbf{W}}^{\intercal}\mathbf{U}  =\widetilde{q}_2\widetilde{\mathbf{W}}^{\intercal}\mathbf{U}  +  \widetilde{\mathbf{W}}^{\intercal}\mathbf{F}.
\end{equation*}
Bae and Kwon \cite[Sect. 5.5]{BK22ARMA}   show that there exists a matrix with  $|\Xi (\varepsilon, x,y)|\le C $ for any $x\le y$ for fixed $C>0$ like in
\eqref{BdW0R2V0_1} where
\begin{align*}\widetilde{\mathbf{W}}^{\intercal}\mathbf{U}& = e_1 \int _{-\infty}^x e^{\mu _1(x-y) } \<  \widetilde{q}_{ 2 } \widetilde{\mathbf{W}}^{\intercal}\mathbf{U} + \widetilde{\mathbf{W}}^{\intercal}\mathbf{F},e_1  \> _{\C ^4} dy \\& - \int _{x} ^{+\infty} \Xi (\varepsilon, x,y) \(   \widetilde{q}_{ 2 } \widetilde{\mathbf{W}}^{\intercal}\mathbf{U} + \widetilde{\mathbf{W}}^{\intercal}\mathbf{F}  \) dy.
\end{align*}
Applying $\widetilde{\mathbf{V}}$ to this equation, by \eqref{W0V01} we obtain
\begin{align*} \mathbf{U}& = \widetilde{\mathbf{V}}e_1 \int _{-\infty}^x e^{\mu _1(x-y) } \<  \widetilde{q}_{ 2 } \widetilde{\mathbf{W}}^{\intercal}\mathbf{U} + \widetilde{\mathbf{W}}^{\intercal}\mathbf{F},e_1  \> _{\C ^4} dy \\& - \int _{x} ^{+\infty} \widetilde{\mathbf{V}}\Xi (\varepsilon, x,y) \(   \widetilde{q}_{ 2 } \widetilde{\mathbf{W}}^{\intercal}\mathbf{U} + \widetilde{\mathbf{W}}^{\intercal}\mathbf{F}  \) dy.
\end{align*}
where
\begin{align*} &
 \left  \|   \widetilde{\mathbf{V}}e_1 \int _{-\infty}^x e^{\mu _1(x-y) } \<   \widetilde{\mathbf{W}}^{\intercal}\mathbf{F},e_1  \> _{\C ^4} dy   -   \int _{x} ^{+\infty} \widetilde{\mathbf{V}}\Xi (\varepsilon, x,y)   \widetilde{\mathbf{W}}^{\intercal}\mathbf{F}    dy                   \right \| _{L^\infty (\R )}\\&  \lesssim \|  \widetilde{\mathbf{V}} \|    \ \|  \widetilde{\mathbf{W}} \|    \left  \|    \mathbf{F}                    \right \| _{L^1 (\R )}\lesssim  \left  \|    \mathbf{F}                    \right \| _{L^1 (\R )}
\end{align*}
where we used $\|  \widetilde{\mathbf{V}} \|\le C |\lambda|$  and $\|  \widetilde{\mathbf{W}} \|\le C |\lambda|^{-1}$  which follow from \eqref{LRVec_A}--\eqref{eigenvec_A2}, formulas \eqref{mujlarge} 
and by $m_j=O(\lambda^{-1})$ for $j=1$ or $4$ and   $m_j=O(\lambda )$ for $j=2,3$, the constants in \eqref{mj}. 
We have  
\begin{align*} &
 \left  \|   \widetilde{\mathbf{V}}e_1 \int _{-\infty}^x e^{\mu _1(x-y) } \<    \widetilde{q}_{ 2 } \widetilde{\mathbf{W}}^{\intercal}\mathbf{U},e_1  \> _{\C ^4} dy   -   \int _{x} ^{+\infty} \widetilde{\mathbf{V}}\Xi (\varepsilon, x,y)    \widetilde{q}_{ 2 } \widetilde{\mathbf{W}}^{\intercal}\mathbf{U}    dy                   \right \| _{L^\infty (\R )}\\&  \lesssim \|  \widetilde{\mathbf{V}} \|         \left  \|    \widetilde{q}_{ 2 } \widetilde{\mathbf{W}}^{\intercal}\mathbf{U}                     \right \| _{L^1 (\R )}\lesssim  \left  \|    \widetilde{q}_{ 2 }                    \right \| _{L^1 (\R )}  \left  \|    \mathbf{U}                    \right \| _{L^\infty (\R )}\lesssim \sqrt{\varepsilon}\left  \|    \mathbf{U}                    \right \| _{L^\infty (\R )}
\end{align*}
 similarly  where at the end we used \eqref{bdR1jk}.  Hence we have proved that the solution of \eqref{eq:resolveq} 
for  in the case $|\lambda|\ge 1$ satisfies $\left  \|    \mathbf{U}                    \right \| _{L^\infty (\R )}\le C  \left  \|    \mathbf{F}                    \right \| _{L^1 (\R )}$ with a fixed constant like in \eqref{BdW0R2V0_1}, that is 
\begin{align}\label{eq:resolveq4}
  \sup _{ \tau \in [  1, +\infty )} \|  R_0 (\pm \im \tau )\| _{\mathcal{L}(L^1(\R ),L^\infty(\R ) ) )} \le C   .
\end{align}
This and \eqref{eq:resolveq3}   complete the  proof of Lemma \ref{lem:smooth2large}.

\qed

\section{Proof of Lemma \ref {lem:gKdV21hard} }\label{sec:gKdV21hard}

We start by proving  \eqref{eq:gKdV21hard1} for $D=D_1$, defined in \eqref{eq:comm}. With a notational abuse we ignore the constant matrix in the definition of $ D_1$: this makes no material difference in the argument below. We proceed like in \cite{CM25D1} writing
for $\varepsilon _4$  the constant in \eqref{eq:smooth21small} and \eqref{eq:smooth21large},
\begin{align*}
    \int _{0}^t e^{  (t-s)\mathcal{L} _{c_0}} Q   \zeta_{B}'\widetilde{V}       ds &= (2\pi )^{-\frac{1}{2}}\int _{|\tau|< \varepsilon _4} e^{-\im \tau t}    R _{\mathcal{L} _{c_0}}(\im \tau ) Q \zeta _B ' \widehat{\widetilde{V}}(\tau )    d\tau\\& + (2\pi )^{-\frac{1}{2}}\int _{|\tau|> \varepsilon _4} e^{-\im \tau t}    R _{\mathcal{L} _{c_0}}(\im \tau ) Q \zeta _B ' \widehat{\widetilde{V}}(\tau )    d\tau=:T_1(t) \zeta_{B}'\widetilde{V}  + T_2(t) \zeta_{B}'\widetilde{V}.
\end{align*}
As in \cite{CM25D1},  proceeding like in \eqref{eq:fourtran} but using $ L^{ 1 }\( \R\)$    instead of $ L^{ 1,1 }\( \R\)$ thanks to Lemma \ref{lem:smooth2large}  we have
\begin{align*}&
     \left \| T_2  \zeta_{B}'\widetilde{V} \right \| _{L^2 \( I ,    \widetilde{\Sigma}\) }    \lesssim B^{-1/2}\(  \| V \| _{L^2(I,  { \Sigma }_{1A A_1} )} +
\|V \| _{L^2(I,  { \Sigma }_{2A A_1} )} \) .
\end{align*}
Next, we write \small
  \begin{align*}
    T_1(t) \zeta_{B}'\widetilde{V}  &= (2\pi )^{-\frac{1}{2}}\int _{|\tau|< \varepsilon _4} e^{-\im \tau t}    \(   R _{\mathcal{L} _{c_0}}(\im \tau )  -     {R} _{233}^{(1)}(\im  \tau )  -  {R} _{333}^{(1)}(\im  \tau)\)Q \zeta _B ' \widehat{\widetilde{V}}(\tau )    d\tau \\& + (2\pi )^{-\frac{1}{2}}\int _{|\tau|< \varepsilon _4} e^{-\im \tau t}       {R} _{233}^{(1)}(\im  \tau ) Q \zeta _B ' \widehat{\widetilde{V}}(\tau )    d\tau \\& + (2\pi )^{-\frac{1}{2}}\int _{|\tau|< \varepsilon _4} e^{-\im \tau t}     {R} _{333}^{(1)}(\im  \tau ) Q \zeta _B ' \widehat{\widetilde{V}}(\tau )    d\tau =: T _{11}(t) \zeta_{B}'\widetilde{V}  + T _{12}(t) \zeta_{B}'\widetilde{V} + T _{13}(t) \zeta_{B}'\widetilde{V}.
  \end{align*}
\normalsize
Like in  \cite{CM25D1},  like for $T_2$ we obtain
   \begin{align*}&
     \left \| T _{11}  \zeta_{B}'\widetilde{V} \right \| _{L^2 \( I ,    \widetilde{\Sigma}\) }  \lesssim B^{-1/2}\(  \|V\| _{L^2(I,  { \Sigma }_{1A A_1} )} +
\| V \| _{L^2(I,  { \Sigma }_{2A A_1} )} \) .
\end{align*}
We write here only the proof for $ T _{12}$ of the following inequality
\begin{equation}\label{eq:hard}
    \left \| T _{12}  \zeta_{B}'\widetilde{V} \right \| _{L^2 \( I ,    \widetilde{\Sigma}\) } +   \left \| T _{13}  \zeta_{B}'\widetilde{V} \right \| _{L^2 \( I ,    \widetilde{\Sigma}\) } \lesssim B ^{-1} A_1^{3/2}   \| V \| _{L^2(I,  { \Sigma }_{1A A_1} )} +B  ^{1/2}
\| V \| _{L^2(I,  { \Sigma }_{2A A_1} )},  
\end{equation}
    since for $ T _{13}$ the proof is similar and is similar to the analogous proof in \cite{CM25D1}.
We write 
\begin{align*}
   {R} _{233}^{(1)}(\im  \tau )Q\zeta_{B}'\widehat{\widetilde{V}} =   {R} _{233}^{(1)}(\im  \tau ) \zeta_{B}'\widehat{\widetilde{V}} -  {R} _{233}^{(1)}(\im  \tau ) P\zeta_{B}'\widehat{\widetilde{V}}.
\end{align*}
Omitting irrelevant factors,  schematically   we  have
\begin{align*}
    &  {R} _{233}^{(1)}(\im  \tau )(\im  \tau ) \zeta_{B}'\widehat{\widetilde{V}}(\tau) (x)  = \frac{ \Xi _1[x,c_0] }{\im \tau }: \int_  {\R }   \( e^{-\mu _2(\im  \tau) y} -1  \)   \widetilde{\phi} (y)\zeta_{B}' (y) \widehat{\widetilde{V}}(\tau , y) dy \\& =  \frac{ \Xi _1[x,c_0] }{\im \tau }:  \int_  {\R_- }   \( e^{-\mu _2(\im  \tau) y} -1  \)   \widetilde{\phi} (y)\frac{\zeta_{B}' (y)}{\zeta_{A}  (y)}\zeta_{A}  (y) \widehat{\widetilde{V}}(\tau , y) dy \\& + \frac{ \Xi _1[x,c_0] }{\im \tau }:  \int_  {\R _+}   \( e^{-\mu _2(\im  \tau) y} -1  \)   \widetilde{\phi} (y)\frac{\zeta_{B}' (y)}{\zeta_{A}  (y)}\zeta_{A}  (y) \vartheta _{1A_1}(y) \widehat{\widetilde{V}}(\tau , y) dy\\& + \frac{ \Xi _1[x,c_0] }{\im \tau }:  \int_  {\R _+}   \( e^{-\mu _2(\im  \tau) y} -1  \)   \widetilde{\phi} (y)\frac{\zeta_{B}' (y)}{\zeta_{A}  (y)}\zeta_{A}  (y) \vartheta _{2A_1}(y) \widehat{\widetilde{V}}(\tau , y) dy =:\sum _{j=1}^{3}\mathbf{A}_j (\tau , x).
\end{align*}
Verbatim from    \cite{CM25D1}
we have  \small
\begin{align*}&
  \|  \mathbf{A}_1 (\tau , x)\| _{L^2( |\tau|< \varepsilon _4, \widetilde{\Sigma}) }  \\& \lesssim  \left \| \frac{|\mu _2(\im  \tau)|}{|\tau|} \right \| _{L^\infty(|\tau|< \varepsilon _4  )}    \| \Xi _1[x,c_0]\| _{ \widetilde{\Sigma}} \|  \<y \> \widetilde{\phi}\| _{L^2(\R _-)}   \left  \| \frac{\zeta_{B}'  }{\zeta_{A}  }\right \| _{L^\infty(\R  )}
 \(  \| V \| _{L^2(I,  { \Sigma }_{1A A_1} )} +
\| V \| _{L^2(I,  { \Sigma }_{2A A_1} )} \) \\& \lesssim B ^{-1}  \(  \| V \| _{L^2(I,  { \Sigma }_{1A A_1} )} +
\|V \| _{L^2(I,  { \Sigma }_{2A A_1} )} \),
\end{align*}
\normalsize 
\begin{align*}&
  \|  \mathbf{A}_2 (\tau , x)\| _{L^2( |\tau|< \varepsilon _4, \widetilde{\Sigma}) }  \\& \lesssim  \left \| \frac{|\mu _2(\im  \tau)|}{|\tau|} \right \| _{L^\infty(|\tau|< \varepsilon _4  )}    \| \Xi _1[x,c_0]\| _{ \widetilde{\Sigma}} \|  \widetilde{\phi}\| _{L^\infty(\R )}   \left  \| \< y\> \frac{\zeta_{B}'  }{\zeta_{A}  }\right \| _{L^2\( 0, 2A_1  \)}
    \| V \| _{L^2(I,  { \Sigma }_{1A A_1} )}   \\&  \lesssim B ^{-1} A_1^{3/2}   \| V \| _{L^2(I,  { \Sigma }_{1A A_1} )}  \quad  \text{ and}
\end{align*}
\begin{align*}
  \|  \mathbf{A}_3 (\tau , x)\| _{L^2( |\tau|< \varepsilon _4, \widetilde{\Sigma}) } &\lesssim  \left \| \frac{|\mu _2(\im  \tau)|}{|\tau|} \right \| _{L^\infty(|\tau|< \varepsilon _4  )}    \| \Xi _1[x,c_0] \| _{ \widetilde{\Sigma}} \|  \widetilde{\phi}\| _{L^\infty(\R )} \\& \times \left  \| \< y\> \frac{\zeta_{B}'  }{\zeta_{A}  }\right \| _{L^2\(\R  \)}
    \| V \| _{L^2(I,  { \Sigma }_{2A A_1} )}    \lesssim B  ^{1/2}
\| V \| _{L^2(I,  { \Sigma }_{2A A_1} )}  .
\end{align*}
Finally, we have the following, also proved like   in \cite{CM25D1} for which we refer for the proof, 
\begin{align*}
   &\| {R} _{233}^{(1)}(\im  \tau ) P\zeta_{B}'\widehat{\widetilde{V}}\| _{L^2( |\tau|< \varepsilon _4, \widetilde{\Sigma}) }  \lesssim   B^{-1/2}  \(  \|v \| _{L^2(I,  { \Sigma }_{1A A_1} )} +
\|v \| _{L^2(I,  { \Sigma }_{2A A_1} )} \) .
\end{align*}
   \qed

\appendix
\section{Appendix: Proof of {Lemma}  \ref{lem:solwave4}}\label{app:A}

We follow the  sketch of proof in Remark 1 \cite{BK22ARMA}. 
First, from \cite{BK19JDE} we consider the following functions:
\begin{align*}
  & H(n,c):= \frac{c^2}{2}\(   1-\frac{1}{(1+n)^2}  \)- K \log (1+n) \ , \\& h(n,c):=   \partial _n H(n,c)=  \frac{c^2} {(1+n)^3}   - \frac{K }{1+n}  \ ,
  \\& g(n,c):=     \frac{c^2} { 1+n }   + {K } ({1+n}) + e^{H(n,c)}  \ , \\& E_c (x)  : =-\phi _c ' (x)    \ .
\end{align*}
Let $n^*_c:=n_c(0)$. Then there is the claim that $c\to n^*_c $ is smooth. This can be seen observing that
\begin{align*}
   g(n^*_c ,c)=  g(0 ,c)=c^2+K+1,
\end{align*}
see \cite[formula (2.12)]{BK19JDE},
and
\begin{align*}
  \partial _ng(n^*_c ,c)=- h(n^*_c,c)\( 1+ n^*_c- e^{H(n^*_c,c)}\) .
\end{align*}
For $n\ge 0$ we have
\begin{align*}
  1+ n - e^{H(n ,c)}=0 \text{ only for } n=0, n_{ce} (c)
\end{align*}
for a $0<n_{ce} (c)<n^*_c$, see \cite[p. 3462]{BK19JDE}.  Furthermore by \cite[formula (2.13)]{BK19JDE} we have $h(n,c)<0$ for $ n \in \(0, \frac{c}{\sqrt{K}}\)$. Since    \cite   {BK19JDE}  shows that
$n^*_c\in \(0, \frac{c}{\sqrt{K}}\)$, we can conclude that $ \partial _ng(n^*_c ,c)\neq  0$ and so the smoothness of  $c\to n^*_c $ follows from  the implicit function theorem. From formula (2,2)  in \cite   {BK19JDE},
\begin{align*}
   \left\{
     \begin{array}{ll}
       u_c(0) =  c\( 1 -\frac{1}{n^*_c}  \)   , &   \\
        \phi_c(0) = H(n^*_c ,c) , &
     \end{array}
   \right.
\end{align*}
it follows that $c\to S_c(0)$ is smooth. Then, since $ S_c (x)  = (n_c(x), u_c(x), \phi _c(x))^\intercal $ satisfies the ODE
\begin{align*}
    \left\{
        \begin{aligned}
            &-c {n}_c' =-  ( (1+n_c)  u_c) ' \\
            &-c {u}_c' =-\( \frac{u^2_c}{2} - K \log (1+n_c) -\phi _c\) '   , \\& -  \phi ''_c +e^{\phi _c} =1+n _c .
        \end{aligned}
    \right.
\end{align*}
we have that $(x,c)\to S_c(x)$ is smooth.  Up to this point we have followed Remark 1 \cite{BK22ARMA}.

 We set  $S^{\varepsilon}_R(x)=:(n^{\varepsilon}_R(\xi), u^{\varepsilon}_R(\xi), \phi ^{\varepsilon}_R (\xi))^\intercal$  where $\xi = \sqrt{\varepsilon}x$. Then as functions in the variable $\xi$,
in \cite[Sec. 4.2]{BK19JDE} the following formula is derived:
\begin{align} \label{eq:4.37}
  &  \left(
       \begin{array}{ccc}
         -\mathsf{V} & 1 & 0 \\
         K & -\(1+K\) ^{1/2} & 1 \\
         -1 & 0 &  1 \\
       \end{array}
     \right)   \left(
                 \begin{array}{c}
                   n^{\varepsilon}_R \\
                   u^{\varepsilon}_R \\
                   \phi^{\varepsilon}_R \\
                 \end{array}
               \right) =    \left(
                 \begin{array}{c}
                   \mathcal{M}_{1}^{\varepsilon} \\
                   \mathcal{M}_{2}^{\varepsilon} \\
                    \mathsf{V}\mathcal{M}_{1}^{\varepsilon} + \mathcal{M}_{2}^{\varepsilon} \\
                 \end{array}
               \right)
\end{align}
where
\begin{align*}
  \mathcal{M}_{1}^{\varepsilon}& = \varepsilon (1-\mathsf{V} \psi _{KdV}) n^{\varepsilon}_R- (\varepsilon \psi _{KdV}+n^{\varepsilon}_R)u^{\varepsilon}_R + \varepsilon^2(\psi _{KdV} - \mathsf{V}\psi _{KdV}^2 ) \text{  and}
\\
  \mathcal{M}_{2}^{\varepsilon} &= \varepsilon (1-\mathsf{V} \psi _{KdV}) u^{\varepsilon}_R - \frac{|u^{\varepsilon}_R|^2}{2} \\& + \frac{K}{2}(2\varepsilon \psi _{KdV} n^{\varepsilon}_R+|n^{\varepsilon}_R|^2 )+ \varepsilon^2\(  \mathsf{V}\psi _{KdV} -   \frac{\psi _{KdV}^2}{2}   \) \\& - K \left   [     \log (1+ \varepsilon \psi _{KdV} + n^{\varepsilon}_R  )- (\varepsilon \psi _{KdV} + n^{\varepsilon}_R) +\frac{(\varepsilon \psi _{KdV} + n^{\varepsilon}_R)^2}{2}   \right ] .
\end{align*}
Then
\begin{align*}
  \partial _c S^{\varepsilon}_R(x) = ( \partial _\varepsilon n^{\varepsilon}_R(\xi),\partial _\varepsilon u^{\varepsilon}_R(\xi), \partial _\varepsilon \phi ^{\varepsilon}_R (\xi))^\intercal    + 2^{-1}x\varepsilon ^{-1/2}( \partial _\xi n^{\varepsilon}_R,\partial _\xi u^{\varepsilon}_R(\xi), \partial _\xi \phi ^{\varepsilon}_R (\xi))^\intercal
\end{align*}
where
\begin{align*}
  |( \partial _\varepsilon n^{\varepsilon}_R(\xi),\partial _\varepsilon u^{\varepsilon}_R(\xi), \partial _\varepsilon \phi ^{\varepsilon}_R (\xi))^\intercal| \lesssim  \varepsilon e^{-\alpha |\xi|}
\end{align*}
is obtained differentiating  in $\varepsilon$  formula \eqref{eq:4.37} and applying \eqref{eq:solwave2} which directly implies also
\begin{align*}
  |x\varepsilon ^{-1/2}( \partial _\xi n^{\varepsilon}_R,\partial _\xi u^{\varepsilon}_R(\xi), \partial _\xi \phi ^{\varepsilon}_R (\xi))^\intercal| \lesssim  \varepsilon e^{-\frac{\alpha}{2} |\xi|}.
\end{align*}
Further differentiating in $\xi$   formula \eqref{eq:4.37}   completes the proof of {Lemma}  \ref{lem:solwave4}.  \qed

\section{Proof of Lemma~\ref{lem:mod1}}    \label{app:B}
\begin{proof}[Proof of Lemma~\ref{lem:mod1}]
    Set
    \begin{align*}
        F_B(\tilde D,c,\widetilde{U}):=\begin{pmatrix}
            B^{-1/2} \<e^{B^{1/2}\tilde{D}\partial_x}\widetilde{U}-\widetilde{S}_c,\zeta_B\eta_1[c]\>\\
            -\<e^{B^{1/2}\tilde{D}\partial_x}\widetilde{U}-\widetilde{S}_c,\eta_2[c]\>
        \end{pmatrix}.
    \end{align*}
    Then, we have
    \begin{align}
        F_B(0,c_0,\widetilde{S}_{c_0})&=0,\nonumber\\
        \|F_B(0,c_0,\widetilde{U})\|&\lesssim \|\widetilde{U}-\widetilde{S}_{c_0}\|_{L^2},\label{pr:ift1}
    \end{align}
    and
    \begin{align} \label{pr:ift2}
    F_B(\tilde{D},c,\widetilde{U}) = 
    F_B(0,c,e^{B^{1/2}\tilde{D}\partial_x}\widetilde{U}).
    \end{align}
    Further, we have
    \begin{align*}
        &D_{\tilde{D},c} F_B(\tilde{D},c,\widetilde{U})  
        \\&=
        \begin{pmatrix}
            -\<e^{B^{1/2}\tilde{D}\partial_x}\widetilde{U},(\zeta_B \eta_1[c])'\> & -B^{-1/2}\<\partial_c \widetilde{S}_c , \zeta_B\eta_1[c]\>+B^{-1/2}\<e^{B^{1/2}\tilde{D}\partial_x }\widetilde{U}-\widetilde{S}_c,\zeta_B \partial_c \eta_1[c]\>\\
            B^{1/2}\<e^{B^{1/2}\tilde{D}\partial_x}\widetilde{U}, \eta_2[c]'\>& \<\partial_c\widetilde{S}_c,\eta_2[c]\>-\<e^{B^{1/2}\tilde{D}\partial_x }\widetilde{U}-\widetilde{S}_c,\partial_c \eta_2[c]\>
        \end{pmatrix}
    \end{align*}
    and in particular, from \eqref{eq:xici}, \eqref{orth:xieta2} and \eqref{orth:xieta}, we have
    \begin{align*}
        D_{\tilde{D},c} F_B(0,c_0,\widetilde{S}_{c_0}) 
        &=
        \begin{pmatrix}
            \<\widetilde{S}_{c_0}',\zeta_B \eta_1[c]\> & -B^{-1/2}\<\partial_c\widetilde{S}_{c_0} , \zeta_B\eta_1[c]\>\\
            -B^{1/2}\<\widetilde{S}_{c_0}', \eta_2[c]\>& \<\partial_c\widetilde{S}_{c_0},\eta_2[c]\>
        \end{pmatrix}\\&
        =\begin{pmatrix}
            1 & 0\\0 & 1
        \end{pmatrix}
        +\begin{pmatrix}
            \<\xi_1[c_0],(\zeta_B-1)\eta_1[c_0]\> & B^{-1/2}\<\xi_2[c_0],(1-\zeta_B)\eta_1[c_0]\>\\
            0 & 0
        \end{pmatrix}
    \end{align*}
    Since the second term converges to $0$ as $B\to \infty$, we have $\| D_{\tilde{D},c}F_B(0,c_0,\widetilde{S}_{c_0})^{-1}\|\leq 2$  for sufficiently large $B$.
    Further, by
    \begin{align*}
         &D_{\tilde{D},c} F_B (\tilde{D},c,\widetilde{U}) - D_{\tilde{D},c} F_B(0,c_0,\widetilde{S}_{c_0}) 
         =
        D_{\tilde{D},c} F_B (\tilde{D},c,\widetilde{U}) -
        D_{\tilde{D},c}F_B(\tilde{D},c,e^{-B^{1/2}\tilde{D}\partial_x}\widetilde{S}_c)
        \\&\quad + D_{\tilde{D},c}F_B(\tilde{D},c,e^{-B^{1/2}\tilde{D}\partial_x}\widetilde{S}_c)- D_{\tilde{D},c} F_B(0,c_0,\widetilde{S}_{c_0})  \\&=\begin{pmatrix}
            -\<e^{B^{1/2}\tilde{D}\partial_x}\widetilde{U}-\widetilde{S}_c,(\zeta_B\eta_1[c])'\> & B^{-1/2}\<e^{B^{1/2}\tilde{D}\partial_x }\widetilde{U}-\widetilde{S}_c,\zeta_B \partial_c \eta_1[c]\>\\
            B^{1/2}\<e^{B^{1/2}\tilde{D}\partial_x}\widetilde{U}-\widetilde{S}_c,\eta_2[c]'\> & -\<e^{B^{1/2}\tilde{D}\partial_x }\widetilde{U}-\widetilde{S}_c,\partial_c \eta_2[c]\>
        \end{pmatrix}
        \\&\quad
        +\begin{pmatrix}
            \<\xi_1[c]-\xi_1[c_0],(\zeta_B-1)\eta_1[c_0]\> & B^{-1/2}\<\xi_2[c]-\xi_2[c_0],(1-\zeta_B)\eta_1[c_0]\>\\
            0 & 0
        \end{pmatrix}\\&\quad
        +\begin{pmatrix}
            \<\xi_1[c_0],(\zeta_B-1)(\eta_1[c]-\eta_1[c_0])\> & B^{-1/2}\<\xi_2[c_0],(1-\zeta_B)(\eta_1[c]-\eta_1[c_0])\>\\
            0 & 0
        \end{pmatrix},
    \end{align*}
    we have
    \begin{align*}
        \|D_{\tilde{D},c} F_B (\tilde{D},c,\widetilde{U}) - D_{\tilde{D},c} F_B(0,c_0,\widetilde{S}_{c_0}) \|\lesssim B^{1/2}\|\widetilde{U}-e^{-B^{1/2}\tilde{D}\partial_x}\widetilde{S}_c\|_{L^2} +|c-c_0|.
    \end{align*}
    Thus, for $(\tilde{D},c,\widetilde{U})\in D_{\R\times \R\times L^2(\R,\R^2)}((0,c_0,\widetilde{S}_{c_0}),\delta_1)$ with $\delta_1\ll B^{-1}$, we have 
    $$\|D_{\tilde{D},c} F_B (\tilde{D},c,\widetilde{U}) - D_{\tilde{D},c} F_B(0,c_0,\widetilde{S}_{c_0}) \|\leq 1/4.$$
    With such estimate combined with \eqref{pr:ift1}, we can apply implicit function theorem which shows the existence of $(\widetilde{D},c)\in C^1(D_{L^2(\R,\R^2)})(\widetilde{S}_{c_0},\delta_0)$ for $\delta_0\ll \delta_1$ satisfying $F_B(\tilde{D}(\widetilde{U}),c(\widetilde{U}),\widetilde{U})=0$ and $(\tilde{D}(\widetilde{U}),c(\widetilde{U}))\subset D_{\R^2}((0,c_0),\delta_1)$.
    From the symmetry \eqref{pr:ift2} of $F_B$, we can extend $(c,\widetilde{D})$ on $\mathcal{U}(c_0,\delta_0)$.
    Finally, setting $D=B^{-1/2}\tilde{D}$, we have the conclusion.
\end{proof}

		\bibliography{references}
		\bibliographystyle{plain}

Department of Mathematics, Kyungpook National University, 80 Daehakro, Bukgu, Daegu 41566 Korea.
{\it E-mail Address}: {\tt junsikbae@knu.ac.kr}

Department of Mathematics, Informatics and Geosciences,  University
of Trieste, via Valerio  12/1  Trieste, 34127  Italy.
{\it E-mail Address}: {\tt scuccagna@units.it}

Department of Mathematics and Informatics,
Graduate School of Science,
Chiba University,
Chiba 263-8522, Japan.
{\it E-mail Address}: {\tt maeda@math.s.chiba-u.ac.jp}

\end{document}